 \documentclass[final,number,3p,times]{elsarticle}
\usepackage{amssymb}

\usepackage{amssymb}
\usepackage{amsmath}
 \usepackage{paralist}
\usepackage{amsthm}
\usepackage{subfigure}
\usepackage{booktabs}
\usepackage{mathrsfs}

 \biboptions{square,sort&compress}

\begin{document}

\begin{frontmatter}



\newtheorem{theorem}{Theorem}[section]
\newtheorem{corollary}{Corollary}
\newtheorem*{main}{Main Theorem}
\newtheorem{lemma}[theorem]{Lemma}
\newtheorem{proposition}{Proposition}
\newtheorem{conjecture}{Conjecture}
\newtheorem*{problem}{Problem}
\theoremstyle{definition}
\newtheorem{definition}{Definition}
\newtheorem{remark}{Remark}
\newtheorem*{notation}{Notation}
\newcommand{\ep}{\varepsilon}
\newcommand{\eps}[1]{{#1}_{\varepsilon}}
\renewcommand{\theequation}{\thesection.\arabic{equation}}
\numberwithin{equation}{section}

\title{Uniqueness and stability of traveling waves to the time-like extremal hypersurface in Minkowski space\tnoteref{label1}}
 \tnotetext[label1]{This work was supported by Key Laboratory of Mathematics for Nonlinear Sciences (Fudan University), Ministry of Education of China, P.R.China.
Shanghai Key Laboratory for Contemporary Applied Mathematics, School of Mathematical Sciences, Fudan University, P.R. China, NSFC (grants No. 11401367, grants No. 11421061, grants No.11726611, grants No. 11726612), 973 program (grant No. 2013CB834100) and 111 project.}

\author[math]{Jianli Liu}
\ead{jlliu@shu.edu.cn}
\address[math]{Department of Mathematics, Shanghai University, Shanghai 200444, China}

\author[rvt]{Yi Zhou\corref{cor1}}
\ead{yizhou@fudan.edu.cn}\cortext[cor1]{Corresponding author}
\address[rvt]{School of Mathematical Sciences, Fudan University, Shanghai 200433, China}

%

\begin{abstract}
There is a few results about the global stability of nontrivial solutions to quasilinear wave equations. In this paper we are concerned with the uniqueness and stability of traveling waves to the time-like extremal hypersurface in Minkowski space. Firstly, we can get the existence and uniqueness of traveling wave solutions to the time-like extremal hypersurface in $\mathbb{R}^{1+(n+1)}$, which can be considered as the generalized Bernstein theorem in Minkowski space. Furthermore, we also get the stability of traveling wave solutions with speed of light to time-like extremal hypersurface in $1+(2+1)$ dimensional Minkowski space.
\end{abstract}


\end{frontmatter}


\section{Introduction and main results}
The extremal surface in Minkowski space is the $C^2$ surface with vanishing mean curvature. The time-like extremal surface is an interesting model which may be viewed as simple but nontrivial examples of membrane in field theory. The equation to time-like extremal hypersurface in $1+(n+1)$ dimensional Minkowski space are as follows
\begin{equation}\label{1.1}
\left(\frac{v_{t}}{\sqrt{1+|\nabla v|^{2}- v_{t}^{2}}}\right)_{t}-\nabla\cdot\left(\frac{\nabla v}{\sqrt{1+|\nabla v|^{2}-v_{t}^{2}}}\right)=0.
\end{equation}
where $\Delta\doteq 1+|\nabla v|^{2}- v_{t}^{2} > 0, v(t,x)$ is the scalar function, $t$ is the time variable and $x = (x_1,\cdots, x_n)$ is the space variable.

In this paper, we will give the uniqueness and stability of the traveling wave solution to the time-like extremal hypersurface in Minkowski space. There are two main parts. Firstly, we will give the existence and uniqueness of traveling wave solution to time-like extremal surface in Minkowski space, which is correspondent to the famous Bernstein theorem of minimal surface in $\mathbb{R}^n$. The classical Bernstein Theorem is solved by Bernstein in three dimensional Riemanian manifold \cite{Bernstein}. It was proved in dimensions up to 8 by \cite{Moser}, \cite{Fleming},\cite{De Giorgi}, \cite{Almgren}, \cite{Simons}, \cite{Bombieri}. For the space-like maximal surface in a $n$-dimensional Lorentzian manifold, there is the similar Calabi-Bernstein theorem, which was first proved by Calabi in \cite{Calabi}, and extended to the general $n$-dimensional case by Cheng and Yau \cite{Cheng and Yau}. We can also refer to  \cite{Kobayashi}, \cite{Ncnertey}, \cite{Estudillo and Romero 1, Estudillo and Romero 2, Estudillo and Romero 3}, \cite{Romero}, \cite{Alias and Palmer}.
Now we will consider the Bernstein type theorem of the system (\ref{1.1}) and find out the representation of traveling wave solution.  We assume that there exists
a traveling wave solution of the form $f(x - \vec{c}t)$, where $f$ is scalar funtion.
Without loss of generality, let the generalized velocity is $\vec{c}=(0,\cdots,0,c)$,
 $c\neq0$. Then, $\Delta = 1 + \sum\limits_{i=1}^{n-1}|\partial_{i}f|^{2}+(1-c^{2})|\partial_{n}f|^{2}$. Therefore, system (\ref{1.1}) can be rewritten as
\begin{equation}\label{1.2}
c^{2}\partial_{n}\left(\frac{\partial_{n}f}{\sqrt{\Delta}}\right)-\partial_{1}\left(\frac{\partial_{1}f}{\sqrt{\Delta}}\right)-\cdots-\partial_{n-1}\left(\frac{\partial_{n-1}f}{\sqrt{\Delta}}\right)-\partial_{n}\left(\frac{\partial_{n}f}{\sqrt{\Delta}}\right)=0.
\end{equation}
Thus
\begin{equation}\label{1.3}
\partial_{1}\left(\frac{\partial_{1}f}{\sqrt{\Delta}}\right)+\cdots+\partial_{n-1}\left(\frac{\partial_{n-1}f}{\sqrt{\Delta}}\right)+(1-c^{2})\partial_{n}(\frac{\partial_{n}f}{\sqrt{\Delta}})=0.
\end{equation}

When $|c| < 1$, let $x'_n = \frac{1}{\sqrt{1-c^2}}(x_n - ct)$,  the above system (\ref{1.3}) can be rewritten as
\begin{equation}\label{1.4}
\partial_{1}\left(\frac{\partial_{1}f}{\sqrt{\Delta}}\right)+\cdots+\partial_{n}\left(\frac{\partial_{n-1}f}{\sqrt{\Delta}}\right)+ \tilde{\partial}_{n}(\frac{\tilde{\partial}_{n}f}{\sqrt{\Delta}})=0.
\end{equation}
where $\tilde{\partial}_{n} = \partial_{x'_n}$. Then, system (\ref{1.4}) can be considered as the equation to minimal surface in $\mathbb{R}^n$. By Bernstein theorem of minimal surface in Euclidean space, $f(x-\overrightarrow{c}t)$ is the linear function of $x_{1},\cdots, (x_{n}-ct)$ for $n\leq 8$. Then, we can get the affine solutions of time like extremal surface
\begin{eqnarray*}
f = a_{1}x_{1} + \cdots + a_{n}x'_n + b =  a_{1}x_{1} + \cdots +\frac{1}{\sqrt{1-c^2}}(x_n-ct) + b.
  \end{eqnarray*}
For the stability of this kind of flat plane solution for time-like extremal surface in Minkowski space, Allen et al. gave the positive answer about its stability in \cite{Allen}.

When $c = 1$, we can get $\Delta = 1 + \sum\limits_{i=1}^{n-1}|\partial_{i}f|^{2}$. Therefore, the system (\ref{1.3}) can be rewritten as
\begin{equation}\label{1.5}
\partial_{1}(\frac{\partial_{1}f}{\sqrt{\Delta}})+\cdots+\partial_{n-1}(\frac{\partial_{n-1}f}{\sqrt{\Delta}})=0.
\end{equation}
Then, the equation (\ref{1.5}) can be considered as the minimal surface equation in $\mathbb{R}^{n-1}$, which is independent of the $n$-th variable. Using Bernstein theorem in Euclidean space, we can get
\begin{equation}\label{1.6}
f(x_{1},\cdots,x_{n}\pm t)=(a_{1}x_{1}+a_{2}x_{2}+\cdots+a_{n-1}x_{n-1}+b)F(x_{n}\pm t).
\end{equation}
We can easily check the following form
\begin{equation}\label{1.7}
f(x_{1},\cdots,x_{n}\pm t)=a_{1}x_{1}F_{1}(x_{n}\pm t)+a_{2}x_{2}F_{2}(x_{n} \pm t)+\cdots+a_{n-1}x_{n-1}F_{n-1}(x_n \pm t)+bF_{n}(x_{n}\pm t)
\end{equation}
is also the exact solution of time-like extremal surface equation, where $F_{i}\ (i = 1, \cdots, n)$ are $C^{2}$ functions.

\begin{remark}
In Minkowski space $\mathbb{R}^{1+(1+n)}$, the authors \cite{Kong Sun and Zhou} gave the coefficient and necessary condition of the global classical solution to time-like extremal surface in one dimensional space. Liu and Zhou gave the asymptotic behavior to global classical solutions, which tends to the combinations of traveling wave solutions \cite{Liu and Zhou} and got the exact solutions of the traveling wave solutions with the form $\phi(x\pm t)$. The authors also got the stability of traveling wave solution to Cauchy problem to the equation of timelike extremal surface in Minksowski space $\mathbb{R}^{1+(1+n)}$ \cite{Liu and Liu}. The global existence of the initial boundary value problem of timelike extremal surface equation was studied in \cite{Liu and Zhou 1} and \cite{Liu and Zhou 2}.
\end{remark}

\begin{remark}
In this case, we get the exact solutions with the form as (\ref{1.6}) or (\ref{1.7}) for  $n \leq 9$ for time-like extremal hypersurface in Minkowski space. It is different to the Bernstein theorem of minimal surface. In the second part in this paper, we will also consider the global stability of traveling wave solution having the special form with the speed of light.
\end{remark}

When $c > 1$, $\Delta = 1 + \sum\limits_{i=1}^{n-1}|\partial_{i}f|^{2} - (c^{2}-1)|\partial_{n}f|^{2}$. Therefore,
\begin{equation}\nonumber
\partial_{n}\left(\frac{\partial_{n}f}{\sqrt{\Delta}}\right)-\frac{1}{c^{2}-1}\partial_{1}\left(\frac{\partial_{1}f}{\sqrt{\Delta}}\right)-\frac{1}{c^{2}-1}\partial_{2}\left(\frac{\partial_{2}f}{\sqrt{\Delta}}\right)
-\cdots-\frac{1}{(c^{2}-1)}\partial_{n-1}\left(\frac{\partial_{n-1}f}{\sqrt{\Delta}}\right)=0.
\end{equation}
Using the variable transformation $x'_n =  \frac{1}{\sqrt{c^2 -1}}(x_n - ct)$, we have
\begin{equation}\label{1.8}
\tilde{\partial}_{n}\left(\frac{\tilde{\partial}_{n}f}{\sqrt{\Delta'}}\right)-\partial_{1}\left(\frac{\partial_{1}f}{\sqrt{\Delta'}}\right)- \partial_{2}\left(\frac{\partial_{2}f}{\sqrt{\Delta'}}\right)
-\cdots- \partial_{n-1}\left(\frac{\partial_{n-1}f}{\sqrt{\Delta'}}\right)=0.
\end{equation}
where $\tilde{\partial}_{n} = \partial_{x'_n}, \Delta' = 1+\sum\limits_{i=1}^{n-1}|\partial_{x_i}f|^{2} - |\partial_{x'_n}f|^{2}$. Then, we find that the system (\ref{1.8}) is the equation of time-like extremal hypersurface in Minkowski space $\mathbb{R}^{1, n-1}$.

\begin{remark}
The above results can be considered as the generalized Bernstein theorem of the time-like extremal surface in Minkowski space.
\end{remark}

The equation (\ref{1.1}) can be considered as the $n$-dimensional quasilinear wave equation. Most of the global results to nonlinear wave equations are concerned with Cauchy problem with small initial data, especially in high space dimensional case. Recently, one kind of large solution called "short pulse solution" are considered in \cite{Christodoulou}, \cite{Klainerman and Rodnianski}. For semilinear wave equations satisfying the null condition, global solution with large inital data is considered in \cite{Wang and Yu 1}, \cite{Wang and Yu 2}, \cite{Miao Pei and Yu}, \cite{Yang}. Wang and Wei gave the global existence of short pulse solution to relativistic membrane equations \cite{Wang and Wei}. For the stability of time-like extremal surface in Minkowski space, Brendle obtained the stability of a flat hyperplane for $n\geq 3$ in \cite{Brendle}. Krieger and Lindblad \cite{Krieger and Lindblad} studied the radial perturbations of the static catenoid solution to hyperbolic
vanishing mean curvature flow which are supported far away from the `collar' of the catenoid. Donninger, Krieger, Szeftel and Wong \cite{Donninger Krieger Szeftel and Wong} showed that the linear instability of catenoid is the only obstruction to the global nonlinear stability. In the following, we will consider the stability of traveling wave solutions with the  speed of light for the time-like extremal hypersurface in $1+(2+1)$ dimensional Minkowski space. The equation to time-like extremal hypersurface in $\mathbb{R}^{1+(2+1)}$ is as follows
\begin{equation}\label{1.9}
(\frac{v_t}{\sqrt{1 - Q_0(v,v)}})_t - \sum\limits_{i = 1}^2 (\frac{v_{x_i}}{\sqrt{1 - Q_0(v,v)}})_{x_i} = 0
\end{equation}
where $Q_0(\phi, \psi) \doteq \phi_t\psi_t - \phi_{x_1}\psi_{x_1} - \phi_{x_2}\psi_{x_2}$ is the null form (see \cite{Klainerman}, \cite{Alinhac}). Barbashov, Nesterenko and Chervyakov studied the nonlinear differential equations and obtained explicitly their general solutions to relativistic string in one dimensional case \cite{Barbashov}. Milnor described all entire time-like minimal surfaces in the three-Minkowski space via a kind of Weierstrass representation \cite{Milnor}. Recently, for the vanishing mean curvature equation, the existence of global smooth solutions for small initial data has been addressed successfully by Lindblad \cite{Lindblad}.  Allen, Andersson and Isenberg \cite{Allen} proved the small data global existence for timelike extremal submanifold with codimension larger than one.
Here, we first consider the stability of a class of traveling wave solution with the velocity 1. We denote $v(t, x)$ as a small perturbation of the traveling wave solution with the speed of light. By rotational symmetry, we assume that the traveling wave is of the form $(ax_2+b)F(x_1 + t)$. Let
\begin{equation} \label{1.10}
v(t, x) = (ax_2 + b)F(x_1 +t) + u(t, x)
\end{equation}
where $x = (x_1, x_2)$. Then, we can get
\begin{equation}\label{1.11}
(\frac{u_t + (ax_2 + b)F_t}{\sqrt{1 - Q_0(u, u) - 2[ (ax_2 + b) F'(u_t - u_{x_1}) - au_{x_2}F]}})_t - \sum\limits_{i = 1}^2 (\frac{u_{x_i} + (ax_2 + b)F_{x_i}}{\sqrt{1 - Q_0(u, u) - 2[ (ax_2 + b) F'(u_t - u_{x_1}) - au_{x_2}F]}})_{x_i} = 0.
\end{equation}
We recast the system (\ref{1.11}) as follows
\begin{eqnarray} \nonumber
\Box u &=& - \frac{1}{2}\frac{Q_0(u + (ax_2 + b)F, Q_0(u, u) + 2[(ax_2 + b) F'(u_t - u_{x_1}) - au_{x_2}F])}{1 - Q_0(u,u)- 2[(ax_2 + b) F'(u_t - u_{x_1}) - au_{x_2}F]}\\ \label{1.12}
&=& \frac{1}{2}(1-\tilde{H})Q_0(u + (ax_2 + b)F, Q_0(u, u) + 2[(ax_2 + b) F'(u_t - u_{x_1}) - au_{x_2}F])
\end{eqnarray}
where $\Box = \partial_{tt} - \partial_{x_1x_1} -\partial_{x_2x_2}$ and $\tilde{H} = \frac{1}{1 - Q_0(u,u)- 2[(ax_2 + b) F'(u_t - u_{x_1}) - au_{x_2}F]} + 1$.

Then, because of the traveling wave solutions, there is one more linear term in above system than the original system. We can also rewrite the system as
 \begin{eqnarray} \nonumber
&&\Box u - Q_0((ax_2 + b)F, 2[(ax_2 + b) F'(u_t - u_{x_1}) - au_{x_2}F]) \\ \nonumber
&&= \frac{1}{2}(1-\tilde{H})\{Q_0(u, Q_0(u, u) + 2[(ax_2 + b) F'(u_t - u_{x_1}) - au_{x_2}F]) + Q_0((ax_2 + b)F, Q_0(u, u))\} \\ \nonumber
&&- \frac{1}{2}\tilde{H}Q_0((ax_2 + b)F, 2[(ax_2 + b) F'(u_t - u_{x_1}) - au_{x_2}F])
\end{eqnarray}
It is easily get the above system is also hyperbolic.

Under the assumptions
\begin{equation*}
\tilde{H}_1)\ \ \ \ \ \  | (\xi \frac{d}{d\xi})^{k_2} (\frac{d}{d\xi})^{k_1} F(\xi)| \leq C_{k_1, k_2}(2+\xi )^{-1}
\end{equation*}
where $k_1\geq 0, k_2 \geq 0, k_1 +k_2 \leq s\ \ (s\geq 13)$ and $\xi = t + x_1$, we shall consider the following Cauchy problem
\begin{equation}\label{1.13}
t = 0:\ \ u =  f(x),\ \ u_t = g(x),\ \  \ x \in \mathbb{R}^2
\end{equation}
with
\begin{eqnarray*}
\mathrm{Supp}\ \{ f, g\} \subset \{ x | \ |x|\leq 1\} \ \ \mbox{and}  \ \ \|f\|_{H^{s+1}} + \|g\|_{H^s} <  \varepsilon, \ \ \ s \geq 13.
\end{eqnarray*}
By the finite propagation speed of waves, we can obtain
\begin{eqnarray}\label{1.14}
\mathrm{Supp}\  u(t,\cdot) \subset \{ x | \ |x|\leq t + 1\}.
\end{eqnarray}

Then, we can get the following result
\begin{theorem}
 Under the assumption $\tilde{H}_1)$, there exists the global classical solutions to Cauchy problem (\ref{1.11}) (\ref{1.13}), provided that $\varepsilon$ is sufficiently small.
 \end{theorem}

\begin{remark}\label{Remark 4}
The above main result establishes some kind of stability of the traveling wave solution $F(x_1 + t)$ for the equation of time-like extremal hypersurface in Minkowski space. For the global solution $u, u_\eta$ and $u_{x_2}$ are decaying in $\xi$ with exponent $-1/4$ and $u_{\xi}$ are increasing in $\xi$ with the rate $(2+\xi)^{\delta}$. The parameter $\delta$ is the arbitrary small positive constant.
\end{remark}

\begin{remark}
Using the above main result, we can also get the interesting result  above the stability of certain kind of traveling wave solution  with the speed larger than 1. For $n = 2$, system (\ref{1.8}) can be considered as the time-like extremal hypersurface in Minkoski space $\mathbb{R}^{1+(1+1)}$. Using the result in \cite{Liu and Zhou 1}, there is an exact traveling solutions $\Phi(x_1\pm x'_2)$, The exact traveling wave solution of time-like extremal hypersurface in Minkowski space $R^{1+(2+1)}$ is $\Phi(x_1 \pm \frac{1}{\sqrt{c^2-1}}(x_2 - ct))$, where the speed of traveling wave $c$ is large than the speed of light. Using Lorentz transformation
\begin{equation*}
\tilde{x}_1 =\frac{\sqrt{c^2-1}}{c} x_1 + \frac{1}{c}x_2,\ \ \ \tilde{x}_2 =\frac{1}{c} x_1 -\frac{\sqrt{c^2-1}}{c} x_2
\end{equation*}
we can get the traveling wave solution $\Phi(\sqrt{\frac{c^2}{c^2-1}}(\tilde{x}_1 - t)) = \tilde{\Phi}(\tilde{x}_1 - t)$. By the above main result of Theorem 1.1, we can get the stability of this kind of traveling wave solution.
\end{remark}

\begin{remark}
The above result establishes the global existence of classical solutions for a class of large initial data of  quasilinear wave equations.
\end{remark}

By the local existence result of nonlinear wave equations with small initial data, we can get the classical solutions in the time interval $[-2, 0]$. For getting the global existence result for nonlinear wave equations, the now classical method is to use Lorentz invariance and introduce the Klainerman's vector fields. we will introduce operator $Z$ which are infinitesimal generators of the Lorentz group as follows
\begin{equation}\label{1.15}
Z = \{\partial_t, \partial_{x_1}, \partial_{x_2}, L_0, L_1, L_2, \Omega\}
\end{equation}
where
\begin{eqnarray}\label{1.16}
L_0 = t\partial_t + x_1 \partial_{x_1}+ x_2 \partial_{x_2}, L_1 = x_1 \partial_{t} + t \partial_{x_1}, L_2 = x_2 \partial_{t} + t \partial_{x_2}, \Omega = x_1 \partial_{x_2} - x_2 \partial_{x_1}.
\end{eqnarray}
However, $Z$ operators does not communicate with multiplication of $F'(x_1 + t)$ . Then we introduce the $\Gamma$ operator
\begin{equation}\label{1.17}
\Gamma = \{\Gamma_1, \Gamma_2, \Gamma_3, \Gamma_4, \Gamma_5, \Gamma_6\}
\end{equation}
where
\begin{eqnarray}\label{1.18}
\begin{cases}\Gamma_1 = \partial_t + \partial_{x_1}, \ \Gamma_2 =\partial_t - \partial_{x_1}, \ \Gamma_3 = \partial_{x_2}, \ \ \Gamma_4 = (t - x_1)(\partial_t - \partial_{x_1}) + x_2\partial_{x_2} = L_0 - L_1,\\
\Gamma_5 = (t + x_1)(\partial_t + \partial_{x_1}) + x_2\partial_{x_2} = L_0 + L_1, \Gamma_6 = (t + x_1)\partial_{x_2} + x_2 (\partial_t - \partial_{x_1})= L_2 + \Omega.
\end{cases}
\end{eqnarray}
These operators will communicate with the linearized equation for $u$. Compared to the $Z$ operators, the $\Gamma$ operator has just one operator less. However, this is the crucial point to prove our main result.
 Although, we can also get the decay in the global Klainerman Sobolev inequality in Klein-Gorden equations without the commutator $L_0$ and wave equations with multiple speed without the commutator $L_i$ by the Klainerman-Sideris inequality. Here we can not get any decay of the classical solutions in t direction with only $\Gamma$ operators. Because the effect of traveling wave solution, we will consider the system in Goursat coordinates. Fortunately, using the Goursat coordinates, we can get the decay in $\xi$ direction, which is also more weak (only of exponent $-1/4$) than the usual Cauchy problem (which is $-1/2$ in $t$ direction). It is another main difficulty in our problem. Therefore, in the following we can recast our problem to study the generalized Goursat problem. Let $\xi = t + x_1, \eta = t - x_1$, we consider the Goursat problem with the data as follows
 $$\xi  = -1: v = h_1(\eta, x_2);\ \ \ \ \ \eta = -1: v = h_2(\xi, x_2)$$
and satisfy the compatibility condition of order $s + 1$ at the line $(\xi, \eta) = (-1,-1)$. Moreover, by the local existence theorem of quasilinear wave equation, we have
$$\|h_1\|_{H^{s+1}} \leq C_0 \varepsilon, \ \ \ \ \ \|h_2\|_{H^{s+1}} \leq C_0 \varepsilon  $$
where $C_0$ is a positive constant independent of $\varepsilon$.
Moreover,
$$\mathrm{Supp}\ \ h_1 \subset \{-1 \leq \eta \leq 0, -1 \leq x_2 \leq 1\} $$
$$\mathrm{Supp}\ \ h_2 \subset \{-1 \leq \xi \leq 0, -1 \leq x_2 \leq 1\} $$
Noting (\ref{1.14}), we can get
\begin{eqnarray}\label{1.19}
\mathrm{Supp}\ \  v  \subset \{ (\xi , \eta, x_2)| \ |x_2|\leq  \sqrt{(2+\xi)(2+\eta)}\}.
\end{eqnarray}
Therefore, we will consider the generalized Goursat problem in coordinates $(\xi, \eta, x_2)$ instead of the original system.
%

\section{Preliminaries}
For getting the stability result of the traveling wave solutions, we will give the key estimates in this section, which plays an important role in  proving our main result.
Noting (\ref{1.18}), in coordinates $(\xi, \eta)$, we have
\begin{eqnarray}\label{2.3}
\begin{cases}\Gamma_1 = 2\partial_\xi, \ \Gamma_2 =2\partial_{\eta}, \ \Gamma_3 = \partial_{x_2}, \\
\Gamma_4 = 2\eta\partial_\eta  + x_2\partial_{x_2} ,\
\Gamma_5 = 2\xi\partial_\xi  + x_2\partial_{x_2} , \Gamma_6 = \xi \partial_{x_2} + 2 x_2 \partial_\eta.
\end{cases}
\end{eqnarray}
Firstly, the elementary facts about $\Gamma$ operators are as follows
\begin{lemma}\label{lem 2.2}
(\cite{Li and zhou})Noting the relations of $Z$ and $\Gamma$, we can easily get
\begin{eqnarray}\label{2.1}
&&\Gamma^k Q_0(\phi, \psi) = \sum\limits_{0\leq k_1 + k_2 \leq k} A_{k_1, k_2} Q_0(\Gamma^{k_1}\phi, \Gamma^{k_2}\psi),\\  \label{2.2}
&&\Box \Gamma^k v = \Gamma^k \Box v + \sum\limits_{k' < k} A^{(1)}_{k', k} \Gamma^{k'} \Box v.
\end{eqnarray}
\end{lemma}

Through a simple computation, we can also get
\begin{equation}\label{2.4}
\Box = 4 \partial_{\xi\eta} - \partial_{x_2x_2}, \ \ Q_0(\phi, \psi) = 2(\phi_\xi \psi_\eta + \phi_\eta\psi_\xi) - \phi_{x_2}\psi_{x_2}
\end{equation}

In the following, we will consider the estimates of the commutators in coordinates $(\xi, \eta, x_2)$.
\begin{lemma}
In Goursat coordinates, for the null form $Q_0$, there hold
\begin{eqnarray}\label{2.5}
&&|Q_0(\phi, \psi)| \lesssim (2+\xi)^{-1} (|\Gamma\phi||\nabla \psi| + |\nabla \phi||\Gamma \psi|),\\ \label{2.6}
&&|Q_0(\phi, \psi)| \lesssim (2+\eta)^{-1} [|\Gamma\phi|(|\psi_\xi|+ |\psi_{x_2}|) + (|\phi_\xi|+ |\phi_{x_2}|)|\Gamma \psi|),
\end{eqnarray}
where $\nabla = \{\partial_{\eta}, \partial_{x_2}\}$.
\end{lemma}

\begin{proof}
Noting the null form of (\ref{2.4}), we have
\begin{eqnarray*}
&&(2+\xi) Q_0(\phi, \psi) = 2Q_0(\phi, \psi)+\xi[2(\phi_\xi \psi_\eta + \phi_\eta\psi_\xi) - \phi_{x_2}\psi_{x_2}]\\
&& = 2Q_0(\phi, \psi)+ \Gamma_5 \phi  \psi_\eta+\Gamma_5 \psi  \phi_\eta - \frac{1}{2}\Gamma_6 \psi \phi_{x_2} - \frac{1}{2}\Gamma_6 \phi \psi_{x_2}
\end{eqnarray*}
Then, we can get the estimate of (\ref{2.5}). By the similar way, we can easily obtain the estimate (\ref{2.6}).
\end{proof}

\begin{lemma}\label{lem1}
Let $\phi$ having the compact support as (\ref{1.19}), we have
\begin{eqnarray}\label{2.7}
&&|\frac{\phi(\xi, \cdot)}{\sqrt{(3+\xi)(3+\eta)} -|x_2|}|_{L^2(D)}\lesssim | \phi_{x_2}(\xi, \cdot)|_{L^2(D)},\\ \label{2.8}
&&\frac{|\phi(\xi, \eta, x_2)|}{\sqrt{(3+\xi)(3+\eta)} -|x_2|}\lesssim \sup\limits_{x_2}| \phi_{x_2}(\xi, \cdot)|.
\end{eqnarray}
where, the domain $D = \{(\eta, x_2) | \ -1\leq \eta< +\infty, \ \ -\infty < x_2 < +\infty\}$.
\end{lemma}

\begin{proof}
We first prove the estimate (\ref{2.7}). It is only necessary to prove
\begin{eqnarray}\label{2.9}
|\frac{f(x_2)}{a - |x_2|}|_{L^2(\mathbb{R})} \leq 2 |f_{x_2}|_{L^2(\mathbb{R})},
\end{eqnarray}
provided that $\mathrm{Supp} f \subset \{x_2 \ |\ |x_2|< a\} $.
In fact, we can get the desired estimate by takeing $ a = \sqrt{(3+\xi)(3+\eta)}, f = \phi$ and taking $L^2$-norm on the both side of (\ref{2.9}) for $\eta$.
\begin{eqnarray*}
&&|\frac{f(x_2)}{a - |x_2|}|^2_{L^2} = \int_{-(a-\varepsilon)}^{a-\varepsilon}\frac{|f(x_2)|^2}{(a - |x_2|)^2}d x_2 \\
&& = \int_{-(a-\varepsilon)}^{0}\frac{|f(x_2)|^2}{(a + x_2)^2}d x_2 + \int_{0}^{a-\varepsilon}\frac{|f(x_2)|^2}{(a - x_2)^2}d x_2\\
&& = -2\frac{f^2(0)}{a} +2 \int_{-(a-\varepsilon)}^{0}\frac{f f'}{a + x_2}d x_2 - 2\int_{0}^{a-\varepsilon}\frac{f f'}{a - x_2}d x_2\\
&&\leq 2|\frac{f(x_2)}{a - |x_2|}|_{L^2}|f_{x_2}|_{L^2}.
\end{eqnarray*}
In a similar way, to prove (\ref{2.8}), we only need to get
\begin{equation}\label{2.10}
\frac{|f(x_2)|}{a - |x_2|} \leq \sup\limits_{x_2}|f_{x_2}|.
\end{equation}
Without loss of generality, we may assume $x_2 > 0$ and $f(a) = 0$, then
\begin{eqnarray*}
|f(x_2)| = |-\int_{x_2}^a f_{y}dy|\leq \sup\limits_{x_2}|f_{x_2}|(a - x_2).
\end{eqnarray*}
\end{proof}

\subsection{Sobolev inequality}
For proving our main result, in this subsection we will give Sobolev inequalities as follows
\begin{proposition}\label{1}
Let $\phi$ having the compact support as (\ref{1.19}), we have
\begin{equation}\label{2.11}
|\phi(\xi, \eta, x_2)|\lesssim (2+\eta )^{-\frac{1}{4}}(2+\xi )^{-\frac{1}{4}} \sum\limits_{0\leq |k_1|,|k_2|\leq 1} |\Gamma^{k_1}\nabla^{k_2}\phi(\xi, \cdot)|_{L^2(D)}
\end{equation}
where  $\nabla = \{ \partial_{\eta}, \partial_{x_2}\}$.
\begin{proof}
We first consider the case $|x_2| \leq \frac{\sqrt{(2+ \xi)(2 + \eta)}}{4}$. Then,  we have
\begin{eqnarray}\nonumber
&&|\phi(\xi, \eta, \cdot)|_{L^\infty(|x_2| \leq  \lambda)}\\  \label{2.12}
&&\leq C |\phi(\xi, \eta, \cdot)|_{L^2(|x_2| \leq \ \lambda)}^\frac{1}{2}(|\partial_{x_2} \phi(\xi, \eta, \cdot)|_{L^2(|x_2| \leq  \lambda)}^\frac{1}{2} + \frac{1}{ \lambda^{\frac{1}{2}}}|\phi(\xi, \eta, \cdot)|_{L^2(|x_2| \leq \lambda)}^\frac{1}{2}).
\end{eqnarray}
When $\lambda = 1$, the above inequality follows from Nirenberg's inequality. Then the general case follows from the scaling. In our case, we take
$$\lambda = \frac{\sqrt{(2+ \xi)(2 + \eta)}}{4}.$$
Noting the definition of $\Gamma$, we can get
\begin{eqnarray*}
(2+ \eta)(\Gamma_6\phi +2 \phi_{x_2}) - x_2(\Gamma_4\phi + 4 \phi_\eta) = [(2+\xi)(2+\eta) - x_2^2
]\phi_{x_2}
\end{eqnarray*}
Then,
\begin{equation} \label{2.13}
|\phi_{x_2}| \leq [(2+\eta )^{-\frac{1}{2}}(2+\xi )^{-\frac{1}{2}} + (2+\xi )^{-1}] |\Gamma \phi|.
\end{equation}
Therefore,
\begin{equation}\label{2.14}
|\phi(\xi, \eta, \cdot)|_{L^\infty(|x_2| \leq \frac{\sqrt{(2+ \xi)(2 + \eta)}}{4})} \lesssim [(2+\eta )^{-\frac{1}{4}}(2+\xi )^{-\frac{1}{4}} + (2+\xi )^{-\frac{1}{2}}] \sum\limits_{|k|\leq 1}|\Gamma^k \phi(\xi, \eta, \cdot)|_{L^2(\mathbb{R})}
\end{equation}
When $\xi \geq \eta$, we apply one dimensional Sobolev inequality for the $\eta$
variable to get
\begin{equation}\label{2.15}
|\phi(\xi, \eta, x_2)|\lesssim (2+\eta )^{-\frac{1}{4}}(2+\xi )^{-\frac{1}{4}} \sum\limits_{0\leq |k_1|,|k_2|\leq 1} |\Gamma^{k_1}\nabla^{k_2}\phi(\xi, \cdot)|_{L^2(D)}\ \ \ \mbox{for}\ \  |x_2| \leq \frac{\sqrt{(2+ \xi)(2 + \eta)}}{4}.
\end{equation}
When $\xi \leq \eta$, we have
\begin{eqnarray*}
[(2+\xi)(2+\eta) - x_2^2] \partial_{\eta} = \frac{1}{2}[(2+\xi)(\Gamma_4 + 4\partial_\eta) - x_2(\Gamma_6 + 2\partial_{x_2})].
\end{eqnarray*}
Therefore
$$|\phi_\eta| \lesssim (2+\eta )^{-\frac{1}{2}}(2+\xi )^{-\frac{1}{2}}|\Gamma \phi|.$$
Noting that
\begin{eqnarray*}
\phi (\xi, \eta, x_2)^2 = - \int_\eta^\infty \partial_\eta \phi^2 d\eta ,
\end{eqnarray*}
then apply one dimensional Sobolev inequality for the $x_2$ variable, we can get
\begin{equation}\label{2.16}
|\phi(\xi, \eta, x_2)|\lesssim (2+\eta )^{-\frac{1}{4}}(2+\xi )^{-\frac{1}{4}} \sum\limits_{0\leq |k_1|,|k_2|\leq 1} |\Gamma^{k_1}\nabla^{k_2}\phi(\xi, \cdot)|_{L^2(D)}\ \ \ \mbox{for}\ \  |x_2| \leq \frac{\sqrt{(2+ \xi)(2 + \eta)}}{4}.
\end{equation}

On the other hand, for the case of $|x_2| \geq \frac{\sqrt{(2+ \xi)(2 + \eta)}}{4}$, we introduce the polar coordinates
$$\sqrt{2+\eta} = r \cos \theta, \ \ \ x_2 = r \sin \theta.$$
Then, we can get
\begin{equation}\label{2.17}
 r\partial_r = 2(2+\eta)\partial_\eta + x_2\partial_{x_2} = 4 \partial_\eta + \Gamma_4, \ \ \ \partial_\theta = \sqrt{2+\eta}\partial_{x_2} - x_2 \partial_{\sqrt{2+\eta}} = \sqrt{2+\eta} (\partial_{x_2} - x_2\partial_\eta)
\end{equation}
\begin{equation}\label{2.18}
2\sqrt{2+\eta}r dr d\theta =  d\eta dx_2
\end{equation}
By Nirenberg's inequality,
\begin{eqnarray*}
\sup\limits_{\theta\in \mathbb{S}^1} |\phi(\xi, r, \theta)|^2 &\leq& |\phi(\xi, r, \cdot)|_{L^2{(\theta\in \mathbb{S}^1)}}(|\partial_\theta \phi(\xi, r, \cdot)|_{L^2{(\theta\in \mathbb{S}^1)}} + |\phi(\xi, r, \cdot)|_{L^2{(\theta\in \mathbb{S}^1)}}) \\
&=& -\int_r^\infty \partial_r \{|\phi(\xi, r, \cdot)|_{L^2{(\theta\in \mathbb{S}^1)}}(|\partial_\theta \phi(\xi, r, \cdot)|_{L^2{(\theta\in \mathbb{S}^1)}} + |\phi(\xi, r, \cdot)|_{L^2{(\theta\in \mathbb{S}^1})})\}\\
\end{eqnarray*}
Noting,
\begin{eqnarray*}
-\partial_r |\phi(\xi, r, \cdot)|_{L^2{(\theta\in \mathbb{S}^1)}} = -\frac{1}{2|\phi(\xi, r, \cdot)|_{L^2{(\theta\in \mathbb{S}^1)}}}\int_{\theta\in \mathbb{S}^1} \phi \phi_r d\theta \leq  |\phi_r(\xi, r, \cdot)|_{L^2{(\theta\in \mathbb{S}^1)}}
\end{eqnarray*}
Similarly, we can get
\begin{eqnarray*}
-\partial_r |\phi_\theta(\xi, r, \cdot)|_{L^2{(\theta\in \mathbb{S}^1)}}  \leq  |\phi_{\theta r}(\xi, r, \cdot)|_{L^2{(\theta\in \mathbb{S}^1)}}
\end{eqnarray*}
Then,
\begin{eqnarray}\nonumber
\sup\limits_{\theta\in \mathbb{S}^1} |\phi(\xi, r, \theta)|^2
 &\lesssim& \frac{1}{\sqrt{(2+ \xi)(2 + \eta)}}\int_r^\infty r \{|\phi_r(\xi, r, \cdot)|_{L^2{(\theta\in \mathbb{S}^1)}}(|\partial_\theta \phi(\xi, r, \cdot)|_{L^2{(\theta\in \mathbb{S}^1)}}
  +|\phi(\xi, r, \cdot)|_{L^2{(\theta\in \mathbb{S}^1})}) \\ \label{2.19}
  &+&|\phi(\xi, r, \cdot)|_{L^2{(\theta\in \mathbb{S}^1)}}(| \phi_{r\theta}(\xi, r, \cdot)|_{L^2{(\theta\in \mathbb{S}^1)}} + |\phi_r(\xi, r, \cdot)|_{L^2{(\theta\in \mathbb{S}^1})})\}dr
\end{eqnarray}
Noting (\ref{2.17})(\ref{2.18}), it is not difficulty to get
\begin{equation}\label{2.20}
|\phi(\xi, \eta, x_2)|\lesssim (2+\xi )^{-\frac{1}{4}}(2+\eta )^{-\frac{1}{4}}\sum\limits_{0\leq |k_1|,|k_2|\leq 1} |\Gamma^{k_1}\partial^{k_2}\phi(\xi, \cdot)|_{L^2(D)}\ \ \ \mbox{for}\ \  |x_2| \geq \frac{\sqrt{(2+ \xi)(2 + \eta)}}{4}.
\end{equation}
Therefore,  the estimate (\ref{2.11}) follows from (\ref{2.16}) and (\ref{2.19}).
\end{proof}
\end{proposition}

\begin{proposition}\label{2}
Let $\phi$ has the compact support as (\ref{1.19}), we can obtain
\begin{equation}\label{2.21}
|\phi_{\eta}(\xi, \eta, x_2)|\lesssim (2+\eta )^{-\frac{3}{4}}(2+\xi )^{\frac{1}{4}} \sum\limits_{0\leq |k_1|+|k_2|\leq 2} |\Gamma^{k_1}\nabla^{k_2}\phi_{x_2}(\xi, \cdot)|_{L^2(D)}
\end{equation}
\begin{equation}\label{2.22}
|\phi_{\xi}(\xi, \eta, x_2)|\lesssim (2+\eta )^{\frac{1}{4}}(2+\xi )^{-\frac{3}{4}} \sum\limits_{0\leq |k_1|+|k_2|\leq 2} |\Gamma^{k_1}\nabla^{k_2}\phi_{x_2}(\xi, \cdot)|_{L^2(D)}
\end{equation}
\end{proposition}
\begin{proof}
\begin{displaymath}
2(2+\eta)\phi_{\eta}=(4\partial_{\eta}+\Gamma_4-x_2\partial_{x_2})\phi
\end{displaymath}
By Proposition \ref{1},
\begin{eqnarray*}
|x_2\partial_{x_2} \phi| & \leq & (2+\eta)^{\frac{1}{2}}(2+\xi)^{\frac{1}{2}} |\partial_{x_2} \phi| \\
& \leq &  (2+\eta)^{\frac{1}{4}}(2+\xi)^{\frac{1}{4}} \sum\limits_{0\leq |k_1|,|k_2|\leq 1} |\Gamma^{k_1}\nabla^{k_2}\phi_{x_2}|_{L^2(D)} \\
\end{eqnarray*}
By Lemma \ref{lem1} and Proposition \ref{1}, we get
\begin{eqnarray*}
(4\partial_{\eta}+\Gamma_4)\phi &=&  (2+\eta)^{\frac{1}{2}}(2+\xi)^{\frac{1}{2}} \sup\limits_{x_2} \left| \frac{(4\partial_{\eta}+\Gamma_4)\phi}{ (2+\eta)^{\frac{1}{2}}(2+\xi)^{\frac{1}{2}}-|x_2|} \right| \\
& \leq &  (2+\eta)^{\frac{1}{2}}(2+\xi)^{\frac{1}{2}} \sup\limits_{x_2} \left| (4\partial_{\eta}+\Gamma_4)\phi_{x_2} \right| \\
& \leq &  (2+\eta)^{\frac{1}{4}}(2+\xi)^{\frac{1}{4}} \sum\limits_{0\leq |k_1|+|k_2|\leq 2} |\Gamma^{k_1}\nabla^{k_2}\phi_{x_2}|_{L^2(D)}
\end{eqnarray*}
Combining the above estimates, we can get the proof of the estimate (\ref{2.21}). Using the similar procedures, we can get the estimate $(\ref{2.22})$.
\end{proof}

\begin{corollary}
Combining Proposition 1 and Proposition 2, we can get
\begin{equation}\label{2.23}
|\phi_{\eta}(\xi, \eta, x_2)|\lesssim (2+\eta )^{-\frac{1}{2}} \sum\limits_{0\leq |k_1|+|k_2+1|\leq 2} |\Gamma^{k_1}\nabla^{k_2}\phi(\xi, \cdot)|_{L^2(D)}
\end{equation}
\end{corollary}
%
%

%

\section{Stability of the traveling wave solution with the form $F(x_1 + t)$ in $\mathbb{R}^{1+(2+1)}$}
In the following we will prove the stability of the traveling wave solution with the form $F(x_1 + t)$, i.e. we take $a=0, b = 1$ in (\ref{1.10}). Then we can get the perturbed system of timelike extremal hypersurface in Minkowski space $\mathbb{R}^{1+(2+1)}$ as follows
\begin{equation}\label{3.1}
(\frac{u_t + F'}{\sqrt{1 - Q_0(u, u) - 2 F'(u_t - u_{x_1})}})_t - (\frac{u_{x_1}+ F'}{\sqrt{1 - Q_0(u, u) - 2 F'(u_t - u_{x_1})}})_{x_1}- (\frac{u_{x_2}}{\sqrt{1 - Q_0(u, u) - 2 F'(u_t - u_{x_1})}})_{x_2} = 0.
\end{equation}
We rewrite the system (\ref{3.1}) as the following form
\begin{equation}\label{3.2}
\Box u = - \frac{Q_0(u + F, Q_0(u, u) + 2F'(u_t - u_{x_1}))}{2(1 - Q_0(u,u)- 2F'(u_t - u_{x_1}))}.
\end{equation}
Under the weaker assumption than ($\tilde{H}_1$)
\begin{equation*}
H_1)\ \ \  | (\xi \frac{d}{d\xi})^{k_2} (\frac{d}{d\xi})^{k_1} F'(\xi)| \leq C_{k_1, k_2}(2+\xi )^{-1}
\end{equation*}
we will prove the stability of traveling wave solution $F(x_1 + t)$ to system (\ref{3.1}). The proof of the general traveling wave solution (\ref{1.10}) is similar to that of above theorem. In the end of this paper, we will point out the key difference in the proof.

We can rewrite system $(\ref{3.2})$ in coordinates $(\xi, \eta, x_2)$ as following
\begin{eqnarray}\label{3.3}
\Box u  - 4F'^2 u_{\eta\eta} &=& 4 \partial_{\xi\eta} u - \partial_{x_2x_2} u - 4F'^2 u_{\eta\eta}\\ \nonumber
&=& \frac{1}{2}(1-H)[Q_0(u, Q_0(u, u))+ 4F' Q_0(u, u_{\eta}) + 6u_{\xi}F'u_{\eta\eta}+ 8 F'' u_{\eta}^2] - 4F'^2u_{\eta\eta} H
\end{eqnarray}
where $H (\varsigma) = 1 + \frac{1}{1 - \varsigma}$.
Taking the operator $\Gamma^k$ to the above equation and denoting $u_k = \Gamma^k u$, we have
\begin{eqnarray}\label{3.4}
\Box u_k  - 4\Gamma^k(F'^2 u_{\eta\eta})&=& 4 \partial_{\xi\eta} u_k - \partial_{x_2x_2} u_k - 4\Gamma^k(F'^2 u_{\eta\eta})\\ \nonumber
&=& \frac{1}{2}(1-H)[Q_0(u, 2Q_0(u, u_k))+ 6F' Q_0(u, u_{k\eta})+ J_k] - 4\sum\limits_{k_1+k_2=k, k_2 \geq 1}\Gamma^{k_1}(F'^2u_{k\eta\eta})\Gamma^{k_2}H
\end{eqnarray}
with
$$J_k = \sum\limits_{k_1+k_2+k_3+k_4 = k, k_3 < k, k_4 < k} \Gamma^{k_1}(1-H)Q_0(\Gamma^{k_2}u, Q_0(\Gamma^{k_3}u, \Gamma^{k_4}u))+ \cdots + 4 \Gamma^k(F''u^2_\eta).$$

For proving main result, we will give the energy estimates. Define the higher order energy
\begin{equation}
\begin{split}
  E_s=&\sum_{|k|\leq s}\Big(\iiint\frac{u_{k\eta}^2}{(2+\xi)^{1+\frac1{10}}(2+\eta)^\frac1{10}}d\xi d\eta dx_2+\iiint\frac{u_{k\xi}^2}{(2+\eta)^{1+\frac1{10}}(2+\xi)^{\frac1{10}}}d\xi d\eta dx_2\\ \label{3.5}
  &+\iiint\frac{u_{k x_2}^2}{(2+\xi)^{1+\frac1{10}}(2+\eta)^{\frac1{10}}}d\xi d\eta dx_2+\iiint\frac{u_{k x_2}^2}{(2+\eta)^{1+\frac1{10}}(2+\xi)^{\frac1{10}}}d\xi d\eta dx_2\Big),
\end{split}
\end{equation}
and the lower order energies
\begin{align} \label{3.6}
  e_s=\sup_\xi\sum_{|l|\leq s-7}\iint(u_{l\eta}^2+u_{lx_2}^2)d\eta dx_2.
\end{align}
For getting the low order energy estimate, we also introduce the weighted lower derivative $L^\infty$ norm estimates
\begin{eqnarray}\label{3.7}
 \tilde{e}_s = (2+\xi)^{-\delta}\sum\limits_{|k|\leq [\frac{s-7}{2}]+1} |\Gamma^k u|.
\end{eqnarray}
where the parameter $\delta$ is an arbitrary positive small constant less than $\frac{3}{20}$.

\subsection{Higher order energy estimates}
In the subsection we will give the energy estimates terms by terms. For get our main result, firstly we introduce the weight function $B(\xi)$ satisfying
\begin{eqnarray*}
B'(\xi) =  F'^2(\xi).
\end{eqnarray*}
Then, we have
\begin{eqnarray*}
m \leq e^{-B(\xi)} \leq M.
\end{eqnarray*}
where $m, M$ are positive constants.

Firstly, multiplying $(2+\xi)^{-\frac{1}{10}}(2+\eta)^{-\frac{1}{10}}u_{k\eta}e^{-B(\xi)}$ to the equation (\ref{3.4}), we can get the left hand side term
\begin{eqnarray*}
&&\iiint (2+\xi)^{-\frac{1}{10}}(2+\eta)^{-\frac{1}{10}} e^{-B(\xi)}u_{k\eta}[\Box u_k - 4\Gamma^k(F'^2u_{\eta\eta})]d\xi d\eta dx_2 \\
&=& \iiint (2+\xi)^{-\frac{1}{10}}(2+\eta)^{-\frac{1}{10}} e^{-B(\xi)}u_{k\eta} [4 \partial_{\xi\eta} u_k - \partial_{x_2x_2} u_k - 4\Gamma^k(F'^2 u_{\eta\eta})] d\xi d\eta dx_2 \\
&=& \iiint (2+\xi)^{-\frac{1}{10}}(2+\eta)^{-\frac{1}{10}}e^{-B(\xi)}[2\frac{d}{d\xi} u^2_{k\eta} - \frac{d}{d x_2}(u_{k x_2} u_{k\eta}) + \frac{1}{2}\frac{d}{d\eta} u^2_{kx_2}- 2F'^2\frac{d}{d\eta}  u^2_{k\eta}] d\xi d\eta dx_2\\
&=& \iiint 2\frac{d}{d\xi}[(2+\xi)^{-\frac{1}{10}}(2+\eta)^{-\frac{1}{10}} e^{-B(\xi)}u^2_{k\eta}] + \frac{1}{2}\frac{d}{d\eta}[(2+\xi)^{-\frac{1}{10}}(2+\eta)^{-\frac{1}{10}} e^{-B(\xi)}u^2_{kx_2}] \\ \nonumber
&-&  \frac{1}{2}\frac{d}{dx_2}[(2+\xi)^{-\frac{1}{10}}(2+\eta)^{-\frac{1}{10}} e^{-B(\xi)}u_{k\eta}u_{kx_2}]  d\xi d\eta dx_2 
\\
&+& \frac{1}{5}\iiint (2+\xi)^{-\frac{11}{10}}(2+\eta)^{-\frac{1}{10}} e^{-B(\xi)}u^2_{k\eta}  d\xi d\eta dx_2 + \frac{1}{20} \iiint (2+\xi)^{-\frac{1}{10}}(2+\eta)^{-\frac{11}{10}}e^{-B(\xi)} u^2_{kx_2} d\xi d\eta dx_2
\end{eqnarray*}
 The right hand side terms are as follows
\begin{eqnarray*}
&&\iiint (2+\xi)^{-\frac{1}{10}}(2+\eta)^{-\frac{1}{10}}u_{k\eta} e^{-B(\xi)}(1-H)[Q_0(u, 2Q_0(u, u_k))+ 4F' Q_0(u, u_{k\eta}) + J_k\\
&&- 4\sum\limits_{k_1+k_2=k}\Gamma^{k_1}(F'^2u_{\eta\eta})\Gamma^{k_2}H] d\xi d\eta dx_2
\end{eqnarray*}
By the bound of $e^{-B(\xi)}$ and the decay property of $B'(\xi)$, without loss of generality, we can estimate the right hand terms above without the weighted function $e^{-B(\xi)}$. Then, the first term of the right hand side can be rewritten as
\begin{eqnarray*}
&&\iiint (2+\xi)^{-\frac{1}{10}}(2+\eta)^{-\frac{1}{10}}u_{k\eta} (1-H)Q_0(u, 2Q_0(u, u_k)) d\xi d\eta dx_2\\
 &=& \iiint (2+\xi)^{-\frac{1}{10}}(2+\eta)^{-\frac{1}{10}}u_{k\eta} (1-H)[4 u_{\xi}Q_{0\eta} + 4 u_{\eta}Q_{0\xi} - 2u_{x_2}Q_{0x_2}] d\xi d\eta dx_2\\
&=&\iiint 4\{\frac{d}{d\eta}[(2+\xi)^{-\frac{1}{10}}(2+\eta)^{-\frac{1}{10}}(1-H)u_{\xi}u_{k\eta} Q_0 ] - (2+\xi)^{-\frac{1}{10}}(2+\eta)^{-\frac{1}{10}}(1-H)u_{\xi\eta}u_{k\eta} Q_0\\ &-& [(2+\xi)^{-\frac{1}{10}}(2+\eta)^{-\frac{1}{10}}(1-H)u_{k\eta}]_{\eta} u_{\xi} Q_0\}
+ 4\{\frac{d}{d\xi}[(2+\xi)^{-\frac{1}{10}}(2+\eta)^{-\frac{1}{10}}(1-H)u_{\eta}u_{k\eta} Q_0]  \\
&-& (2+\xi)^{-\frac{1}{10}}(2+\eta)^{-\frac{1}{10}}(1-H)u_{\xi\eta}u_{k\eta} Q_0 - [(2+\xi)^{-\frac{1}{10}}(2+\eta)^{-\frac{1}{10}}(1-H)u_{k\eta}]_{\xi}u_{\eta} Q_0\} \\
&-& 2 \{\frac{d}{d x_2}[(2+\xi)^{-\frac{1}{10}}(2+\eta)^{-\frac{1}{10}}(1-H)u_{x_2}u_{k\eta} Q_0]  - (2+\xi)^{-\frac{1}{10}}(2+\eta)^{-\frac{1}{10}}(1-H)u_{x_2x_2}u_{k\eta} Q_0 \\
&-& [(2+\xi)^{-\frac{1}{10}}(2+\eta)^{-\frac{1}{10}}(1-H)u_{k\eta}]_{x_2}u_{x_2} Q_0\} d\xi d\eta dx_2\\
&=&  \iiint 4 \frac{d}{d\eta}[(2+\xi)^{-\frac{1}{10}}(2+\eta)^{-\frac{1}{10}}(1-H)u_{\xi}u_{k\eta} Q_0 ] + 4 \frac{d}{d\xi}[(2+\xi)^{-\frac{1}{10}}(2+\eta)^{-\frac{1}{10}}(1-H)u_{\eta}u_{k\eta} Q_0]d\xi d\eta dx_2 \\
&-& 2 \frac{d}{d x_2}[(2+\xi)^{-\frac{1}{10}}(2+\eta)^{-\frac{1}{10}}(1-H)u_{x_2}u_{k\eta} Q_0]d\xi d\eta dx_2
- 2 \iiint Q_0((2+\xi)^{-\frac{1}{10}}(2+\eta)^{-\frac{1}{10}}(1-H)u_{k\eta}, u)Q_0(u_k,u)d\xi d\eta dx_2\\
 &-&  2\iiint (2+\xi)^{-\frac{1}{10}}(2+\eta)^{-\frac{1}{10}}(1-H)u_{k\eta}(\Box u + 4F'^2u_{\eta\eta}) Q_0(u_k,u)d\xi d\eta dx_2.
\end{eqnarray*}
The second term of the right hand side is
\begin{eqnarray*}
 && 2\iiint (2+\xi)^{-\frac{1}{10}}(2+\eta)^{-\frac{1}{10}}(1-H)F' Q_0(u, u^2_{k\eta}) \\
&=& 2\iiint (2+\xi)^{-\frac{1}{10}}(2+\eta)^{-\frac{1}{10}}(1-H)F'[2 u_{\xi}(u^2_{k\eta})_{\eta} + 2 u_{\eta}(u^2_{k\eta})_{\xi} -  u_{x_2}(u^2_{k\eta})_{x_2} ]d\xi d\eta dx_2\\
&=& \iiint 4 \frac{d}{d\eta}[(2+\xi)^{-\frac{1}{10}}(2+\eta)^{-\frac{1}{10}}(1-H)F' u_{\xi}u^2_{k\eta}]- 4\frac{d}{d\eta}[(2+\xi)^{-\frac{1}{10}}(2+\eta)^{-\frac{1}{10}}(1-H)F' u_{\xi}]u^2_{k\eta}\\
&+& 4 \frac{d}{d\xi}[(2+\xi)^{-\frac{1}{10}}(2+\eta)^{-\frac{1}{10}}(1-H)F' u_{\eta}u^2_{k\eta}]- 4\frac{d}{d\xi}[(2+\xi)^{-\frac{1}{10}}(2+\eta)^{-\frac{1}{10}}(1-H)F' u_{\eta}]u^2_{k\eta}\\
&-& 2 \frac{d}{dx_2}[(2+\xi)^{-\frac{1}{10}}(2+\eta)^{-\frac{1}{10}}(1-H)F' u_{x_2}u^2_{k\eta}] + 2\frac{d}{d x_2}[(2+\xi)^{-\frac{1}{10}}(2+\eta)^{-\frac{1}{10}}(1-H)F' u_{x_2}]u^2_{k\eta} d\xi d\eta dx_2\\
&=& \iiint 4 \frac{d}{d\eta}[(2+\xi)^{-\frac{1}{10}}(2+\eta)^{-\frac{1}{10}}(1-H)F' u_{\xi}u^2_{k\eta}] + 4 \frac{d}{d\xi}[(2+\xi)^{-\frac{1}{10}}(2+\eta)^{-\frac{1}{10}}(1-H)F' u_{\eta}u^2_{k\eta}]\\
&-& 2 \frac{d}{dx_2}[(2+\xi)^{-\frac{1}{10}}(2+\eta)^{-\frac{1}{10}}(1-H)F' u_{x_2}u^2_{k\eta}] - 2[(2+\xi)^{-\frac{1}{10}}(2+\eta)^{-\frac{1}{10}}(1-H)F'](\Box u + 4F'^2u_{\eta\eta}) u^2_{k\eta} \\
&-& 4[(2+\xi)^{-\frac{1}{10}}(2+\eta)^{-\frac{1}{10}}(1-H)F']_{\eta}u_{\xi} u^2_{k\eta}- 4[(2+\xi)^{-\frac{1}{10}}(2+\eta)^{-\frac{1}{10}}(1-H)F']_{\xi} u_{\eta}u^2_{k\eta}\\
&+& 2[(2+\xi)^{-\frac{1}{10}}(2+\eta)^{-\frac{1}{10}}(1-H)F']_{x_2} u_{x_2}u^2_{k\eta} d\xi d\eta dx_2.
\end{eqnarray*}
Noting
\begin{eqnarray*}
 &&Q_0((2+\xi)^{-\frac{1}{10}}(2+\eta)^{-\frac{1}{10}}(1-H)u_{k\eta}, u)Q_0(u_k,u) \\
 &=& (2+\xi)^{-\frac{1}{10}}(2+\eta)^{-\frac{1}{10}}(1-H)Q_0(u_{k\eta},u)Q_0(u_k,u) + u_{k\eta} Q_0((2+\xi)^{-\frac{1}{10}}(2+\eta)^{-\frac{1}{10}}(1-H), u)Q_0(u_k,u)\\
 &=& \frac{1}{2}(2+\xi)^{-\frac{1}{10}}(2+\eta)^{-\frac{1}{10}}(1-H)\frac{d}{d\eta}Q^2_0(u_k,u) - (2+\xi)^{-\frac{1}{10}}(2+\eta)^{-\frac{1}{10}}(1-H)Q_0(u_{k},u_{\eta})Q_0(u_k,u)  \\
 &+& u_{k\eta} Q_0((2+\xi)^{-\frac{1}{10}}(2+\eta)^{-\frac{1}{10}}(1-H), u)Q_0(u_k,u)
\end{eqnarray*}
Then, the right hand side term as follows
\begin{eqnarray*}\nonumber
&&\iiint (2+\xi)^{-\frac{1}{10}}(2+\eta)^{-\frac{1}{10}}u_{k\eta} (1-H)[Q_0(u, 2Q_0(u, u_k))+ 4F' Q_0(u, u_{k\eta}) + J_k - 4\sum\limits_{k_1+k_2=k}\Gamma^{k_1}(F'^2u_{k\eta\eta})\Gamma^{k_2}H] d\eta dx_2\\ \nonumber
&& = \iiint 4 \frac{d}{d\eta}[(2+\xi)^{-\frac{1}{10}}(2+\eta)^{-\frac{1}{10}}(1-H)u_{\xi}u_{k\eta} Q_0 ] + 4 \frac{d}{d\xi}[(2+\xi)^{-\frac{1}{10}}(2+\eta)^{-\frac{1}{10}}(1-H)u_{\eta}u_{k\eta} Q_0] \\
&&- 2 \frac{d}{d x_2}[(2+\xi)^{-\frac{1}{10}}(2+\eta)^{-\frac{1}{10}}(1-H)u_{x_2}u_{k\eta} Q_0]d\xi d\eta dx_2\\
&&-  \iiint \frac{d}{d\eta} [(2+\xi)^{-\frac{1}{10}}(2+\eta)^{-\frac{1}{10}}(1-H)Q^2_0(u_k,u)] - [(2+\xi)^{-\frac{1}{10}}(2+\eta)^{-\frac{1}{10}}(1-H)]_{\eta}Q^2_0(u_k,u)\\
&& - 2(2+\xi)^{-\frac{1}{10}}(2+\eta)^{-\frac{1}{10}}(1-H)Q_0(u_{k},u_{\eta})Q_0(u_k,u) + 2u_{k\eta} Q_0((2+\xi)^{-\frac{1}{10}}(2+\eta)^{-\frac{1}{10}}(1-H), u)Q_0(u_k,u)d\xi d\eta dx_2\\  &&- 2 \iiint (2+\xi)^{-\frac{1}{10}}(2+\eta)^{-\frac{1}{10}}(1-H)u_{k\eta}(\Box u + 4F'^2u_{\eta\eta}) Q_0(u_k,u)d\xi d\eta dx_2\\
&&+\iiint 4 \frac{d}{d\eta}[(2+\xi)^{-\frac{1}{10}}(2+\eta)^{-\frac{1}{10}}(1-H)F' u_{\xi}u^2_{k\eta}] + 4 \frac{d}{d\xi}[(2+\xi)^{-\frac{1}{10}}(2+\eta)^{-\frac{1}{10}}(1-H)F' u_{\eta}u^2_{k\eta}]\\
&&- 2 \frac{d}{dx_2}[(2+\xi)^{-\frac{1}{10}}(2+\eta)^{-\frac{1}{10}}(1-H)F' u_{x_2}u^2_{k\eta}] - 2[(2+\xi)^{-\frac{1}{10}}(2+\eta)^{-\frac{1}{10}}(1-H)F'](\Box u + 4F'^2u_{\eta\eta}) u^2_{k\eta}
\end{eqnarray*}
\begin{eqnarray*}
&&- 4[(2+\xi)^{-\frac{1}{10}}(2+\eta)^{-\frac{1}{10}}(1-H)F']_{\eta}u_{\xi} u^2_{k\eta}- 4[(2+\xi)^{-\frac{1}{10}}(2+\eta)^{-\frac{1}{10}}(1-H)F']_{\xi} u_{\eta}u^2_{k\eta}\\
&&+ 2[(2+\xi)^{-\frac{1}{10}}(2+\eta)^{-\frac{1}{10}}(1-H)F']_{x_2} u_{x_2}u^2_{k\eta} d\xi d\eta dx_2\\
&&+ \iiint (2+\xi)^{-\frac{1}{10}}(2+\eta)^{-\frac{1}{10}} u_{k\eta} (1-H)J_k d\xi d\eta dx_2 \\
&&- 4\sum\limits_{k_1+k_2=k}\iiint (2+\xi)^{-\frac{1}{10}}(2+\eta)^{-\frac{1}{10}} u_{k\eta}\Gamma^{k_1}(F'^2u_{\eta\eta})\Gamma^{k_2}H d\xi d\eta dx_2.
\end{eqnarray*}
Finally
\begin{eqnarray}\nonumber
&&\iiint (2+\xi)^{-\frac{11}{10}}(2+\eta)^{-\frac{1}{10}} u^2_{k\eta}  d\xi d\eta dx_2 + \iiint (2+\xi)^{-\frac{1}{10}}(2+\eta)^{-\frac{11}{10}} u^2_{kx_2} d\xi d\eta dx_2 \\ \nonumber
&& \lesssim C\varepsilon  + \iiint  [(2+\xi)^{-\frac{1}{10}}(2+\eta)^{-\frac{1}{10}}(1-H)]_{\eta}Q^2_0(u_k,u) - 2(2+\xi)^{-\frac{1}{10}}(2+\eta)^{-\frac{1}{10}}(1-H)Q_0(u_{k},u_{\eta})Q_0(u_k,u)\\ \nonumber
&&+ 2u_{k\eta} Q_0((2+\xi)^{-\frac{1}{10}}(2+\eta)^{-\frac{1}{10}}(1-H), u)Q_0(u_k,u)d\xi d\eta dx_2  \\ \nonumber
&&+ 2 \iiint (2+\xi)^{-\frac{1}{10}}(2+\eta)^{-\frac{1}{10}}(1-H)u_{k\eta}(\Box u + 4F'^2u_{\eta\eta}) Q_0(u_k,u)d\xi d\eta dx_2
\\ \nonumber
&&+\iiint [(2+\xi)^{-\frac{1}{10}}(2+\eta)^{-\frac{1}{10}}(1-H)F'](\Box u + 4F'^2u_{\eta\eta}) u^2_{k\eta} - 4[(2+\xi)^{-\frac{1}{10}}(2+\eta)^{-\frac{1}{10}}(1-H)F']_{\eta}u_{\xi} u^2_{k\eta}\\ \nonumber
&&- 4[(2+\xi)^{-\frac{1}{10}}(2+\eta)^{-\frac{1}{10}}(1-H)F']_{\xi} u_{\eta}u^2_{k\eta}+ 2[(2+\xi)^{-\frac{1}{10}}(2+\eta)^{-\frac{1}{10}}(1-H)F']_{x_2} u_{x_2}u^2_{k\eta} d\xi d\eta dx_2\\ \nonumber
&&+ \iiint (2+\xi)^{-\frac{1}{10}}(2+\eta)^{-\frac{1}{10}} u_{k\eta} (1-H)J_k d\xi d\eta dx_2  \\ \nonumber
&&- 4\sum\limits_{k_1+k_2=k, k_2 \geq 1}\iiint (2+\xi)^{-\frac{1}{10}}(2+\eta)^{-\frac{1}{10}} u_{k\eta}\Gamma^{k_1}(F'^2u_{\eta\eta})\Gamma^{k_2}Hd\xi d\eta dx_2\\  \label{3.9}
&&\doteq C\varepsilon  +  A_1 + A_2+ A_3 + A_4 + A_5.
\end{eqnarray}
In the following we will estimate the above equation terms by terms. Firstly, we will deal with the term $A_1$. Noting Proposition 1 and Proposition 2,
\begin{eqnarray}\nonumber
&&\iiint [(2+\xi)^{-\frac{1}{10}}(2+\eta)^{-\frac{1}{10}}(1-H)]_{\eta}Q^2_0(u_k,u)d\xi d\eta dx_2\\ \nonumber
&&= \iiint-\frac{1}{10}(2+\xi)^{-\frac{1}{10}}(2+\eta)^{-\frac{11}{10}}(1-H)Q^2_0(u_k,u)+ (2+\xi)^{-\frac{1}{10}}(2+\eta)^{-\frac{1}{10}}H_{\eta}Q^2_0(u_k,u)d\xi d\eta dx_2\\ \nonumber
&&\lesssim 
\iiint(2+\xi)^{-\frac{1}{10}}(2+\eta)^{-\frac{1}{10}}|H_{\eta}|[u^2_{\xi}u^2_{k\eta} + u^2_{x_2}u^2_{kx_2} + u^2_{\eta}u^2_{k\xi} ]d\xi d\eta dx_2\\ \nonumber
&& \lesssim 
\iiint(2+\xi)^{-\frac{1}{10}}(2+\eta)^{-\frac{1}{10}}|(Q_{0\eta}+4F'(\xi)u_{\eta\eta})|[u^2_{\xi}u^2_{k\eta} + u^2_{x_2}u^2_{kx_2} + u^2_{\eta}u^2_{k\xi} ]d\xi d\eta dx_2\\ \label{3.10}
&& \lesssim (1+e_s)e_sE_s + \iiint(2+\xi)^{-\frac{1}{10}}(2+\eta)^{-\frac{1}{10}}|(Q_{0\eta}+4F'(\xi)u_{\eta\eta})|[u^2_{\xi}u^2_{k\eta} + u^2_{x_2}u^2_{kx_2} + u^2_{\eta}u^2_{k\xi} ]d\xi d\eta dx_2.
\end{eqnarray}
Noting Proposition 1 and Proposition 2, we can get
\begin{eqnarray*}
|\Gamma u|\lesssim (2+\xi)^{\frac{1}{2}}(2+\eta)^{\frac{1}{2}}\frac{|\Gamma u|}{\sqrt{(2+\xi)(2+\eta)}} \lesssim (2+\xi)^{\frac{1}{2}}(2+\eta)^{\frac{1}{2}}|\Gamma u_{x_2}| \lesssim (2+\xi)^{\frac{1}{4}}(2+\eta)^{\frac{1}{4}}e_s^{\frac{1}{2}}.
\end{eqnarray*}
By Lemma 3.2 and (\ref{3.6}),
\begin{eqnarray}\nonumber
&&\iiint(2+\xi)^{-\frac{1}{10}}(2+\eta)^{-\frac{1}{10}}Q_{0\eta}[u^2_{\xi}u^2_{k\eta} + u^2_{x_2}u^2_{kx_2} + u^2_{\eta}u^2_{k\xi} ]d\xi d\eta dx_2\\ \nonumber
&&\leq \iiint(2+\xi)^{-\frac{11}{10}}(2+\eta)^{-\frac{1}{10}}(|\Gamma u_{\eta}||\nabla u|+ |\Gamma u||\nabla u_{\eta}|)[u^2_{\xi}u^2_{k\eta} + u^2_{x_2}u^2_{kx_2} + u^2_{\eta}u^2_{k\xi} ]d\xi d\eta dx_2\\ \nonumber
&&\lesssim  \iiint(2+\xi)^{-\frac{11}{10}}(2+\eta)^{-\frac{1}{10}} (|\Gamma u_{\eta}||\nabla u|+ |\Gamma u||\nabla u_{\eta}|)u^2_{\xi}u^2_{k\eta} d\xi d\eta dx_2 \\ \nonumber
&& + \iiint(2+\xi)^{-\frac{11}{10}}(2+\eta)^{-\frac{1}{10}}(|\Gamma u_{\eta}||\nabla u|+ |\Gamma u||\nabla u_{\eta}|)[u^2_{x_2}u^2_{kx_2} + u^2_{\eta}u^2_{k\xi} ]d\xi d\eta dx_2\\ \nonumber
&& \lesssim  \iiint(2+\xi)^{-\frac{11}{10}}(2+\eta)^{-\frac{1}{10}} |\Gamma u||\nabla u_{\eta}|u^2_{\xi}u^2_{k\eta}d\xi d\eta dx_2 + e^2_s E_s \\ \nonumber
&& \lesssim  \iiint(2+\xi)^{-\frac{11}{10}}(2+\eta)^{-\frac{1}{10}} (2+\xi)^{\frac{1}{4}}(2+\eta)^{\frac{1}{4}}e^{\frac{1}{2}}_s (2+\eta)^{-\frac{1}{2}}e^{\frac{1}{2}}_s (2+\xi)^{-\frac{3}{4}}(2+\eta)^{\frac{1}{4}}e^{\frac{1}{2}}_s  (2+\xi)^{\delta} \tilde{e}_s u^2_{k\eta}d\xi d\eta dx_2 + e^2_s E_s\\  \label{3.11}
%
&&\lesssim e^{\frac{3}{2}}_s\tilde{e}_s E_s + e^2_s E_s.
 \end{eqnarray}
 Noting the assumption of $H_1)$ and Proposition 2,
 \begin{eqnarray} \nonumber
&&4 \iiint(2+\xi)^{-\frac{1}{10}}(2+\eta)^{-\frac{1}{10}}|F'(\xi)u_{\eta\eta}|[u^2_{\xi}u^2_{k\eta} + u^2_{x_2}u^2_{kx_2} + u^2_{\eta}u^2_{k\xi} ]d\xi d\eta dx_2\\ \nonumber
&&\lesssim  \iiint(2+\xi)^{-\frac{11}{10}}(2+\eta)^{-\frac{1}{10}}|u_{\eta\eta}|[u^2_{\xi}u^2_{k\eta} + u^2_{x_2}u^2_{kx_2} + u^2_{\eta}u^2_{k\xi} ]d\xi d\eta dx_2\\ \nonumber
&&\lesssim \iiint(2+\xi)^{-\frac{13}{5}}(2+\eta)^{-\frac{1}{10}}e^{\frac{3}{2}}_su^2_{k\eta}d\xi d\eta dx_2 + \iiint(2+\xi)^{-\frac{11}{10}}(2+\eta)^{-\frac{1}{10}}u_{\eta\eta}[u^2_{x_2}u^2_{kx_2} + u^2_{\eta}u^2_{k\xi} ]d\xi d\eta dx_2\\ \label{3.12}
&& \lesssim e_s^{\frac{3}{2}} + \iiint(2+\xi)^{-\frac{11}{10}}(2+\eta)^{-\frac{1}{10}}|u_{\eta\eta}|[u^2_{x_2}u^2_{kx_2} + u^2_{\eta}u^2_{k\xi} ]d\xi d\eta dx_2.
\end{eqnarray}
 By Proposition 1 and Proposition 2, we can get
\begin{eqnarray}\nonumber
&&\iiint(2+\xi)^{-\frac{11}{10}}(2+\eta)^{-\frac{1}{10}}|u_{\eta\eta}|u^2_{x_2}u^2_{kx_2}d\xi d\eta dx_2\\ \nonumber
&&\leq \iiint(2+\xi)^{-\frac{37}{20}}(2+\eta)^{-\frac{17}{20}}e_s^{\frac{3}{2}} u^2_{kx_2}d\xi d\eta dx_2\lesssim e_s^{\frac{3}{2}} E_s.
\end{eqnarray}
Combining with Corollary 1, we also have
\begin{eqnarray}\nonumber
&&\iiint(2+\xi)^{-\frac{11}{10}}(2+\eta)^{-\frac{1}{10}}|u_{\eta\eta}| u^2_{\eta}u^2_{k\xi} d\xi d\eta dx_2\\ \label{3.13}
&&\lesssim e_s^{\frac{3}{2}}\iiint(2+\xi)^{-\frac{8}{5}}(2+\eta)^{-\frac{11}{10}} u^2_{k\xi}d\xi d\eta dx_2 \lesssim e_s^{\frac{3}{2}} E_s.
\end{eqnarray}
Therefore, we can get the first term of $A_1$
 \begin{eqnarray}\label{3.14}
&&\iiint [(2+\xi)^{-\frac{1}{10}}(2+\eta)^{-\frac{1}{10}}(1-H)]_{\eta}Q^2_0(u_k,u)d\xi d\eta dx_2 \lesssim (e_s + e_s^{\frac{3}{2}} + e_s^{\frac{3}{2}}\tilde{e}_s + e^2_s) E_s.
\end{eqnarray}
We will continue to estimate the second term of $A_1$. Using Proposition 1 and Proposition 2,
\begin{eqnarray}\nonumber
&&\iiint 2(2+\xi)^{-\frac{1}{10}}(2+\eta)^{-\frac{1}{10}}(1-H)Q_0(u_{k},u_{\eta})Q_0(u_k,u)d\xi d\eta dx_2\\ \nonumber
&& \lesssim (1+ e_s)\iiint (2+\xi)^{-\frac{1}{10}}(2+\eta)^{-\frac{1}{10}}[|u_\xi||u_{\eta\xi}|u_{k\eta}^2+
|u_\xi||u_{k\eta}|(|u_{\eta\eta}||u_{k\xi}|+|u_{\eta x_2}||u_{k x_2}|)]d\xi d\eta dx_2\\ \label{3.15}
&& +(1+ e_s)\iiint (2+\xi)^{-\frac{1}{10}}(2+\eta)^{-\frac{1}{10}}[u_{\eta\xi}^2u_{k\eta}^2+
(u_{\eta\eta}^2+u_{\eta}^2)u_{k\xi}^2+(u_{\eta x_2}^2+u_{x_2}^2)u_{k x_2}^2]d\xi d\eta dx_2.
\end{eqnarray}
Noting Proposition 2 and Corollary 1, we have
\begin{eqnarray}\nonumber
|u_\xi||u_{\eta\xi}|\lesssim (2+\xi)^{-\frac{3}{4}}(2+\eta)^{\frac{1}{4}}(2+\xi)^{-\frac{1}{4}}(2+\eta)^{-\frac{1}{4}}e_s.
\end{eqnarray}
Therefore, we can get
\begin{equation} \label{3.16}
 \iiint (2+\xi)^{-\frac{1}{10}}(2+\eta)^{-\frac{1}{10}}|u_{\xi}u_{\xi\eta}|u^2_{k\eta} d\xi d\eta dx_2\lesssim e_sE_s.
\end{equation}
Noting Proposition 1, we have
\begin{eqnarray}\nonumber
&&\iiint (2+\xi)^{-\frac{1}{10}}(2+\eta)^{-\frac{1}{10}}u^2_{x_2}u^2_{kx_2} d\xi d\eta dx_2 \\ \nonumber
&&\lesssim e_s\iiint (2+\xi)^{-\frac{3}{5}}(2+\eta)^{-\frac{3}{5}}u^2_{kx_2} d\xi d\eta dx_2\\ \label{3.17}
&&\lesssim e_s\iiint [(2+\xi)^{-\frac{1}{10}}(2+\eta)^{-\frac{11}{10}}+(2+\xi)^{-\frac{11}{10}}(2+\eta)^{-\frac{1}{10}}]u^2_{kx_2} d\xi d\eta dx_2\lesssim e_s E_s.
\end{eqnarray}
We note
\begin{eqnarray*}
&&Q_0((2+\xi)^{-\frac{1}{10}}(2+\eta)^{-\frac{1}{10}}(1-H), u)Q_0(u_k,u)\\
&&= \{2[(2+\xi)^{-\frac{1}{10}}(2+\eta)^{-\frac{1}{10}}(1-H)]_{\xi} u_{\eta} + 2[(2+\xi)^{-\frac{1}{10}}(2+\eta)^{-\frac{1}{10}}(1-H)]_{\eta} u_{\xi} \\
&&- [(2+\xi)^{-\frac{1}{10}}(2+\eta)^{-\frac{1}{10}}(1-H)]_{x_2} u_{x_2}\}[2 u_{k\xi}u_{\eta} + 2 u_{k\eta}u_{\xi}- u_{kx_2}u_{x_2}]\\
&& = \{-\frac{1}{5}[(2+\xi)^{-\frac{11}{10}}(2+\eta)^{-\frac{1}{10}}(1-H) + 2(2+\xi)^{-\frac{1}{10}}(2+\eta)^{-\frac{1}{10}}H_{\xi}] u_{\eta} \\
&&+ 2[-\frac{1}{5}(2+\xi)^{-\frac{11}{10}}(2+\eta)^{-\frac{1}{10}}(1-H)] + (2+\xi)^{-\frac{1}{10}}(2+\eta)^{-\frac{1}{10}}H_{\eta}] u_{\xi} \\
&&- [(2+\xi)^{-\frac{1}{10}}(2+\eta)^{-\frac{1}{10}}H_{x_2} u_{x_2}\}[2 u_{k\xi}u_{\eta} + 2 u_{k\eta}u_{\xi}- u_{kx_2}u_{x_2}].
\end{eqnarray*}
Noting proposition 1 and corollary 1, the last term of $A_1$
\begin{eqnarray} \nonumber
&&\iiint 2u_{k\eta} Q_0((2+\xi)^{-\frac{1}{10}}(2+\eta)^{-\frac{1}{10}}(1-H), u)Q_0(u_k,u)d\xi d\eta dx_2\\ \nonumber
&&= \iiint 2u_{k\eta}[2u_{k\eta}u_{\xi} + 2 u_{k\xi}u_{\eta} - u_{kx_2}u_{x_2}]\{-\frac{1}{5}[(2+\xi)^{-\frac{11}{10}}(2+\eta)^{-\frac{1}{10}}(1-H) + 2(2+\xi)^{-\frac{1}{10}}(2+\eta)^{-\frac{1}{10}}H_{\xi}] u_{\eta} \\ \nonumber
&&+ 2[-\frac{1}{5}(2+\xi)^{-\frac{1}{10}}(2+\eta)^{-\frac{11}{10}}(1-H)] + (2+\xi)^{-\frac{1}{10}}(2+\eta)^{-\frac{1}{10}}H_{\eta}] u_{\xi} - [(2+\xi)^{-\frac{1}{10}}(2+\eta)^{-\frac{1}{10}}H_{x_2} u_{x_2}\}\\ \nonumber
&&\lesssim \iiint [u^2_{k\eta}|u_{\xi}| + |u_{k\eta}u_{k\xi}||u_{\eta}| +|u_{k\eta} u_{kx_2}||u_{x_2}|]\{(2+\xi)^{-\frac{1}{10}}(2+\eta)^{-\frac{1}{10}}|H_{\xi}u_{\eta}+H_{\eta}u_{\xi}+H_{x_2} u_{x_2}|\\ \nonumber
&&+[(2+\xi)^{-\frac{11}{10}}(2+\eta)^{-\frac{1}{10}}|u_{\eta}|+(2+\xi)^{-\frac{1}{10}}(2+\eta)^{-\frac{11}{10}}|u_{\xi}|](1-H)]\}\\ \nonumber
&&\lesssim e_s E_s + \iiint  (2+\xi)^{-\frac{11}{10}}(2+\eta)^{-\frac{1}{10}}|1-H||(u_{\eta}+u_{x_2})u_{\eta}| u^2_{k\xi} d\xi d\eta dx_2 \\ \nonumber
&&+ \iiint (2+\xi)^{-\frac{1}{10}}(2+\eta)^{-\frac{11}{10}}(1-H)|(u_{\xi}+u_{\eta}+u_{x_2})||u_{\xi}| u^2_{k\eta} d\xi d\eta dx_2\\ \nonumber
&&+\iiint [u^2_{k\eta}|u_{\xi}| + (u^2_{k\eta}+u^2_{k\xi})|u_{\eta}| +(u^2_{k\eta}+ u^2_{kx_2})|u_{x_2}|]\{(2+\xi)^{-\frac{1}{10}}(2+\eta)^{-\frac{1}{10}}(H_{\xi}u_{\eta}+H_{\eta}u_{\xi}+H_{x_2} u_{x_2})\}d\xi d\eta dx_2\\ \nonumber
&&\lesssim\iiint (2+\xi)^{-\frac{1}{10}}(2+\eta)^{-\frac{1}{10}}[u^2_{k\eta}|(u_{\xi}+u_{\eta}+u_{x_2})| + u^2_{k\xi}|u_{\eta}| + u^2_{kx_2}|u_{x_2}|]|H_{\xi}u_{\eta}+H_{\eta}u_{\xi}+H_{x_2} u_{x_2}|d\xi d\eta dx_2\\ \label{3.18}
&&+ e_s E_s.
\end{eqnarray}
Using the similar method as (\ref{3.5})-(\ref{3.16}) in the above inequality, we can get
\begin{eqnarray}\nonumber
&&\iiint 2u_{k\eta} Q_0((2+\xi)^{-\frac{1}{10}}(2+\eta)^{-\frac{1}{10}}(1-H), u)Q_0(u_k,u)d\xi d\eta dx_2\\  \label{3.19}
&&\lesssim \varepsilon +  e_s E_s + e^{\frac{3}{2}}_s E_s + e^2_s E_s.
\end{eqnarray}

In the following we will estimate the term $A_2$.
\begin{eqnarray} \nonumber
&&\iiint (2+\xi)^{-\frac{1}{10}}(2+\eta)^{-\frac{1}{10}}(1-H)u_{k\eta}(\Box u + 4F'^2u_{\eta\eta}) Q_0(u_k,u)d\xi d\eta dx_2\\ \nonumber
&&\leq  \iiint (2+\xi)^{-\frac{1}{10}}(2+\eta)^{-\frac{1}{10}}|u_{k\eta}||[Q_0(u, Q_0(u, u))+ 4F' Q_0(u, u_{\eta}) + 4 F'' u_{\eta}^2]|[|u_{k\xi}u_{\eta}| + |u_{k\eta}u_{\xi}| + |u_{kx_2}u_{x_2}|]d\xi d\eta dx_2\\ \nonumber
&&+ \iiint (2+\xi)^{-\frac{1}{10}}(2+\eta)^{-\frac{1}{10}}|u_{k\eta}||F'^2 |u_{\eta\eta}|[|u_{k\xi}u_{\eta}| + |u_{k\eta}u_{\xi}| + |u_{kx_2}u_{x_2}|]d\xi d\eta dx_2.
\end{eqnarray}
By the assumption of $F$ and Proposition 2, we first give the estimate
\begin{eqnarray*}
&& \iiint (2+\xi)^{-\frac{1}{10}}(2+\eta)^{-\frac{1}{10}}(1-H)^2|u_{k\eta}||4F' Q_0(u, u_{\eta}) + 4 F'' u_{\eta}^2|[|u_{k\xi}u_{\eta}| + |u_{k\eta}u_{\xi}| + |u_{kx_2}u_{x_2}|]d\xi d\eta dx_2\\
&&\lesssim \iiint (2+\xi)^{-\frac{11}{10}}(2+\eta)^{-\frac{1}{10}}|u_{k\eta}|| Q_0(u, u_{\eta})|[|u_{k\xi}u_{\eta}| + |u_{k\eta}u_{\xi}| + |u_{kx_2}u_{x_2}|]d\xi d\eta dx_2\\
&& + \iiint (2+\xi)^{-\frac{11}{10}}(2+\eta)^{-\frac{1}{10}}|u_{k\eta}|u_{\eta}^2[|u_{k\xi}u_{\eta}| + |u_{k\eta}u_{\xi}| + |u_{kx_2}u_{x_2}|]d\xi d\eta dx_2\\
&&\lesssim   (1+e_s)^2e_s^{\frac{3}{2}}\iiint (2+\xi)^{-1}(2+\eta)^{-1}|u_{k\eta}||u_{k\xi}| + (2+\xi)^{-\frac{11}{10}}(2+\eta)^{-\frac{1}{10}}u^2_{k\eta} + (2+\xi)^{-1}(2+\eta)^{-\frac{17}{20}}|u_{kx_2}u_{k\eta}|]d\xi d\eta dx_2
\\
&&\lesssim  (1+e_s)^2e_s^{\frac{3}{2}}E_s.
\end{eqnarray*}
Using the similar method, we can also get the estimate of last term. Then, noting (\ref{3.3}), we can get
\begin{eqnarray}\nonumber
&&\iiint (2+\xi)^{-\frac{1}{10}}(2+\eta)^{-\frac{1}{10}}(1-H)u_{k\eta}(\Box u + 4F'^2u_{\eta\eta}) Q_0(u_k,u)d\xi d\eta dx_2\\ \nonumber
&&\leq \iiint (2+\xi)^{-\frac{1}{10}}(2+\eta)^{-\frac{1}{10}}|u_{k\eta}|Q_0(u, Q_0(u, u))[|u_{k\xi}u_{\eta}| + |u_{k\eta}u_{\xi}| + |u_{kx_2}u_{x_2}|]d\xi d\eta dx_2 +  (1+e_s)^2e^{\frac{3}{2}}_s E_s \\ \nonumber
&& \lesssim \iiint (2+\xi)^{-\frac{1}{10}}(2+\eta)^{-\frac{1}{10}}|Q_0(u, Q_0(u, u))|[u^2_{k\xi}|u_{\eta}| + u^2_{k\eta}(|u_{\xi}| + |u_{\eta}| + |u_{x_2}|) + u^2_{kx_2}|u_{x_2}|]d\xi d\eta dx_2\\ \label{3.20}
&&+ (1+e_s)^2e^{\frac{3}{2}}_s E_s.
\end{eqnarray}
Furthermore,
\begin{eqnarray*}
&&Q_0(u, Q_0(u, u))[u^2_{k\xi}|u_{\eta}| + u^2_{k\eta}(|u_{\xi}| + |u_{\eta}| + |u_{x_2}|) + u^2_{kx_2}|u_{x_2}|]\\
&& =[2u_{\xi}Q_{0\eta} + 2u_{\eta}Q_{0\xi} - u_{x_2}Q_{0x_2}][u^2_{k\xi}|u_{\eta}| + u^2_{k\eta}(|u_{\xi}| + |u_{\eta}| + |u_{x_2}|) + u^2_{kx_2}|u_{x_2}|].
\end{eqnarray*}
Here we only estimate the first term and the other terms can be obtained using the similar way. By Proposition 1 and Proposition 2, we have
\begin{eqnarray} \nonumber
&&\iiint (2+\xi)^{-\frac{1}{10}}(2+\eta)^{-\frac{1}{10}}(1-H)^2|u_{\xi}||Q_{0\eta}|[u^2_{k\xi}|u_{\eta}| + u^2_{k\eta}(|u_{\xi}| + |u_{\eta}| + |u_{x_2}|) + u^2_{kx_2}|u_{x_2}|]d\xi d\eta dx_2\\ \nonumber
&&\lesssim\iiint (2+\xi)^{-\frac{1}{10}}(2+\eta)^{-\frac{1}{10}}|u_{\xi}|[|u_{\xi\eta}u_{\eta}|+|u_{\xi}u_{\eta\eta}| + |u_{x_2}u_{x_2\eta}|]|u_{\eta}|u^2_{k\xi} d\xi d\eta dx_2\\  \nonumber
&&+ \iiint (2+\xi)^{-\frac{1}{10}}(2+\eta)^{-\frac{1}{10}}|u_{\xi}|[|u_{\xi\eta}u_{\eta}|+|u_{\xi}u_{\eta\eta}| + |u_{x_2}u_{x_2\eta}|](|u_{\xi}| + |u_{\eta}| + |u_{x_2}|) u^2_{k\eta}d\xi d\eta dx_2\\ \nonumber
&&+ C\iiint (2+\xi)^{-\frac{1}{10}}(2+\eta)^{-\frac{1}{10}}|u_{\xi}|[|u_{\xi\eta}u_{\eta}|+|u_{\xi}u_{\eta\eta}| + |u_{x_2}u_{x_2\eta}|] u^2_{kx_2}|u_{x_2}|d\xi d\eta dx_2\\ \nonumber
&&\lesssim \iiint (2+\xi)^{-\frac{1}{10}}(2+\eta)^{-\frac{1}{10}}[(2+\xi)^{-\frac{1}{4}}(2+\eta)^{-1}\tilde{e}_se_s^{\frac{3}{2}}+ (2+\xi)^{-\frac{1}{2}}(2+\eta)^{-\frac{3}{2}}e_s^2]u^2_{k\xi} d\xi d\eta dx_2\\ \nonumber
&&+ e_s^2\iiint (2+\xi)^{-\frac{1}{10}}(2+\eta)^{-\frac{1}{10}}(2+\xi)^{-1}(2+\eta)^{-1} u^2_{k\eta}d\xi d\eta dx_2\\ \nonumber
&&+ e_s^2\iiint (2+\xi)^{-\frac{1}{10}}(2+\eta)^{-\frac{1}{10}}[(2+\xi)^{-\frac{1}{2}}(2+\eta)^{-\frac{3}{2}}+ (2+\xi)^{-\frac{3}{4}}(2+\eta)^{-\frac{5}{4}}] u^2_{kx_2}d\xi d\eta dx_2\\ \label{3.21}
&& \lesssim (\tilde{e}_se_s^{\frac{3}{2}} +  e^{2}_s) E_s.
\end{eqnarray}
Then we can get
\begin{eqnarray} \label{3.22}
A_2 \lesssim e^{\frac{3}{2}}_s E_s + \tilde{e}_se_s^{\frac{3}{2}}E_s + e_s^{2} E_s + e_s^{\frac{5}{2}} E_s.
\end{eqnarray}
Noting the decay of $F'$ and the above process, we will estimate the term $A_3$
\begin{eqnarray}\nonumber
A_3 &=& \iiint (2+\xi)^{-\frac{11}{10}}(2+\eta)^{-\frac{1}{10}}(1-H)(\Box u + 4F'^2u_{\eta\eta}) u^2_{k\eta} + (2+\xi)^{-\frac{11}{10}}[(2+\eta)^{-\frac{1}{10}}(1-H)]_{\eta}u_{\xi} u^2_{k\eta}\\ \nonumber
&+& [(2+\xi)^{-\frac{1}{10}}(2+\eta)^{-\frac{1}{10}}(1-H)F']_{\xi} u_{\eta}u^2_{k\eta}+ (2+\xi)^{-\frac{21}{10}}(2+\eta)^{-\frac{1}{10}}H_{x_2} u_{x_2}u^2_{k\eta} d\xi d\eta dx_2\\ \nonumber
&\lesssim&  e^{\frac{1}{2}}_s E_s + e_s E_s + e^{2}_s E_s + \iiint (2+\xi)^{-\frac{31}{10}}(2+\eta)^{-\frac{1}{10}}(1-H) |u_{\eta}|u^2_{k\eta}d\xi d\eta dx_2\\ \label{3.23}
&\lesssim&  e^{\frac{1}{2}}_s E_s + e_s E_s  + \tilde{e}_se_s^{\frac{3}{2}}E_s  + e^{2}_s E_s.
\end{eqnarray}
For getting the estimation of the term $A_4$, we will first estimate $J_k$. Denote
\begin{eqnarray*}
 &&J_{k1} = \sum\limits_{k_1+k_2+k_3+k_4 = k, k_3 < k, k_4 < k} \Gamma^{k_1}(1-H) Q_0(\Gamma^{k_2}u, Q_0(\Gamma^{k_3}u, \Gamma^{k_4}u))\\
 && J_{k2} = \sum\limits_{0\leq k_1\leq k} \Gamma^{k_1}(1-H)\Gamma^{k-k_1}[F' Q_0(u, u_{\eta})]\\
  && J_{k3} = \sum\limits_{0\leq k_1\leq k} \Gamma^{k_1}(1-H)\Gamma^{k-k_1}[F''u^2_\eta].
\end{eqnarray*}
For $J_{k1}$, when $k_1 \leq [\frac{1+s}{2}]$, we have
\begin{eqnarray}\nonumber
 &&|J_{k1}| = |\sum\limits_{k_1+k_2+k_3+k_4 = k, k_3 < k, k_4 < k} \Gamma^{k_1}(1-H) Q_0(\Gamma^{k_2}u, Q_0(\Gamma^{k_3}u, \Gamma^{k_4}u))|\\  \label{3.24}
 && \leq (1 + e_s)[|(\Gamma^{k_2}u)_{\xi}Q_{0\eta}(\Gamma^{k_3}u, \Gamma^{k_4}u)| + |(\Gamma^{k_2}u)_{\eta}Q_{0\xi}(\Gamma^{k_3}u, \Gamma^{k_4}u)| + |(\Gamma^{k_2}u)_{x_2}Q_{0x_2}(\Gamma^{k_3}u, \Gamma^{k_4}u)|.
\end{eqnarray}
Without loss of generality, we assume $|k_3|,|k_4| \leq [\frac{s+1}{2}]$ and noting (\ref{2.21}), then
\begin{eqnarray}\nonumber
 &&J_{k1} \lesssim(1 + e_s)|(\Gamma^{k_2}u)_{\xi}||Q_{0\eta}(\Gamma^{k_3}u, \Gamma^{k_4}u)| + |(\Gamma^{k_2}u)_{\eta}||Q_{0\xi}(\Gamma^{k_3}u, \Gamma^{k_4}u)| + |(\Gamma^{k_2}u)_{x_2}||Q_{0x_2}(\Gamma^{k_3}u, \Gamma^{k_4}u)|\\ \label{3.25}
 && \doteq J_{k11} + J_{k12} + J_{k13}.
 \end{eqnarray}
 Noting Lemma 3.2, Lemma 3.3 and Proposition 1,
 \begin{eqnarray}\nonumber
 J_{k11} &=& C(1 + e_s)|(\Gamma^{k_2}u)_{\xi}||Q_{0\eta}(\Gamma^{k_3}u, \Gamma^{k_4}u)| \\ \nonumber
 &\lesssim& (1 + e_s)|(\Gamma^{k_2}u)_{\xi}|(2+\xi)^{-1}[|(\Gamma\Gamma^{k_3}u)_{\eta}||\nabla\Gamma^{k_4}u|
 +(2+\xi)^{\frac{1}{2}}(2+\eta)^{\frac{1}{2}}|(\nabla\Gamma^{k_3}u)_{\eta}|
 |\frac{\Gamma\Gamma^{k_4}u}{\sqrt{(2+\xi)(2+\eta)}}|]\\ \nonumber
 &\lesssim& (1 + e_s)|(\Gamma^{k_2}u)_{\xi}|(2+\xi)^{-1}[|(\Gamma\Gamma^{k_3}u)_{\eta}||\nabla\Gamma^{k_4}u|
 +(2+\xi)^{\frac{1}{2}}(2+\eta)^{-\frac{1}{2}}|(2\nabla\Gamma^{k_3}u_{\eta}+ \Gamma_4\nabla\Gamma^{k_3}u - x_2\nabla\Gamma^{k_3}u_{x_2})||\Gamma\Gamma^{k_4}u_{x_2}|]\\ \nonumber
 &\lesssim& (1 + e_s)|(\Gamma^{k_2}u)_{\xi}|(2+\xi)^{-1}[|(\Gamma\Gamma^{k_3}u)_{\eta}||\nabla\Gamma^{k_4}u|
 +(2+\xi)|( \Gamma\nabla\Gamma^{k_3}u)||\Gamma\Gamma^{k_4}u_{x_2}|]\\ \label{3.26}
 &\lesssim& (2+\xi)^{-\frac{1}{2}}(2+\eta)^{-\frac{1}{2}}(1 + e_s)e_s|(\Gamma^{k_2}u)_{\xi}| \lesssim (2+\xi)^{-\frac{1}{2}}(2+\eta)^{-\frac{1}{2}}(1 + e_s)e_s|u_{k\xi}|.
\end{eqnarray}
Meanwhile,
 \begin{eqnarray}\nonumber
 J_{k12} &=& C(1 + e_s)|(\Gamma^{k_2}u)_{\eta}||Q_{0\xi}(\Gamma^{k_3}u, \Gamma^{k_4}u)| \\ \nonumber
 &\lesssim& (1 + e_s)|(\Gamma^{k_2}u)_{\eta}|(2+\xi)^{-1}[|(\Gamma\Gamma^{k_3}u)_{\xi}||\nabla\Gamma^{k_4}u|
 +(2+\xi)^{\frac{1}{2}}(2+\eta)^{\frac{1}{2}}|(\nabla\Gamma^{k_3}u)_{\xi}|
 |\frac{\Gamma\Gamma^{k_4}u}{\sqrt{(2+\xi)(2+\eta)}}|]\\ \nonumber
 &\lesssim&(1 + e_s)|(\Gamma^{k_2}u)_{\eta}|(2+\xi)^{-1}[(2+\xi)^{-\frac{3}{4}}(2+\eta)^{\frac{1}{4}}e^{\frac{1}{2}}_s(2+\xi)^{-\frac{1}{4}}(2+\eta)^{-\frac{1}{4}}e^{\frac{1}{2}}_s
 +(2+\xi)^{\frac{1}{4}}(2+\eta)^{\frac{1}{4}}e^{\frac{1}{2}}_s|\Gamma\Gamma^{k_4}u_{x_2}|]\\ \nonumber
 &\lesssim& (1 + e_s)|(\Gamma^{k_2}u)_{\xi}|(2+\xi)^{-1}[(2+\xi)^{-1}e_s
 +e_s]\\ \label{3.27}
 &\lesssim& (2+\xi)^{-1}(1 + e_s)e_s|(\Gamma^{k_2}u)_{\eta}| \lesssim (2+\xi)^{-1}(1 + e_s)e_s|u_{k\eta}|
\end{eqnarray}
and
 \begin{eqnarray}\nonumber
 J_{k13} &=& C(1 + e_s)|(\Gamma^{k_2}u)_{x_2}||Q_{0x_2}(\Gamma^{k_3}u, \Gamma^{k_4}u)| \\ \nonumber
 &\lesssim& (1 + e_s)|(\Gamma^{k_2}u)_{x_2}|(2+\xi)^{-1}[|(\Gamma\Gamma^{k_3}u)_{x_2}||\nabla\Gamma^{k_4}u|
 +|(\nabla\Gamma^{k_3}u)_{x_2}|(2+\xi)^{\frac{1}{2}}(2+\eta)^{\frac{1}{2}}|\frac{\Gamma\Gamma^{k_4}u}{\sqrt{(2+\xi)(2+\eta)}}|]\\ \nonumber
 &\lesssim& (1 + e_s)|(\Gamma^{k_2}u)_{x_2}|(2+\xi)^{-1}[(2+\xi)^{-\frac{1}{2}}(2+\eta)^{-\frac{1}{2}}e_s
 +(2+\xi)^{\frac{1}{4}}(2+\eta)^{\frac{1}{4}}e^{\frac{1}{2}}_s|\Gamma\Gamma^{k_4}u_{x_2}|]\\ \nonumber
 &\lesssim& (1 + e_s)|(\Gamma^{k_2}u)_{x_2}|(2+\xi)^{-1}[(2+\xi)^{-\frac{1}{2}}(2+\eta)^{-\frac{1}{2}}e_s
 +e_s]\\ \label{3.28}
 &\lesssim& (2+\xi)^{-1}(1 + e_s)e_s|(\Gamma^{k_2}u)_{x_2}|\lesssim (2+\xi)^{-1}(1 + e_s)e_s |u_{kx_2}|.
\end{eqnarray}

For $J_{k1}$, when $k_1 \geq [\frac{1+s}{2}]$, we have
\begin{eqnarray}\nonumber
 &&|J_{k1}| \lesssim|\sum\limits_{k'_1+k_2+k_3+k_4 = k, k'_1 < k_1, k_3 < k, k_4 < k} \Gamma^{k'_1}(Q_0 + F'u_{\eta}) Q_0(\Gamma^{k_2}u, Q_0(\Gamma^{k_3}u, \Gamma^{k_4}u))|\\ \nonumber
 &&\leq \sum\limits_{k'_1+k_2+k_3+k_4 = k, k'_1 < k_1, k_3 < k, k_4 < k} [|\Gamma^{k'_1}Q_0(u,u)| + |\Gamma^{k'_1}(F'u_{\eta})|]| Q_0(\Gamma^{k_2}u, Q_0(\Gamma^{k_3}u, \Gamma^{k_4}u))|.
\end{eqnarray}
In the following we will estimate the above two parts separately.
\begin{eqnarray}\nonumber
&&\sum\limits_{k'_1+k_2+k_3+k_4 = k, k'_1 < k_1, k_3 < k, k_4 < k} |\Gamma^{k'_1}Q_0(u,u)|
| Q_0(\Gamma^{k_2}u, Q_0(\Gamma^{k_3}u, \Gamma^{k_4}u))|\\ \nonumber
&&= \sum\limits_{k_2+k_3+k_4 +k_5+k_6 \leq k} |(\Gamma^{k_5}u_{\xi}\Gamma^{k_6}u_{\eta} + \Gamma^{k_5}u_{x_2}\Gamma^{k_6}u_{x_2})|
| Q_0(\Gamma^{k_2}u, Q_0(\Gamma^{k_3}u, \Gamma^{k_4}u))|\\ \nonumber
&& \leq \sum\limits_{k_2+k_3+k_4 +k_5+k_6 \leq k} |(\Gamma^{k_5}u_{\xi}\Gamma^{k_6}u_{\eta} + \Gamma^{k_5}u_{x_2}\Gamma^{k_6}u_{x_2})|
(2+\xi)^{-1}| \Gamma\Gamma^{k_2}u|| Q_0(\nabla\Gamma^{k_3}u, \Gamma^{k_4}u)| \\ \nonumber
&&+ \sum\limits_{k_2+k_3+k_4 +k_5+k_6 \leq k} |(\Gamma^{k_5}u_{\xi}\Gamma^{k_6}u_{\eta} + \Gamma^{k_5}u_{x_2}\Gamma^{k_6}u_{x_2})|
(2+\xi)^{-1}| \nabla\Gamma^{k_2}u || Q_0(\Gamma\Gamma^{k_3}u, \Gamma^{k_4}u)|\\ \label{3.29}
 &&\doteq B_1 + B_2
\end{eqnarray}
where
\begin{eqnarray}\nonumber
 B_1 &\leq&  (2+\xi)^{-1}|(\Gamma^{k_5}u_{\xi}\Gamma^{k_6}u_{\eta} + \Gamma^{k_5}u_{x_2}\Gamma^{k_6}u_{x_2})|
|\Gamma\Gamma^{k_2}u|[|(\nabla\Gamma^{k_3}u)_{\xi}(\Gamma^{k_4}u)_{\eta}| + |(\Gamma^{k_3}u)_{\eta}(\nabla\Gamma^{k_4}u)_{\xi}| + |(\nabla\Gamma^{k_3}u)_{x_2}(\Gamma^{k_4}u)_{x_2}|] \\ \label{3.30}
 &\doteq& B_{11} + B_{12}+ B_{13}.
\end{eqnarray}
For the case $|k_5|\geq |k_6|$, by Corollary 1 and Lemma 3, we have
\begin{eqnarray}\nonumber
B_{11} &=&  (2+\xi)^{-1}|(\Gamma^{k_5}u_{\xi}\Gamma^{k_6}u_{\eta} + \Gamma^{k_5}u_{x_2}\Gamma^{k_6}u_{x_2})|
|\Gamma\Gamma^{k_2}u||(\nabla\Gamma^{k_3}u)_{\xi}||(\Gamma^{k_4}u)_{\eta}|\\ \nonumber
&\lesssim& (2+\xi)^{-1}(|\Gamma^{k_5}u_{\xi}||\Gamma^{k_6}u_{\eta}| + |\Gamma^{k_5}u_{x_2}||\Gamma^{k_6}u_{x_2}|)
|\Gamma\Gamma^{k_2}u||(\nabla\Gamma^{k_3}u)_{\xi}||(\Gamma^{k_4}u)_{\eta}|\\ \nonumber
&\lesssim& [|\Gamma^{k_5}u_{\xi}|(2+\eta)^{-\frac{1}{2}} + |\Gamma^{k_5}u_{x_2}|(2+\xi)^{-\frac{1}{4}}(2+\eta)^{-\frac{1}{4}}]
|\frac{\Gamma\Gamma^{k_2}u}{\sqrt{(2+\eta)(2+\xi)}}|(2+\xi)^{-\frac{1}{2}}
e^{\frac{3}{2}}_s\\ \nonumber
&\lesssim& [|\Gamma^{k_5}u_{\xi}|(2+\eta)^{-\frac{1}{2}} + |\Gamma^{k_5}u_{x_2}|(2+\xi)^{-\frac{1}{4}}(2+\eta)^{-\frac{1}{4}}]
|(\Gamma\Gamma^{k_2}u)_{x_2}|(2+\xi)^{-\frac{1}{2}}
e^{\frac{3}{2}}_s\\ \nonumber
&\lesssim& [|\Gamma^{k_5}u_{\xi}|(2+\eta)^{-\frac{3}{4}}(2+\xi)^{-\frac{3}{4}} + |\Gamma^{k_5}u_{x_2}|(2+\xi)^{-1}(2+\eta)^{-\frac{1}{2}}]e^{2}_s\\ \nonumber
&\lesssim& e^{2}_s[(2+\eta)^{-\frac{3}{4}}(2+\xi)^{-\frac{3}{4}}|u_{k\xi}| + (2+\xi)^{-1}(2+\eta)^{-\frac{1}{2}}|u_{kx_2}|].
\end{eqnarray}

For the case $|k_5|\leq |k_6|$, by Proposition 2 and Lemma 3, we have
\begin{eqnarray}\nonumber
B_{11} &=&  (2+\xi)^{-1}|(\Gamma^{k_5}u_{\xi}\Gamma^{k_6}u_{\eta} + \Gamma^{k_5}u_{x_2}\Gamma^{k_6}u_{x_2})|
|\Gamma\Gamma^{k_2}u||(\nabla\Gamma^{k_3}u)_{\xi}||(\Gamma^{k_4}u)_{\eta}|\\ \nonumber
&\lesssim& (2+\xi)^{-1}(|\Gamma^{k_5}u_{\xi}||\Gamma^{k_6}u_{\eta}| + |\Gamma^{k_5}u_{x_2}||\Gamma^{k_6}u_{x_2}|)
|\Gamma\Gamma^{k_2}u||(\nabla\Gamma^{k_3}u)_{\xi}||(\Gamma^{k_4}u)_{\eta}|\\ \nonumber
&\lesssim& (2+\xi)^{-1}[(2+\xi)^{-\frac{3}{4}}(2+\eta)^{-\frac{1}{4}}e_s^{\frac{1}{2}}|\Gamma^{k_6}u_{\eta}| + (2+\xi)^{-\frac{1}{4}}(2+\eta)^{-\frac{1}{4}}e_s^{\frac{1}{2}}|\Gamma^{k_6}u_{x_2}|]
|\frac{\Gamma\Gamma^{k_2}u}{\sqrt{(2+\eta)(2+\xi)}}|
e_s\\ \nonumber
&\lesssim& (2+\xi)^{-1}[(2+\xi)^{-\frac{3}{4}}(2+\eta)^{\frac{1}{4}}e_s^{\frac{1}{2}}|\Gamma^{k_6}u_{\eta}| + (2+\xi)^{-\frac{1}{4}}(2+\eta)^{-\frac{1}{4}}e^{\frac{1}{2}}|\Gamma^{k_6}u_{x_2}|]
|(\Gamma\Gamma^{k_2}u)_{x_2}|e_s\\ \nonumber
&\lesssim& (2+\xi)^{-1}[(2+\xi)^{-\frac{3}{4}}(2+\eta)^{\frac{1}{4}}e_s^{\frac{1}{2}}|\Gamma^{k_6}u_{\eta}| + (2+\xi)^{-\frac{1}{4}}(2+\eta)^{-\frac{1}{4}}e^{\frac{1}{2}}|\Gamma^{k_6}u_{x_2}|]
(2+\xi)^{-\frac{1}{4}}(2+\eta)^{-\frac{1}{4}}e^{\frac{3}{2}}_s\\ \nonumber
&\lesssim& (2+\xi)^{-1}[(2+\xi)^{-1}|\Gamma^{k_6}u_{\eta}| + (2+\xi)^{-\frac{1}{2}}(2+\eta)^{-\frac{1}{2}}|\Gamma^{k_6}u_{x_2}|]
e^{2}_s\\ \nonumber
&\lesssim& e^{2}_s[(2+\xi)^{-2}|u_{k\eta}| + (2+\xi)^{-\frac{3}{2}}(2+\eta)^{-\frac{1}{2}}|u_{kx_2}|].
\end{eqnarray}
The estimation of $B_{12}$ can be get using the similar way to $B_{11}$. Next we will give the estimation of $B_{13}$.
For the case $|k_5|\geq |k_6|$, by Corollary 1 and Lemma 3, we have
\begin{eqnarray}\nonumber
B_{13} &=&  (2+\xi)^{-1}|(\Gamma^{k_5}u_{\xi}\Gamma^{k_6}u_{\eta} + \Gamma^{k_5}u_{x_2}\Gamma^{k_6}u_{x_2})|
|\Gamma\Gamma^{k_2}u||(\nabla\Gamma^{k_3}u)_{x_2}||(\Gamma^{k_4}u)_{x_2}|\\ \nonumber
&\lesssim&(2+\xi)^{-1}(|\Gamma^{k_5}u_{\xi}||\Gamma^{k_6}u_{\eta}|+ |\Gamma^{k_5}u_{x_2}||\Gamma^{k_6}u_{x_2}|)
|\Gamma\Gamma^{k_2}u||(\nabla\Gamma^{k_3}u)_{x_2}||(\Gamma^{k_4}u)_{x_2}|\\ \nonumber
&\lesssim&(2+\xi)^{-1}[|\Gamma^{k_5}u_{\xi}|(2+\eta)^{-\frac{1}{2}} + |\Gamma^{k_5}u_{x_2}|(2+\xi)^{-\frac{1}{4}}(2+\eta)^{-\frac{1}{4}}]
|\frac{\Gamma\Gamma^{k_2}u}{\sqrt{(2+\eta)(2+\xi)}}|
e^{\frac{3}{2}}_s\\ \nonumber
&\lesssim&(2+\xi)^{-1}[|\Gamma^{k_5}u_{\xi}|(2+\eta)^{-\frac{1}{2}} + |\Gamma^{k_5}u_{x_2}|(2+\xi)^{-\frac{1}{4}}(2+\eta)^{-\frac{1}{4}}]
|(\Gamma\Gamma^{k_2}u)_{x_2}|e^{\frac{3}{2}}_s\\ \nonumber
&\lesssim&(2+\xi)^{-1}[(2+\eta)^{-\frac{3}{4}}(2+\xi)^{-\frac{1}{4}}|\Gamma^{k_5}u_{\xi}| + |\Gamma^{k_5}u_{x_2}|(2+\xi)^{-\frac{1}{2}}(2+\eta)^{-\frac{1}{2}}]e^{2}_s\\ \nonumber
&\lesssim& e^{2}_s [(2+\eta)^{-\frac{3}{4}}(2+\xi)^{-\frac{5}{4}}|u_{k\xi}| + (2+\xi)^{-\frac{3}{2}}(2+\eta)^{-\frac{1}{2}}|u_{kx_2}|].
\end{eqnarray}

For the case $|k_5|\leq |k_6|$, by Proposition 2 and Lemma 3, we have
\begin{eqnarray}\nonumber
B_{13} &=&  (2+\xi)^{-1}|(\Gamma^{k_5}u_{\xi}\Gamma^{k_6}u_{\eta} + \Gamma^{k_5}u_{x_2}\Gamma^{k_6}u_{x_2})|
|\Gamma\Gamma^{k_2}u||(\nabla\Gamma^{k_3}u)_{x_2}||(\Gamma^{k_4}u)_{x_2}|\\ \nonumber
&\lesssim& (2+\xi)^{-1}[|\Gamma^{k_5}u_{\xi}||\Gamma^{k_6}u_{\eta}| + |\Gamma^{k_5}u_{x_2}||\Gamma^{k_6}u_{x_2}|]
|\Gamma\Gamma^{k_2}u||(\nabla\Gamma^{k_3}u)_{x_2}||(\Gamma^{k_4}u)_{x_2}|\\ \nonumber
&\lesssim& (2+\xi)^{-1}[(2+\xi)^{\frac{1}{2}}(2+\eta)^{\frac{1}{2}}|\Gamma^{k_5}u_{\xi x_2}||\Gamma^{k_6}u_{\eta}| + |\Gamma^{k_5}u_{x_2}||\Gamma^{k_6}u_{x_2}|]
|\Gamma\Gamma^{k_2}u||(\nabla\Gamma^{k_3}u)_{x_2}||(\Gamma^{k_4}u)_{x_2}|\\ \nonumber
&\lesssim& (2+\xi)^{-1}[(2+\xi)^{\frac{1}{4}}(2+\eta)^{\frac{1}{4}}e_s^{\frac{1}{2}}|\Gamma^{k_6}u_{\eta}|  + (2+\xi)^{-\frac{1}{4}}(2+\eta)^{-\frac{1}{4}}e_s^{\frac{1}{2}}|\Gamma^{k_6}u_{x_2}|]
|\frac{\Gamma\Gamma^{k_2}u}{\sqrt{(2+\eta)(2+\xi)}}|
e_s\\ \nonumber
&\lesssim& (2+\xi)^{-1}[(2+\xi)^{\frac{1}{4}}(2+\eta)^{\frac{1}{4}}
e_s^{\frac{1}{2}}|\Gamma^{k_6}u_{\eta}| + (2+\xi)^{-\frac{1}{4}}(2+\eta)^{-\frac{1}{4}}e_s^{\frac{1}{2}}|\Gamma^{k_6}u_{x_2}|]
|(\Gamma\Gamma^{k_2}u)_{x_2}|
e_s\\ \nonumber
&\lesssim& e^{2}_s[(2+\xi)^{-1}|u_{k\eta}| + (2+\xi)^{-\frac{3}{2}}(2+\eta)^{-\frac{1}{2}}|u_{kx_2}|].
\end{eqnarray}
Using the similar procedures, we can get the estimate of $B_2$. Then, we can get the estimate of $J_{k1}$.

In the following, we will give the estimation of $J_{k2}$. When $k_1 \leq [\frac{k}{2}]$, the other case can be get easily.
\begin{eqnarray*}
 J_{k2} &\leq& \sum\limits_{0\leq k_1 + k_2 = k} \Gamma^{k_1}(1-H)\Gamma^{k_2}[F' Q_0(u, u_{\eta})]\\ \nonumber
&\lesssim& (1+e_s) \sum\limits_{0\leq k_3 + k_4 = k_2} \Gamma^{k_3}F' \Gamma^{k_4}Q_0(u, u_{\eta})\\ \nonumber
&\lesssim& (1+e_s) (1+\xi)^{-2}\sum\limits_{k_5 + k_6 = k_4}Q_0(\Gamma^{k_5}u, \Gamma^{k_6}u_{\eta})\\
&\lesssim&(1+e_s) (1+\xi)^{-2}\sum\limits_{k_5 + k_6 = k_4}[|\Gamma^{k_5}u_{\xi}|| \Gamma^{k_6}u_{\eta\eta}| + |\Gamma^{k_5}u_{\eta}|| \Gamma^{k_6}u_{\eta\xi}| + |\Gamma^{k_5}u_{x_2}|| \Gamma^{k_6}u_{\eta x_2}|).
 \end{eqnarray*}
For the case of $|k_5| \geq |k_6| $, noting Corollary 1, we can get
\begin{eqnarray} \label{3.31}
 J_{k2} \lesssim(1+e_s) (2+\xi)^{-1}(2+\eta)^{-\frac{1}{2}}e_s^{\frac{1}{2}}[|\Gamma^{k_5}u_{\xi}| + |\Gamma^{k_5}u_{\eta}| + |\Gamma^{k_5}u_{x_2}|]
\end{eqnarray}
For the case of $|k_5| \leq |k_6| $,  we can get
\begin{eqnarray}\label{3.32}
 J_{k2} \lesssim(1+e_s) (2+\xi)^{-1}e_s^{\frac{1}{2}}[(2+\xi)^{-\frac{1}{4}}(2+\eta)^{-\frac{1}{4}}|\Gamma^{k_6}u_{\eta \eta}| + (2+\eta)^{-\frac{1}{2}}|\Gamma^{k_6}u_{\xi \eta}| + (2+\xi)^{-\frac{1}{4}}(2+\eta)^{-\frac{1}{4}}|\Gamma^{k_6}u_{\eta x_2}|].
\end{eqnarray}
Furthermore, we will estimate the last term $J_{k3}$. When $k_1 < k_2$, it is easily to get
\begin{eqnarray}\nonumber
J_{k3} &=& \sum\limits_{0\leq k_1 + k_2 = k} \Gamma^{k_1}(1-H)\Gamma^{k_2}[F''u^2_\eta]\\  \label{3.33}
&\lesssim &(1+e_s)e_s^{\frac{1}{2}} (2+\xi)^{-\frac{21}{10}} (2+\eta)^{-\frac{1}{2}}\sum\limits_{0\leq  k_2 \leq k}|\Gamma^{k_2}u_{\eta}|\lesssim (1+e_s)e^{\frac{1}{2}}_s (2+\xi)^{-\frac{21}{10}} (2+\eta)^{-\frac{1}{2}} |u_{k\eta}|.
\end{eqnarray}
When $k_1 > k_2$, we have
\begin{eqnarray} \nonumber
J_{k3} &\lesssim& \sum\limits_{0\leq k_1 + k_2 = k} |\Gamma^{k_1}(Q_0 + 4 F'u_\eta)||\Gamma^{k_2}[F''u^2_\eta]|\\ \nonumber
&\lesssim & (2+\xi)^{-2} (2+\eta)^{-1}e_s\sum\limits_{0\leq  k_1 \leq k}[|\Gamma^{k_1}(Q_0(u,u)| + | \Gamma^{k_1} u_\eta|]\\ \nonumber
&\lesssim& (2+\xi)^{-2} (2+\eta)^{-1}e_s[((2+\xi)^{\frac{1}{2}}(2+\eta)^{\frac{1}{2}}|\frac{u_{\xi}}{\sqrt{(2+\xi)(2+\eta)}}|+1)|u_{k\eta}| + | u_\eta||u_{k\xi}| + |u_{x_2}||u_{kx_2}|]\\ \nonumber
&\lesssim& (2+\xi)^{-2} (2+\eta)^{-1}e_s[((2+\xi)^{\frac{1}{2}}(2+\eta)^{\frac{1}{2}}|u_{\xi x_2}|+1)|u_{k\eta}| + | u_\eta||u_{k\xi}| + |u_{x_2}||u_{kx_2}|]\\ \nonumber
&\lesssim& (2+\xi)^{-2} (2+\eta)^{-1}e_s[((2+\xi)^{\frac{1}{4}}(2+\eta)^{\frac{1}{4}}e^{\frac{1}{2}}_s+1)|u_{k\eta}| + | u_\eta||u_{k\xi}| + |u_{x_2}||u_{kx_2}|]\\ \label{3.34}
&\lesssim& (2+\xi)^{-\frac{7}{4}} (2+\eta)^{-\frac{3}{4}}e_s(e^{\frac{1}{2}}_s+1)|u_{k\eta}| + (2+\xi)^{-2} (2+\eta)^{-\frac{3}{2}}e^{\frac{3}{2}}|u_{k\xi}| + (2+\xi)^{-\frac{9}{4}} (2+\eta)^{-\frac{5}{4}}e^{\frac{3}{2}}|u_{kx_2}|.
\end{eqnarray}
Then, the estimation of $A_4$ can be got
\begin{eqnarray}\nonumber
A_4 &=& \iiint (2+\xi)^{-\frac{1}{10}}(2+\eta)^{-\frac{1}{10}} u_{k\eta} (1-H)J_k d\xi d\eta dx_2 \\ \label{3.35}
&\lesssim& e_s^{\frac{1}{2}} (1+ e_s^{\frac{1}{2}} + e_s + e_s^{\frac{3}{2}}+ e_s^{2} + e_s^{\frac{5}{2}})E_s.
\end{eqnarray}
There is only last term $A_5$ to be estimated.
\begin{eqnarray}\nonumber
A_5 &=&  - 4\sum\limits_{k_1+k_2=k}\iiint (2+\xi)^{-\frac{1}{10}}(2+\eta)^{-\frac{1}{10}} u_{k\eta}\Gamma^{k_1}(F'^2u_{\eta\eta})(\Gamma^{k_2}H)d\xi d\eta dx_2\\ \nonumber
&\lesssim& \iiint (2+\xi)^{-\frac{21}{10}}(2+\eta)^{-\frac{1}{10}}[u_{k\eta\eta} H u_{k\eta} + u_{\eta\eta}\Gamma^k(Q_0 + 4F''u_\eta^2)u_{k\eta}]d\xi d\eta d x_2 \\ \nonumber
&\lesssim& \iiint (2+\xi)^{-\frac{21}{10}}(2+\eta)^{-\frac{1}{10}}[H \frac{d}{d\eta}u^2_{k\eta} + |u_{\eta\eta}|[(|Q_0(u,u_k)| + |u_\eta||u_{k\eta}|]|u_{k\eta}|]d\xi d\eta d x_2 \\ \nonumber
&\lesssim&  \iiint (2+\xi)^{-\frac{21}{10}}(2+\eta)^{-\frac{1}{10}}|H_{\eta}| u^2_{k\eta}  + |u_{\eta\eta}|[(|Q_0(u,u_k)| + |u_\eta||u_{k\eta}|]|u_{k\eta}|d\xi d\eta d x_2 \\ \nonumber
&+& \iiint (2+\xi)^{-\frac{21}{10}}(2+\eta)^{-\frac{11}{10}}|H| u^2_{k\eta}d\xi d\eta d x_2 \\ \nonumber
&\lesssim&  \iiint (2+\xi)^{-\frac{21}{10}}(2+\eta)^{-\frac{1}{10}}|Q_0(u,u_{\eta}) + F'u_{\eta\eta}| u^2_{k\eta}  + |u_{\eta\eta}||u_{\eta}u_{k\xi} + u_{\xi}u_{k\eta} +u_{x_2}u_{kx_2}||u_{k\eta}|d\xi d\eta d x_2 + e_sE_s + e_s^{\frac{1}{2}}E_s \\ \nonumber
&\lesssim&  \iiint (2+\xi)^{-\frac{21}{10}}(2+\eta)^{-\frac{1}{10}}|u_{\eta}u_{\eta\xi} + u_{\xi}u_{\eta\eta} +u_{x_2}u_{\eta x_2}| u^2_{k\eta} d\xi d\eta d x_2 +  \tilde{e}_se_s^{\frac{1}{2}}E_s + e_sE_s + e_s^{\frac{1}{2}}E_s \\ \nonumber
&\lesssim&  \tilde{e}_se_s^{\frac{1}{2}}E_s + e_sE_s + e_s^{\frac{1}{2}}E_s.
\end{eqnarray}
Then, we get the all estimates of $(\ref{3.9})$.

Multiplying $(2+\xi)^{-\frac{1}{10}}(2+\eta)^{-\frac{1}{10}} u_{k\xi}e^{B(\xi)}$ to the system (\ref{3.4}) and integrating it about $\xi$, $x_2$ and $\eta$, we can get the left hand side parts of system
\begin{eqnarray} \nonumber
&& 4\iiint (2+\xi)^{-\frac{1}{10}}(2+\eta)^{-\frac{1}{10}}\partial_{\xi\eta}u_k u_{k\xi}d\xi d\eta dx_2  - \iiint (2+\xi)^{-\frac{1}{10}}(2+\eta)^{-\frac{1}{10}}\partial_{x_2x_2}u_k u_{k\xi}d\eta dx_2\\ \nonumber
&& + 4\iiint (2+\xi)^{-\frac{1}{10}}(2+\eta)^{-\frac{1}{10}}\Gamma^k{(F'^2 u_{\eta\eta})} u_{k\xi}d\xi d\eta dx_2\\ \nonumber
&& = 2 \iiint (2+\xi)^{-\frac{1}{10}}(2+\eta)^{-\frac{1}{10}} \frac{d}{d\eta}u^2_{k\xi} d \xi d\eta dx_2 + \iiint (2+\xi)^{-\frac{1}{10}}(2+\eta)^{-\frac{1}{10}}[\frac{d}{dx_2}(u_{k\xi}u_{kx_2})- \frac{1}{2}\frac{d}{d\xi}u^2_{kx_2}]d\xi d\eta dx_2 \\ \nonumber
&& = -\frac{1}{2} \iiint \frac{d}{d\xi} [(2+\xi)^{-\frac{1}{10}}(2+\eta)^{-\frac{1}{10}}u^2_{kx_2}] d\xi d\eta dx_2  + \frac{1}{20} \iiint  [(2+\xi)^{-\frac{11}{10}}(2+\eta)^{-\frac{1}{10}}u^2_{kx_2}] d\xi d\eta dx_2 \\ \nonumber
&& +  \frac{1}{5} \iiint (2+\xi)^{-\frac{1}{10}}(2+\eta)^{-\frac{11}{10}}u^2_{k\xi}d\xi d\eta dx_2  + \iiint 2  \frac{d}{d\eta }[(2+\xi)^{-\frac{1}{10}}(2+\eta)^{-\frac{1}{10}} u^2_{k\xi}]d\xi d\eta dx_2\\ \label{3.36}
&& + \iiint \frac{d}{dx_2}[(2+\xi)^{-\frac{1}{10}}(2+\eta)^{-\frac{1}{10}}u_{k\xi}u_{kx_2}] d\xi d\eta dx_2.
\end{eqnarray}
The right hand side parts can be obtained as follows
\begin{eqnarray*}
&&\iiint (2+\xi)^{-\frac{1}{10}}(2+\eta)^{-\frac{1}{10}}u_{k\xi} (1-H)[Q_0(u, 2Q_0(u, u_k))+ 4F' Q_0(u, u_{k\eta}) + J_k + \Gamma^k(4F'^2 u_{\eta\eta} H)] d\xi d\eta dx_2\\
&&\doteq II_1 + II_2 + II_3 + II_4.
\end{eqnarray*}
Then, we can get
\begin{eqnarray*}
\uppercase\expandafter{\romannumeral2}_1 &=& \iiint (2+\xi)^{-\frac{1}{10}}(2+\eta)^{-\frac{1}{10}}u_{k\xi}(1-H)[4u_{\xi}Q_{0\eta}(u, u_k) + 4u_{\eta}Q_{0\xi}(u, u_k) - 2u_{x_2}Q_{0x_2}(u, u_k) d\eta dx_2\\
&=&\iiint 4\{\frac{d}{d\eta}[(2+\xi)^{-\frac{1}{10}}(2+\eta)^{-\frac{1}{10}}(1-H)u_{\xi}u_{k\xi} Q_0 ] - (2+\xi)^{-\frac{1}{10}}(2+\eta)^{-\frac{1}{10}}(1-H)u_{\xi\eta}u_{k\xi} Q_0\\ &-& [(2+\xi)^{-\frac{1}{10}}(2+\eta)^{-\frac{1}{10}}(1-H)u_{k\xi}]_{\eta} u_{\xi} Q_0\}
+ 4\{\frac{d}{d\xi}[(2+\xi)^{-\frac{1}{10}}(2+\eta)^{-\frac{1}{10}}(1-H)u_{\eta}u_{k\xi} Q_0]  \\
&-& (2+\xi)^{-\frac{1}{10}}(2+\eta)^{-\frac{1}{10}}(1-H)u_{\xi\eta}u_{k\xi} Q_0 - [(2+\xi)^{-\frac{1}{10}}(2+\eta)^{-\frac{1}{10}}(1-H)u_{k\xi}]_{\xi}u_{\eta} Q_0\} \\
&-& 2 \{\frac{d}{d x_2}[(2+\xi)^{-\frac{1}{10}}(2+\eta)^{-\frac{1}{10}}(1-H)u_{x_2}u_{k\xi} Q_0]  - (2+\xi)^{-\frac{1}{10}}(2+\eta)^{-\frac{1}{10}}(1-H)u_{x_2x_2}u_{k\xi} Q_0 \\
&-& [(2+\xi)^{-\frac{1}{10}}(2+\eta)^{-\frac{1}{10}}(1-H)u_{k\xi}]_{x_2}u_{x_2} Q_0\} d\xi d\eta dx_2\\
&=&  \iiint 4 \frac{d}{d\eta}[(2+\xi)^{-\frac{1}{10}}(2+\eta)^{-\frac{1}{10}}(1-H)u_{\xi}u_{k\xi} Q_0 ] + 4 \frac{d}{d\xi}[(2+\xi)^{-\frac{1}{10}}(2+\eta)^{-\frac{1}{10}}(1-H)u_{\eta}u_{k\xi} Q_0] \\
&-& 2 \frac{d}{d x_2}[(2+\xi)^{-\frac{1}{10}}(2+\eta)^{-\frac{1}{10}}(1-H)u_{x_2}u_{k\xi} Q_0]d\xi d\eta dx_2\\
&-& 2 \iiint Q_0((2+\xi)^{-\frac{1}{10}}(2+\eta)^{-\frac{1}{10}}(1-H)u_{k\xi}, u)Q_0(u_k,u)d\xi d\eta dx_2 \\
&-&  2\iiint (2+\xi)^{-\frac{1}{10}}(2+\eta)^{-\frac{1}{10}}(1-H)u_{k\xi}(\Box u + 4F'^2u_{\eta\eta}) Q_0(u_k,u)d\xi d\eta dx_2
\end{eqnarray*}
and
\begin{eqnarray*}
\uppercase\expandafter{\romannumeral2}_2 &=& 4\iiint  (2+\xi)^{-\frac{1}{10}}(2+\eta)^{-\frac{1}{10}}(1-H)F' u_{k\xi}[2u_{\xi}u_{k\eta\eta} + 2u_{\eta}u_{k\eta\xi} - u_{x_2}u_{k\eta x_2}] d\xi d\eta dx_2\\ \nonumber
&=& 8\iiint (2+\xi)^{-\frac{1}{10}}(2+\eta)^{-\frac{1}{10}}(1-H)F' u_{\xi} [\frac{d}{d \eta} (u_{k\xi}u_{k\eta})
- \frac{1}{2}\frac{d}{d \xi}u^2_{k\eta}] d\xi d\eta dx_2 \\
&+& 4 \iiint (2+\xi)^{-\frac{1}{10}}(2+\eta)^{-\frac{1}{10}}(1-H)F'u_{\eta} \frac{d}{d \eta} u^2_{k\xi}d\xi d\eta dx_2 \\ \nonumber
&-& 4\iiint \frac{d}{d\eta} [(2+\xi)^{-\frac{1}{10}}(2+\eta)^{-\frac{1}{10}} (1-H) F' u_{x_2}u_{k\xi}u_{kx_2} ] d\xi d\eta dx_2\\
&+& 4\iiint [(2+\xi)^{-\frac{1}{10}}(2+\eta)^{-\frac{1}{10}} (1-H) F' u_{x_2}u_{k\xi}]_{\eta} u_{kx_2} d\xi d\eta dx_2\\
&=&\iiint 8\frac{d}{d \eta} [(2+\xi)^{-\frac{1}{10}}(2+\eta)^{-\frac{1}{10}}(1-H)F' u_{\xi} u_{k\xi}u_{k\eta}]
- 4\frac{d}{d \xi}[(2+\xi)^{-\frac{1}{10}}(2+\eta)^{-\frac{1}{10}}(1-H)F' u_{\xi} u^2_{k\eta}] d\xi d\eta dx_2 \\
&-&8\iiint [(2+\xi)^{-\frac{1}{10}}(2+\eta)^{-\frac{1}{10}}(1-H)F' u_{\xi}]_{\eta} u_{k\xi}u_{k\eta}d\xi d\eta dx_2 +  \iiint 4[(2+\xi)^{-\frac{1}{10}}(2+\eta)^{-\frac{1}{10}}(1-H)F' u_{\xi}]_{\xi} u^2_{k\eta} d\xi d\eta dx_2 \\ \nonumber
&+& 4 \iiint \frac{d}{d \eta}[(2+\xi)^{-\frac{1}{10}}(2+\eta)^{-\frac{1}{10}}(1-H)F' u_{\eta} u^2_{k\xi}] d\xi d\eta dx_2 - 4 \iiint [(2+\xi)^{-\frac{1}{10}}(2+\eta)^{-\frac{1}{10}}(1-H)u_{\eta}]_{\eta} u^2_{k\xi}d\xi d\eta dx_2  \\ \nonumber
&-& 4\iiint \frac{d}{d\eta} [(2+\xi)^{-\frac{1}{10}}(2+\eta)^{-\frac{1}{10}} (1-H) F' u_{x_2}u_{k\xi}u_{kx_2} ] d\xi d\eta dx_2\\
& +& 4\iiint  [(2+\xi)^{-\frac{1}{10}}(2+\eta)^{-\frac{1}{10}} (1-H) F' u_{x_2}]_{\eta} u_{k\xi} u_{kx_2} d\xi d\eta dx_2 \\
&+& 4\iiint [(2+\xi)^{-\frac{1}{10}}(2+\eta)^{-\frac{1}{10}} (1-H) F' u_{x_2}] u_{k\xi\eta} u_{kx_2} d\xi d\eta dx_2.
\end{eqnarray*}

Noting (\ref{2.21}),
$$4u_{k\xi\eta} = \Box u_k + u_{kx_2 x_2} + 4\Gamma^k(F'^2u_{\eta\eta}).$$
Then, the last term of the above equation can be rewritten as
\begin{eqnarray*}
 &&4\iiint [(2+\xi)^{-\frac{1}{10}}(2+\eta)^{-\frac{1}{10}} (1-H) F' u_{x_2}] u_{k\xi\eta} u_{kx_2} d\xi d\eta dx_2 \\
 && =  \iiint [(2+\xi)^{-\frac{1}{10}}(2+\eta)^{-\frac{1}{10}} (1-H) F' u_{x_2}] \Box u_{k} u_{kx_2} d\xi d\eta dx_2 + \iiint [(2+\xi)^{-\frac{1}{10}}(2+\eta)^{-\frac{1}{10}} (1-H) F' u_{x_2}] u_{kx_2x_2} u_{kx_2}d\xi d\eta\\
 &&+ 4\iiint [(2+\xi)^{-\frac{1}{10}}(2+\eta)^{-\frac{1}{10}} (1-H) F' u_{x_2}] \Gamma^k(F'^2u_{\eta\eta}) u_{kx_2}d\xi d\eta dx_2\\
 && =  \iiint [(2+\xi)^{-\frac{1}{10}}(2+\eta)^{-\frac{1}{10}} (1-H) F' u_{x_2}] [\Box u_{k} + 4\Gamma^k(F'^2u_{\eta\eta})] u_{kx_2}d\xi d\eta dx_2\\
 && + \frac{1}{2}\iiint (2+\xi)^{-\frac{1}{10}}(2+\eta)^{-\frac{1}{10}}\{ \frac{d}{dx_2}[(1-H) F' u_{x_2} u^2_{kx_2}] - [(1-H) F' u_{x_2}]_{x_2} u^2_{kx_2} \}d\xi d\eta dx_2.
\end{eqnarray*}

Therefore,
\begin{eqnarray}\nonumber
&&  \frac{1}{20} \iiint (2+\xi)^{-\frac{11}{10}}(2+\eta)^{-\frac{1}{10}}u^2_{kx_2} d\xi d\eta dx_2 + \frac{1}{5} \iiint (2+\xi)^{-\frac{1}{10}}(2+\eta)^{-\frac{11}{10}}u^2_{k\xi} d\xi d\eta dx_2 \\ \nonumber
&=&  \frac{1}{2} \iiint \frac{d}{d\xi} [(2+\xi)^{-\frac{1}{10}}(2+\eta)^{-\frac{1}{10}}u^2_{kx_2}] d\xi d\eta dx_2 - 2  \frac{d}{d\eta }[(2+\xi)^{-\frac{1}{10}}(2+\eta)^{-\frac{1}{10}} u^2_{k\xi}]d\xi d\eta dx_2\\ \nonumber
&-&  \iiint \frac{d}{dx_2}[(2+\xi)^{-\frac{1}{10}}(2+\eta)^{-\frac{1}{10}}u_{k\xi}u_{kx_2}] d\xi d\eta dx_2\\ \nonumber
&+&\iiint 4 \frac{d}{d\eta}[(2+\xi)^{-\frac{1}{10}}(2+\eta)^{-\frac{1}{10}}(1-H)u_{\xi}u_{k\xi} Q_0 ] + 4 \frac{d}{d\xi}[(2+\xi)^{-\frac{1}{10}}(2+\eta)^{-\frac{1}{10}}(1-H)u_{\eta}u_{k\xi} Q_0] \\ \nonumber
&-& 2 \frac{d}{d x_2}[(2+\xi)^{-\frac{1}{10}}(2+\eta)^{-\frac{1}{10}}(1-H)u_{x_2}u_{k\xi} Q_0]d\xi d\eta dx_2\\ \nonumber
&-& 2 \iiint Q_0((2+\xi)^{-\frac{1}{10}}(2+\eta)^{-\frac{1}{10}}(1-H)u_{k\xi}, u)Q_0(u_k,u)d\xi d\eta dx_2 \\ \nonumber
&-&  2\iiint (2+\xi)^{-\frac{1}{10}}(2+\eta)^{-\frac{1}{10}}(1-H)u_{k\xi}(\Box u + 4F'^2u_{\eta\eta}) Q_0(u_k,u)d\xi d\eta dx_2\\ \nonumber
&+& \iiint 8\frac{d}{d \eta} [(2+\xi)^{-\frac{1}{10}}(2+\eta)^{-\frac{1}{10}}(1-H)F' u_{\xi} u_{k\xi}u_{k\eta}]
- 4\frac{d}{d \xi}[(2+\xi)^{-\frac{1}{10}}(2+\eta)^{-\frac{1}{10}}(1-H)F' u_{\xi} u^2_{k\eta}] d\xi d\eta dx_2 \\ \nonumber
&-&8\iiint [(2+\xi)^{-\frac{1}{10}}(2+\eta)^{-\frac{1}{10}}(1-H)F' u_{\xi}]_{\eta} u_{k\xi}u_{k\eta}d\xi d\eta dx_2 +  \iiint 4[(2+\xi)^{-\frac{1}{10}}(2+\eta)^{-\frac{1}{10}}(1-H)F' u_{\xi}]_{\xi} u^2_{k\eta} d\xi d\eta dx_2 \\ \nonumber
&+& 4 \iiint \frac{d}{d \eta}[(2+\xi)^{-\frac{1}{10}}(2+\eta)^{-\frac{1}{10}}(1-H)F' u_{\eta} u^2_{k\xi}] d\xi d\eta dx_2 - 4 \iiint [(2+\xi)^{-\frac{1}{10}}(2+\eta)^{-\frac{1}{10}}(1-H)F'u_{\eta}]_{\eta} u^2_{k\xi}d\xi d\eta dx_2  \\ \nonumber
&-& 4\iiint \frac{d}{d\eta} [(2+\xi)^{-\frac{1}{10}}(2+\eta)^{-\frac{1}{10}} (1-H) F' u_{x_2}u_{k\xi}u_{kx_2} ] d\xi d\eta dx_2\\ \nonumber
& +& 4\iiint  [(2+\xi)^{-\frac{1}{10}}(2+\eta)^{-\frac{1}{10}} (1-H) F' u_{x_2}]_{\eta} u_{k\xi} u_{kx_2} d\xi d\eta dx_2 \\ \nonumber
&+& \iiint [(2+\xi)^{-\frac{1}{10}}(2+\eta)^{-\frac{1}{10}} (1-H) F' u_{x_2}] [\Box u_{k} + 4\Gamma^k(F'^2u_{\eta\eta})]u_{kx_2}d\xi d\eta dx_2 \\ \nonumber
&+& \frac{1}{2}\iiint (2+\xi)^{-\frac{1}{10}}(2+\eta)^{-\frac{1}{10}}\{ \frac{d}{dx_2}[(1-H) F' u_{x_2} u^2_{kx_2}] - [(1-H) F' u_{x_2}]_{x_2} u^2_{kx_2} \}d\xi d\eta dx_2\\ \nonumber
&+& \iiint (2+\xi)^{-\frac{1}{10}}(2+\eta)^{-\frac{1}{10}}u_{k\xi} (1-H)J_k d\xi d\eta dx_2 \\ \label{3.37}
&+&\iiint (2+\xi)^{-\frac{1}{10}}(2+\eta)^{-\frac{1}{10}}u_{k\xi} (1-H)4\Gamma^k(F'^2u_{\eta\eta}H)d\xi d\eta dx_2
\end{eqnarray}
Noting
\begin{eqnarray*} \nonumber
&&Q_0((2+\xi)^{-\frac{1}{10}}(2+\eta)^{-\frac{1}{10}}(1-H)u_{k\xi}, u) \\
&&= (2+\xi)^{-\frac{1}{10}}(2+\eta)^{-\frac{1}{10}}(1-H)Q_0(u_{k\xi}, u) + u_{k\xi}Q_0((2+\xi)^{-\frac{1}{10}}(2+\eta)^{-\frac{1}{10}}(1-H), u),
\end{eqnarray*}
then, we have
 \begin{eqnarray*}\nonumber
&&2Q_0((2+\xi)^{-\frac{1}{10}}(2+\eta)^{-\frac{1}{10}}(1-H)u_{k\xi}, u)Q_0(u, u_k) \\
&&= (2+\xi)^{-\frac{1}{10}}(2+\eta)^{-\frac{1}{10}}(1-H)\frac{d}{d\xi} Q^2_0(u, u_k) - 2(2+\xi)^{-\frac{1}{10}}(2+\eta)^{-\frac{1}{10}}(1-H)Q_0(u_k, u)Q_0(u_{\xi}, u_k) \\ \nonumber
&&+ 2 u_{k\xi} Q_0((2+\xi)^{-\frac{1}{10}}(2+\eta)^{-\frac{1}{10}}(1-H),u)Q(u_k, u)
 \end{eqnarray*}
  Furthermore, noting
 \begin{eqnarray*}
 Q_0(u, u_{k\eta}) &=& 2u_{\xi}u_{k\eta\eta} + 2 u_{\eta}u_{k\xi\eta} - u_{x_2}u_{k\eta x_2}\\ \nonumber
  &=& 2u_{\xi}u_{k\eta\eta} + \frac{1}{2} u_{\eta} [\Box u_k + u_{kx_2x_2} + 4\Gamma^k(F'^2u_{\eta\eta})] - u_{x_2}u_{k\eta x_2}
 \end{eqnarray*}
then we can rewrite $\Box u_k$ as follows
 \begin{eqnarray} \label{3.38}
 \Box u_k &=& (1-H)[Q_0(u, 2Q_0(u, u_k))+ 4F' Q_0(u, u_{k\eta}) + J_k - 4\Gamma^k(F'^2 u_{\eta\eta})] \\   \nonumber
 &=&(1-H) [Q_0(u,Q_0(u,u_k))+4F'(2u_{\xi}u_{k\eta\eta} + \frac{1}{2}u_{\eta}(\Box u_k + u_{k x_2 x_2} + 4 \Gamma^k(F'^2u_{\eta\eta})) - u_{x_2}u_{k\eta x_2})]\\ \nonumber
 &=& \frac{1-H}{1- 2(1-H)F'u_{\eta}} [Q_0(u, 2Q_0(u, u_k))+ 8F'u_{\xi}u_{k\eta \eta} + 2F'u_{\eta}\Gamma^k(F'^2u_{\eta\eta}) + 2F'u_{\eta}u_{kx_2x_2} - 4F'u_{x_2}u_{k\eta x_2} + J_k - 4\Gamma^k(F'^2 u_{\eta\eta})]
 \end{eqnarray}
 Substituting the above equations into (\ref{3.37}), we can get
\begin{eqnarray*}
&& \frac{1}{20} \iiint (2+\xi)^{-\frac{11}{10}}(2+\eta)^{-\frac{1}{10}}u^2_{kx_2} d\xi d\eta dx_2 + \frac{1}{5} \iiint (2+\xi)^{-\frac{1}{10}}(2+\eta)^{-\frac{11}{10}}u^2_{k\xi} d\xi d\eta dx_2 \\ \nonumber
&=&   \frac{1}{2} \iiint \frac{d}{d\xi} [(2+\xi)^{-\frac{1}{10}}(2+\eta)^{-\frac{1}{10}}u^2_{kx_2}] d\xi d\eta dx_2 - 2  \frac{d}{d\eta }[(2+\xi)^{-\frac{1}{10}}(2+\eta)^{-\frac{1}{10}} u^2_{k\xi}]d\xi d\eta dx_2\\
&-&  \iiint \frac{d}{dx_2}[(2+\xi)^{-\frac{1}{10}}(2+\eta)^{-\frac{1}{10}}u_{k\xi}u_{kx_2}] d\xi d\eta dx_2\\
&+&\iiint 4 \frac{d}{d\eta}[(2+\xi)^{-\frac{1}{10}}(2+\eta)^{-\frac{1}{10}}(1-H)u_{\xi}u_{k\xi} Q_0 ] + 4 \frac{d}{d\xi}[(2+\xi)^{-\frac{1}{10}}(2+\eta)^{-\frac{1}{10}}(1-H)u_{\eta}u_{k\xi} Q_0] \\
&-& 2 \frac{d}{d x_2}[(2+\xi)^{-\frac{1}{10}}(2+\eta)^{-\frac{1}{10}}(1-H)u_{x_2}u_{k\xi} Q_0]d\xi d\eta dx_2\\
&-& 2 \iiint  (2+\xi)^{-\frac{1}{10}}(2+\eta)^{-\frac{1}{10}}(1-H)\frac{d}{d\xi} Q^2_0(u, u_k) - 2(2+\xi)^{-\frac{1}{10}}(2+\eta)^{-\frac{1}{10}}(1-H)Q_0(u_k, u)Q_0(u_{\xi}, u_k) \\
&&+ 2 u_{k\xi} Q_0((2+\xi)^{-\frac{1}{10}}(2+\eta)^{-\frac{1}{10}}(1-H),u)Q(u_k, u)d\xi d\eta dx_2 \\
&-& 2 \iiint (2+\xi)^{-\frac{1}{10}}(2+\eta)^{-\frac{1}{10}}(1-H)u_{k\xi}(\Box u + 4F'^2u_{\eta\eta}) Q_0(u_k,u)d\xi d\eta dx_2\\
&+& \iiint 8\frac{d}{d \eta} [(2+\xi)^{-\frac{1}{10}}(2+\eta)^{-\frac{1}{10}}(1-H)F' u_{\xi} u_{k\xi}u_{k\eta}]
- 4\frac{d}{d \xi}[(2+\xi)^{-\frac{1}{10}}(2+\eta)^{-\frac{1}{10}}(1-H)F' u_{\xi} u^2_{k\eta}] d\xi d\eta dx_2 \\
&-&8\iiint [(2+\xi)^{-\frac{1}{10}}(2+\eta)^{-\frac{1}{10}}(1-H)F' u_{\xi}]_{\eta} u_{k\xi}u_{k\eta}d\xi d\eta dx_2 +  \iiint 4[(2+\xi)^{-\frac{1}{10}}(2+\eta)^{-\frac{1}{10}}(1-H)F' u_{\xi}]_{\xi} u^2_{k\eta} d\xi d\eta dx_2 \\ \nonumber
&+& 4 \iiint \frac{d}{d \eta}[(2+\xi)^{-\frac{1}{10}}(2+\eta)^{-\frac{1}{10}}(1-H)F' u_{\eta} u^2_{k\xi}] d\xi d\eta dx_2 - 4 \iiint [(2+\xi)^{-\frac{1}{10}}(2+\eta)^{-\frac{1}{10}}(1-H)F'u_{\eta}]_{\eta} u^2_{k\xi}d\xi d\eta dx_2  
\end{eqnarray*}
\begin{eqnarray*}
&-& 4\iiint \frac{d}{d\eta} [(2+\xi)^{-\frac{1}{10}}(2+\eta)^{-\frac{1}{10}} (1-H) F' u_{x_2}u_{k\xi}u_{kx_2} ] d\xi d\eta dx_2\\
& +& 4\iiint  [(2+\xi)^{-\frac{1}{10}}(2+\eta)^{-\frac{1}{10}} (1-H) F' u_{x_2}]_{\eta} u_{k\xi} u_{kx_2} d\xi d\eta dx_2 \\
&+& \iiint  P [Q_0(u, 2Q_0(u, u_k))+ 8F'u_{\xi}u_{k\eta \eta} + 2F'u_{\eta}u_{kx_2x_2} - 4F'u_{x_2}u_{k\eta x_2} + J_k + 8F'u_{\eta}\Gamma^k(F'^2u_{\eta\eta}) - 4\Gamma^k(F'^2u_{\eta\eta})] u_{kx_2}d\xi d\eta dx_2 \\
&+& \frac{1}{2}\iiint (2+\xi)^{-\frac{1}{10}}(2+\eta)^{-\frac{1}{10}}\{ \frac{d}{dx_2}[(1-H) F' u_{x_2} u^2_{kx_2}] - [(1-H) F' u_{x_2}]_{x_2} u^2_{kx_2} \}d\xi d\eta dx_2\\
&+& \iiint (2+\xi)^{-\frac{1}{10}}(2+\eta)^{-\frac{1}{10}}u_{k\xi} (1-H)J_k d\xi d\eta dx_2\\
&+&\iiint (2+\xi)^{-\frac{1}{10}}(2+\eta)^{-\frac{1}{10}}u_{k\xi} (1-H)4\Gamma^k(F'^2u_{\eta\eta}H)d\xi d\eta dx_2
\end{eqnarray*}
where
\begin{eqnarray} \label{3.39}
P(\xi, \eta, x_2) = (2+\xi)^{-\frac{1}{10}}(2+\eta)^{-\frac{1}{10}} \frac{(1-H)^2 F' u_{x_2}}{1- 2(1-H)F'u_{\eta}}.
\end{eqnarray}
Noting
 \begin{eqnarray}\nonumber
&&\iiint (2+\xi)^{-\frac{1}{10}}(2+\eta)^{-\frac{1}{10}}(1-H)\frac{d}{d\xi} Q^2_0(u, u_k)d\xi d\eta dx_2\\  \nonumber
 && = \iiint \frac{d}{d\xi} [(2+\xi)^{-\frac{1}{10}}(2+\eta)^{-\frac{1}{10}}(1-H) Q^2_0(u, u_k)]d\xi d\eta dx_2 \\ \label{3.40}
 &&- \iiint [(2+\xi)^{-\frac{1}{10}}(2+\eta)^{-\frac{1}{10}}(1-H)]_{\xi} Q^2_0(u, u_k)d\xi d\eta dx_2
\end{eqnarray}
Then
\begin{eqnarray*}
&& \frac{1}{20} \iiint (2+\xi)^{-\frac{11}{10}}(2+\eta)^{-\frac{1}{10}}u^2_{kx_2} d\xi d\eta dx_2 + \frac{1}{5} \iiint (2+\xi)^{-\frac{1}{10}}(2+\eta)^{-\frac{11}{10}}u^2_{k\xi} d\xi d\eta dx_2 \\ \nonumber
&=&    \frac{1}{2} \iiint \frac{d}{d\xi} [(2+\xi)^{-\frac{1}{10}}(2+\eta)^{-\frac{1}{10}}u^2_{kx_2}] d\xi d\eta dx_2 - 2  \frac{d}{d\eta }[(2+\xi)^{-\frac{1}{10}}(2+\eta)^{-\frac{1}{10}} u^2_{k\xi}]d\xi d\eta dx_2\\
&-&  \iiint \frac{d}{dx_2}[(2+\xi)^{-\frac{1}{10}}(2+\eta)^{-\frac{1}{10}}u_{k\xi}u_{kx_2}] d\xi d\eta dx_2\\
&+&\iiint 4 \frac{d}{d\eta}[(2+\xi)^{-\frac{1}{10}}(2+\eta)^{-\frac{1}{10}}(1-H)u_{\xi}u_{k\xi} Q_0 ] + 4 \frac{d}{d\xi}[(2+\xi)^{-\frac{1}{10}}(2+\eta)^{-\frac{1}{10}}(1-H)u_{\eta}u_{k\xi} Q_0] \\
&-& 2 \frac{d}{d x_2}[(2+\xi)^{-\frac{1}{10}}(2+\eta)^{-\frac{1}{10}}(1-H)u_{x_2}u_{k\xi} Q_0]d\xi d\eta dx_2\\
&-& 2 \iiint \frac{d}{d\xi} [(2+\xi)^{-\frac{1}{10}}(2+\eta)^{-\frac{1}{10}}(1-H) Q^2_0(u, u_k)]d\xi d\eta dx_2 \\
&+& 2\iiint [(2+\xi)^{-\frac{1}{10}}(2+\eta)^{-\frac{1}{10}}(1-H)]_{\xi} Q^2_0(u, u_k)d\xi d\eta dx_2 \\
&+& \iiint 4(2+\xi)^{-\frac{1}{10}}(2+\eta)^{-\frac{1}{10}}(1-H)Q_0(u_k, u)Q_0(u_{\xi}, u_k) - 4 u_{k\xi} Q_0((2+\xi)^{-\frac{1}{10}}(2+\eta)^{-\frac{1}{10}}(1-H),u)Q(u_k, u)d\xi d\eta dx_2 \\
&-&  2\iiint (2+\xi)^{-\frac{1}{10}}(2+\eta)^{-\frac{1}{10}}(1-H)u_{k\xi}(\Box u + 4F'^2u_{\eta\eta}) Q_0(u_k,u)d\xi d\eta dx_2\\
&+& \iiint 8\frac{d}{d \eta} [(2+\xi)^{-\frac{1}{10}}(2+\eta)^{-\frac{1}{10}}(1-H)F' u_{\xi} u_{k\xi}u_{k\eta}]
- 4\frac{d}{d \xi}[(2+\xi)^{-\frac{1}{10}}(2+\eta)^{-\frac{1}{10}}(1-H)F' u_{\xi} u^2_{k\eta}] d\xi d\eta dx_2 \\
&-&8\iiint [(2+\xi)^{-\frac{1}{10}}(2+\eta)^{-\frac{1}{10}}(1-H)F' u_{\xi}]_{\eta} u_{k\xi}u_{k\eta}d\xi d\eta dx_2 +  \iiint 4[(2+\xi)^{-\frac{1}{10}}(2+\eta)^{-\frac{1}{10}}(1-H)F' u_{\xi}]_{\xi} u^2_{k\eta} d\xi d\eta dx_2 
\end{eqnarray*}
\begin{eqnarray*}
&+& 4 \iiint \frac{d}{d \eta}[(2+\xi)^{-\frac{1}{10}}(2+\eta)^{-\frac{1}{10}}(1-H)F' u_{\eta} u^2_{k\xi}] d\xi d\eta dx_2 - 4 \iiint [(2+\xi)^{-\frac{1}{10}}(2+\eta)^{-\frac{1}{10}}(1-H)F'u_{\eta}]_{\eta} u^2_{k\xi}d\xi d\eta dx_2  \\ \nonumber
&-& 4\iiint \frac{d}{d\eta} [(2+\xi)^{-\frac{1}{10}}(2+\eta)^{-\frac{1}{10}} (1-H) F' u_{x_2}u_{k\xi}u_{kx_2} ] d\xi d\eta dx_2\\
& +& 4\iiint  [(2+\xi)^{-\frac{1}{10}}(2+\eta)^{-\frac{1}{10}} (1-H) F' u_{x_2}]_{\eta} u_{k\xi} u_{kx_2} d\xi d\eta dx_2 \\
&+& \iiint  P [Q_0(u, 2Q_0(u, u_k))+ 8F'u_{\xi}u_{k\eta \eta} + 2F'u_{\eta}u_{kx_2x_2} - 4F'u_{x_2}u_{k\eta x_2} + J_k + 8F'u_{\eta}\Gamma^k(F'^2u_{\eta\eta}) - 4\Gamma^k(F'^2u_{\eta\eta})] u_{kx_2}d\xi d\eta dx_2 \\
&+& \frac{1}{2}\iiint (2+\xi)^{-\frac{1}{10}}(2+\eta)^{-\frac{1}{10}}\{ \frac{d}{dx_2}[(1-H) F' u_{x_2} u^2_{kx_2}] - [(1-H) F' u_{x_2}]_{x_2} u^2_{kx_2} \}d\xi d\eta dx_2\\
&+& \iiint (2+\xi)^{-\frac{1}{10}}(2+\eta)^{-\frac{1}{10}}u_{k\xi} (1-H)J_k d\xi d\eta dx_2\\
&+&\iiint (2+\xi)^{-\frac{1}{10}}(2+\eta)^{-\frac{1}{10}}u_{k\xi} (1-H)4\Gamma^k(F'^2u_{\eta\eta}H)d\xi d\eta dx_2
\end{eqnarray*}
In the following we also deal with the term
\begin{eqnarray}\nonumber
&&\iiint P[Q_0(u, 2Q_0(u, u_k))+ 8F'u_{\xi}u_{k\eta \eta} + 2F'u_{\eta}u_{kx_2x_2} - 4F'u_{x_2}u_{k\eta x_2} + J_k + 8F'u_{\eta}\Gamma^k(F'^2u_{\eta\eta}) - 4\Gamma^k(F'^2u_{\eta\eta})] u_{kx_2}d\xi d\eta dx_2 \\ \label{3.41}
&&\doteq III_1 + \cdots + III_7
\end{eqnarray}
\begin{eqnarray*}
III_1 &=& \iiint PQ_0(u, 2Q_0(u, u_k))u_{kx_2}d\xi d\eta dx_2 = \iiint P [4 u_{\xi}Q_{0\eta} + 4 u_{\eta}Q_{0\xi} - 2u_{x_2}Q_{0x_2}]u_{kx_2}d\xi d\eta dx_2 \\
&=& \iiint \frac{d}{d\eta}[4P u_{\xi}Q_0u_{kx_2}]- 4Pu_{\xi}u_{kx_2\eta} Q_0 - 4[Pu_{\xi}]_{\eta} Q_0 u_{kx_2} d\xi d\eta dx_2\\
&+&   \iiint \frac{d}{d\xi}[4P u_{\eta}Q_0u_{kx_2}]- 4Pu_{\eta}u_{kx_2\xi} Q_0 - 4[Pu_{\eta}]_{\xi} Q_0 u_{kx_2} d\xi d\eta dx_2\\
&-&    \iiint \frac{d}{dx_2}[2P u_{x_2}Q_0u_{kx_2}]- 2Pu_{\eta}u_{kx_2\xi} Q_0 - 2[Pu_{x_2}]_{x_2} Q_0 u_{kx_2} d\xi d\eta dx_2\\
&=&  \iiint \frac{d}{d\eta}[4P u_{\xi}Q_0u_{kx_2}] + \frac{d}{d\xi}[4P u_{\eta}Q_0u_{kx_2}] - \frac{d}{dx_2}[2P u_{x_2}Q_0u_{kx_2}] d\xi d\eta dx_2\\
&-&   2\iiint PQ_0(u, u_{kx_2}) Q_0d\xi d\eta dx_2 \\
&-&    \iiint \{4[Pu_{\eta}]_{\xi}  + 4[P(x)u_{\xi}]_{\eta}  - 2[Pu_{x_2}]_{x_2}\} Q_0 u_{kx_2} d\xi d\eta dx_2
\end{eqnarray*}
Using the integration by parts, we have
\begin{eqnarray*}
 &&2\iiint PQ_0(u, u_{kx_2}) Q_0d\xi d\eta dx_2 =  \iiint P \frac{d}{dx_2}Q^2_0(u, u_k)d\xi d\eta dx_2 \\
&&= \iiint \frac{d}{dx_2}[PQ^2_0(u, u_k)] d\xi d\eta dx_2 - \iiint P_{x_2}Q^2_0(u, u_k)d\xi d\eta dx_2
\end{eqnarray*}
Then,
\begin{eqnarray*}
III_1 &=&  \iiint \frac{d}{d\eta}[4P u_{\xi}Q_0u_{kx_2}] + \frac{d}{d\xi}[4P u_{\eta}Q_0u_{kx_2}] - \frac{d}{dx_2}[2P u_{x_2}Q_0u_{kx_2}] d\xi d\eta dx_2\\
&-& \iiint \frac{d}{dx_2}[PQ^2_0(u, u_k)] d\xi d\eta dx_2 + \iiint P_{x_2}Q^2_0(u, u_k)d\xi d\eta dx_2  \\
&-&  \iiint \{4[Pu_{\eta}]_{\xi}  + 4[Pu_{\xi}]_{\eta}  - 2[Pu_{x_2}]_{x_2}\} Q_0 u_{kx_2} d\xi d\eta dx_2,\\
III_2 &=& \iiint P8F'u_{\xi}u_{k\eta \eta} u_{kx_2}d\xi d\eta dx_2\\
&=&  8\iiint\frac{d}{d\eta}[PF'u_{\xi}u_{k\eta} u_{kx_2}] - [PF'u_{\xi}u_{k\eta} u_{kx_2\eta}] - [PF'u_{\xi}]_{\eta}u_{k\eta} u_{kx_2}d\xi d\eta dx_2\\
&=&  \iiint 8\frac{d}{d\eta}[PF'u_{\xi}u_{k\eta} u_{kx_2}] - 4\frac{d}{dx_2}[PF'u_{\xi}u^2_{k\eta}] + 4[PF'u_{\xi}]_{x_2}u^2_{k\eta} - 8[PF'u_{\xi}]_{\eta}u_{k\eta} u_{kx_2}d\xi d\eta dx_2\\
III_3 &=& \iiint P2F'u_{\eta}u_{kx_2x_2}u_{kx_2}d\xi d\eta dx_2=\iiint P F'u_{\eta}\frac{d}{dx_2}u^2_{kx_2}d\xi d\eta dx_2\\
&=& \iiint \frac{d}{dx_2}[P F'u_{\eta}u^2_{kx_2}] -  [P F'u_{\eta}]_{x_2}u^2_{kx_2}d\xi d\eta dx_2\\
III_4 &=& \iiint P 4F'u_{x_2}u_{k\eta x_2}u_{kx_2}d\xi d\eta dx_2 = \iiint 2P F'u_{x_2}\frac{d}{d\eta}u^2_{kx_2}d\xi d\eta dx_2\\
&=& 2\iiint\frac{d}{d\eta}[P F'u_{x_2}u^2_{kx_2}] -  [P F'u_{x_2}]_{\eta}u^2_{kx_2}d\xi d\eta dx_2
\end{eqnarray*}
Then
\begin{eqnarray*}
&&\iiint P[Q_0(u, 2Q_0(u, u_k))+ 8F'u_{\xi}u_{k\eta \eta} + 2F'u_{\eta}u_{kx_2x_2} - 4F'u_{x_2}u_{k\eta x_2} + J_k + 8F'u_{\eta}\Gamma^k(F'^2u_{\eta\eta}) - 4\Gamma^k(F'^2u_{\eta\eta})] u_{kx_2}d\xi d\eta dx_2 \\
&=&  \iiint \frac{d}{d\eta}[4P u_{\xi}Q_0u_{kx_2}] + \frac{d}{d\xi}[4P u_{\eta}Q_0u_{kx_2}] - \frac{d}{dx_2}[2P u_{x_2}Q_0u_{kx_2}] d\xi d\eta dx_2\\
&-& \iiint \frac{d}{dx_2}[PQ^2_0(u, u_k)] d\xi d\eta dx_2 + \iiint P_{x_2}Q^2_0(u, u_k)d\xi d\eta dx_2  \\
&-&  \iiint \{4[Pu_{\eta}]_{\xi}  + 4[P(x)u_{\xi}]_{\eta}  - 2[Pu_{x_2}]_{x_2}\} Q_0 u_{kx_2} d\xi d\eta dx_2\\
&+&  \iiint 8\frac{d}{d\eta}[PF'u_{\xi}u_{k\eta} u_{kx_2}] - 4\frac{d}{dx_2}[PF'u_{\xi}u^2_{k\eta}] + 4[PF'u_{\xi}]_{x_2}u^2_{k\eta} - 8[PF'u_{\xi}]_{\eta}u_{k\eta} u_{kx_2}d\xi d\eta dx_2\\
&+& \iiint \frac{d}{dx_2}[P F'u_{\eta}u^2_{kx_2}] -  [P F'u_{\eta}]_{x_2}u^2_{kx_2}d\xi d\eta dx_2\\
&+& 2\iiint\frac{d}{d\eta}[P F'u_{x_2}u^2_{kx_2}] -  [P F'u_{x_2}]_{\eta}u^2_{kx_2}d\xi d\eta dx_2\\
&-& 4\iiint\frac{d}{d\eta}[P F'^2u_{k\eta}u^2_{kx_2}] - P_{\eta}F'^2u_{k\eta}u_{kx_2} -  \frac{1}{2}\frac{d}{dx_2}[P F'^2u^2_{k\eta}] +  \frac{1}{2}P_{x_2}F'^2u^2_{k\eta} d\xi d\eta dx_2\\
&+& 8\iiint\frac{d}{d\eta}[P F'^3u_{k\eta}u^2_{kx_2}] - P_{\eta}F'^3u_{k\eta}u_{kx_2} -  \frac{1}{2}\frac{d}{dx_2}[P F'^3u^2_{k\eta}] +  \frac{1}{2}P_{x_2}F'^3u^2_{k\eta} d\xi d\eta dx_2\\
&+& \iiint P J_k u_{kx_2}d\xi d\eta dx_2
\end{eqnarray*}
Finally, we can get
\begin{eqnarray*}
&& \frac{1}{20} \iiint (2+\xi)^{-\frac{11}{10}}(2+\eta)^{-\frac{1}{10}}u^2_{kx_2} d\xi d\eta dx_2 + \frac{1}{5} \iiint (2+\xi)^{-\frac{1}{10}}(2+\eta)^{-\frac{11}{10}}u^2_{k\xi} d\xi d\eta dx_2 \\ \nonumber
&=&   \frac{1}{2} \iiint \frac{d}{d\xi} [(2+\xi)^{-\frac{1}{10}}(2+\eta)^{-\frac{1}{10}}u^2_{kx_2}] d\xi d\eta dx_2 - 2  \frac{d}{d\eta }[(2+\xi)^{-\frac{1}{10}}(2+\eta)^{-\frac{1}{10}} u^2_{k\xi}]d\xi d\eta dx_2\\
&-&  \iiint \frac{d}{dx_2}[(2+\xi)^{-\frac{1}{10}}(2+\eta)^{-\frac{1}{10}}u_{k\xi}u_{kx_2}] d\xi d\eta dx_2\\
&+&\iiint 4 \frac{d}{d\eta}[(2+\xi)^{-\frac{1}{10}}(2+\eta)^{-\frac{1}{10}}(1-H)u_{\xi}u_{k\xi} Q_0 ] + 4 \frac{d}{d\xi}[(2+\xi)^{-\frac{1}{10}}(2+\eta)^{-\frac{1}{10}}(1-H)u_{\eta}u_{k\xi} Q_0] \\
&-& 2 \frac{d}{d x_2}[(2+\xi)^{-\frac{1}{10}}(2+\eta)^{-\frac{1}{10}}(1-H)u_{x_2}u_{k\xi} Q_0]d\xi d\eta dx_2\\
&-& 2 \iiint \frac{d}{d\xi} [(2+\xi)^{-\frac{1}{10}}(2+\eta)^{-\frac{1}{10}}(1-H) Q^2_0(u, u_k)]d\xi d\eta dx_2 \\
&+& 2\iiint [(2+\xi)^{-\frac{1}{10}}(2+\eta)^{-\frac{1}{10}}(1-H)]_{\xi} Q^2_0(u, u_k)d\xi d\eta dx_2 \\
&+& \iiint 4(2+\xi)^{-\frac{1}{10}}(2+\eta)^{-\frac{1}{10}}(1-H)Q_0(u_k, u)Q_0(u_{\xi}, u_k) - 4 u_{k\xi} Q_0((2+\xi)^{-\frac{1}{10}}(2+\eta)^{-\frac{1}{10}}(1-H),u)Q(u_k, u)d\xi d\eta dx_2 \\
&-&  2\iiint (2+\xi)^{-\frac{1}{10}}(2+\eta)^{-\frac{1}{10}}(1-H)u_{k\xi}(\Box u + 4F'^2u_{\eta\eta}) Q_0(u_k,u)d\xi d\eta dx_2\\
&+& \iiint 8\frac{d}{d \eta} [(2+\xi)^{-\frac{1}{10}}(2+\eta)^{-\frac{1}{10}}(1-H)F' u_{\xi} u_{k\xi}u_{k\eta}] - 4\frac{d}{d \xi}[(2+\xi)^{-\frac{1}{10}}(2+\eta)^{-\frac{1}{10}}(1-H)F' u_{\xi} u^2_{k\eta}] d\xi d\eta dx_2 \\
&-&8\iiint [(2+\xi)^{-\frac{1}{10}}(2+\eta)^{-\frac{1}{10}}(1-H)F' u_{\xi}]_{\eta} u_{k\xi}u_{k\eta}d\xi d\eta dx_2 +  \iiint 4[(2+\xi)^{-\frac{1}{10}}(2+\eta)^{-\frac{1}{10}}(1-H)F' u_{\xi}]_{\xi} u^2_{k\eta} d\xi d\eta dx_2 \\ \nonumber
&+& 4 \iiint \frac{d}{d \eta}[(2+\xi)^{-\frac{1}{10}}(2+\eta)^{-\frac{1}{10}}(1-H)F' u_{\eta} u^2_{k\xi}] d\xi d\eta dx_2 - 4 \iiint [(2+\xi)^{-\frac{1}{10}}(2+\eta)^{-\frac{1}{10}}(1-H)F'u_{\eta}]_{\eta} u^2_{k\xi}d\xi d\eta dx_2  \\ \nonumber
&-& 4\iiint \frac{d}{d\eta} [(2+\xi)^{-\frac{1}{10}}(2+\eta)^{-\frac{1}{10}} (1-H) F' u_{x_2}u_{k\xi}u_{kx_2} ] d\xi d\eta dx_2\\
& +& 4\iiint  [(2+\xi)^{-\frac{1}{10}}(2+\eta)^{-\frac{1}{10}} (1-H) F' u_{x_2}]_{\eta} u_{k\xi} u_{kx_2} d\xi d\eta dx_2 \\
&+& \frac{1}{2}\iiint (2+\xi)^{-\frac{1}{10}}(2+\eta)^{-\frac{1}{10}}\{ \frac{d}{dx_2}[(1-H) F' u_{x_2} u^2_{kx_2}] - [(1-H) F' u_{x_2}]_{x_2} u^2_{kx_2} \}d\xi d\eta dx_2\\
&+& \iiint (2+\xi)^{-\frac{1}{10}}(2+\eta)^{-\frac{1}{10}}u_{k\xi} (1-H)J_k d\xi d\eta dx_2\\
&+& \iiint \frac{d}{d\eta}[4P u_{\xi}Q_0u_{kx_2}] + \frac{d}{d\xi}[4P u_{\eta}Q_0u_{kx_2}] - \frac{d}{dx_2}[2P u_{x_2}Q_0u_{kx_2}] d\xi d\eta dx_2\\
&-& \iiint \frac{d}{dx_2}[PQ^2_0(u, u_k)] d\xi d\eta dx_2 + \iiint P_{x_2}(x)Q^2_0(u, u_k)d\xi d\eta dx_2  \\
&-&  \iiint \{4[Pu_{\eta}]_{\xi}  + 4[Pu_{\xi}]_{\eta}  - 2[Pu_{x_2}]_{x_2}\} Q_0 u_{kx_2} d\xi d\eta dx_2\\
&+&  \iiint 8\frac{d}{d\eta}[PF'u_{\xi}u_{k\eta} u_{kx_2}] - 4\frac{d}{dx_2}[PF'u_{\xi}u^2_{k\eta}] + 4[PF'u_{\xi}]_{x_2}u^2_{k\eta} - 8[PF'u_{\xi}]_{\eta}u_{k\eta} u_{kx_2}d\xi d\eta dx_2\\
&+& \iiint \frac{d}{dx_2}[P F'u_{\eta}u^2_{kx_2}] -  [PF'u_{\eta}]_{x_2}u^2_{kx_2}d\xi d\eta dx_2\\
&+& 2\iiint\frac{d}{d\eta}[P F'u_{x_2}u^2_{kx_2}] -  [P F'u_{x_2}]_{\eta}u^2_{kx_2}d\xi d\eta dx_2
\end{eqnarray*}
\begin{eqnarray*}
&+& \iiint P J_k u_{kx_2}d\xi d\eta dx_2 + \iiint (2+\xi)^{-\frac{1}{10}}(2+\eta)^{-\frac{1}{10}}u_{k\xi} (1-H)4\Gamma^k(F'^2u_{\eta\eta}H)d\xi d\eta dx_2\\
&+& \iiint P[8 F' u_{\eta}\Gamma^k(F'^2 u_{\eta\eta}) - 4 \Gamma^k(F'^2 u_{\eta\eta})]u_{kx_2} d\xi d\eta dx_2
\end{eqnarray*}
Therefore,
\begin{eqnarray*}
&&  \iiint (2+\xi)^{-\frac{11}{10}}(2+\eta)^{-\frac{1}{10}}u^2_{kx_2} d\xi d\eta dx_2 +  \iiint (2+\xi)^{-\frac{1}{10}}(2+\eta)^{-\frac{11}{10}}u^2_{k\xi} d\xi d\eta dx_2 \\ \nonumber
&\lesssim& \varepsilon + \iiint [(2+\xi)^{-\frac{1}{10}}(2+\eta)^{-\frac{1}{10}}(1-H)]_{\xi} Q^2_0(u, u_k)d\xi d\eta dx_2 \\
&+& \iiint (2+\xi)^{-\frac{1}{10}}(2+\eta)^{-\frac{1}{10}}(1-H)Q_0(u_k, u)Q_0(u_{\xi}, u_k) - u_{k\xi} Q_0((2+\xi)^{-\frac{1}{10}}(2+\eta)^{-\frac{1}{10}}(1-H),u)Q(u_k, u)d\xi d\eta dx_2 \\
&+&  \iiint (2+\xi)^{-\frac{1}{10}}(2+\eta)^{-\frac{1}{10}}(1-H)u_{k\xi}(\Box u + 4F'^2u_{\eta\eta}) Q_0(u_k,u)d\xi d\eta dx_2\\
&+& \iiint [(2+\xi)^{-\frac{1}{10}}(2+\eta)^{-\frac{1}{10}}(1-H)F' u_{\xi}]_{\eta} u_{k\xi}u_{k\eta}d\xi d\eta dx_2 +  \iiint [(2+\xi)^{-\frac{1}{10}}(2+\eta)^{-\frac{1}{10}}(1-H)F' u_{\xi}]_{\xi} u^2_{k\eta} d\xi d\eta dx_2 \\ \nonumber
&+&  \iiint [(2+\xi)^{-\frac{1}{10}}(2+\eta)^{-\frac{1}{10}}(1-H)F'u_{\eta}]_{\eta} u^2_{k\xi}d\xi d\eta dx_2  \\
&+& \iiint  [(2+\xi)^{-\frac{1}{10}}(2+\eta)^{-\frac{1}{10}} (1-H) F' u_{x_2}]_{\eta} u_{k\xi} u_{kx_2} d\xi d\eta dx_2
+\iiint (2+\xi)^{-\frac{1}{10}}(2+\eta)^{-\frac{1}{10}}[(1-H) F' u_{x_2}]_{x_2} u^2_{kx_2} d\xi d\eta dx_2\\
&+& \iiint (2+\xi)^{-\frac{1}{10}}(2+\eta)^{-\frac{1}{10}}u_{k\xi} (1-H)J_k d\xi d\eta dx_2\\
&+&\iiint P_{x_2}(x)Q^2_0(u, u_k)d\xi d\eta dx_2  +  \iiint \{4[Pu_{\eta}]_{\xi}  + [Pu_{\xi}]_{\eta}  - 2[Pu_{x_2}]_{x_2}\} Q_0 u_{kx_2} d\xi d\eta dx_2\\
&+&  \iiint [PF'u_{\xi}]_{x_2}u^2_{k\eta} - 8[PF'u_{\xi}]_{\eta}u_{k\eta} u_{kx_2}d\xi d\eta dx_2 + \iiint  [PF'u_{\eta}]_{x_2}u^2_{kx_2}d\xi d\eta dx_2
+ \iiint [P F'u_{x_2}]_{\eta}u^2_{kx_2}d\xi d\eta dx_2 \\&+& \iiint P J_k u_{kx_2}d\xi d\eta dx_2+ \iiint (2+\xi)^{-\frac{1}{10}}(2+\eta)^{-\frac{1}{10}}u_{k\xi} (1-H)4\Gamma^k(F'^2u_{\eta\eta}H)d\xi d\eta dx_2\\
&+& \iiint P[8 F' u_{\eta}\Gamma^k(F'^2 u_{\eta\eta}) - 4 \Gamma^k(F'^2 u_{\eta\eta})]u_{kx_2} d\xi d\eta dx_2\\
&\doteq& \bar{A}_1 + \cdots + \bar{A}_{14}.
\end{eqnarray*}
In the following, we will estimate $\bar{A}_i, i = 1, \cdots, 14$ respectively.
Using the similar procedures to (\ref{3.9}) and (\ref{3.15}), (\ref{3.19}), we can get the estimations of $\bar{A}_1$ and $\bar{A}_2$
\begin{eqnarray}\label{3.42}
&&\bar{A}_1 = \iiint [(2+\xi)^{-\frac{1}{10}}(2+\eta)^{-\frac{1}{10}}(1-H)]_{\xi} Q^2_0(u, u_k)d\xi d\eta dx_2 \lesssim \varepsilon + (e_s + e_s^{\frac{3}{2}} + e_s^{\frac{3}{2}}\tilde{e}_s + e^2_s) E_s, \\ \nonumber
&&\bar{A}_2 = \iiint (2+\xi)^{-\frac{1}{10}}(2+\eta)^{-\frac{1}{10}}(1-H)Q_0(u_k, u)Q_0(u_{\xi}, u_k) - u_{k\xi} Q_0((2+\xi)^{-\frac{1}{10}}(2+\eta)^{-\frac{1}{10}}(1-H),u)Q(u_k, u)d\xi d\eta dx_2\\ \label{3.43}
&&\lesssim (\tilde{e}_se_s^{\frac{3}{2}} +  e^{2}_s) E_s.
\end{eqnarray}
Noting the estimate of $A_2$, we can get
\begin{eqnarray} \label{3.44}
\bar{A}_3 = \iiint (2+\xi)^{-\frac{1}{10}}(2+\eta)^{-\frac{1}{10}}(1-H)u_{k\xi}(\Box u + 4F'^2u_{\eta\eta}) Q_0(u_k,u)d\xi d\eta dx_2 \lesssim   e^{\frac{1}{2}}_s E_s + e_s E_s  + \tilde{e}_se_s^{\frac{3}{2}}E_s  + e^{2}_s E_s.
\end{eqnarray}
Noting,
\begin{eqnarray*}
&&\bar{A}_{41} = \iiint [(2+\xi)^{-\frac{1}{10}}(2+\eta)^{-\frac{1}{10}}(1-H)F' u_{\xi}]_{\eta} u_{k\xi}u_{k\eta}d\xi d\eta dx_2
\\
&&\bar{A}_{42} =\iiint [(2+\xi)^{-\frac{1}{10}}(2+\eta)^{-\frac{1}{10}}(1-H)F' u_{\xi}]_{\xi} u^2_{k\eta} d\xi d\eta dx_2
\end{eqnarray*}
we have
\begin{eqnarray}\nonumber
\bar{A}_{41} &\lesssim& (1+e_s)[\iiint [(2+\xi)^{-\frac{31}{10}}(2+\eta)^{-\frac{11}{10}}|u_{\xi}]u_{k\xi}u_{k\eta}|d\xi d\eta dx_2
+ \iiint [(2+\xi)^{-\frac{11}{10}}(2+\eta)^{-\frac{1}{10}} u_{\xi\eta}] u_{k\xi}u_{k\eta}d\xi d\eta dx_2 \\ \nonumber
&+& \iiint [(2+\xi)^{-\frac{31}{10}}(2+\eta)^{-\frac{1}{10}}|H_{\eta}F' u_{\xi}] u_{k\xi}u_{k\eta}|d\xi d\eta dx_2]\\ \nonumber
&\lesssim & (1+e_s)e^{\frac{1}{2}}_s[E_s + \iiint [(2+\xi)^{-\frac{21}{10}}(2+\eta)^{-\frac{1}{10}}|(Q_{0\eta}+4F'(\xi)u_{\eta\eta}) u_{\xi}] u_{k\xi}u_{k\eta}|d\xi d\eta dx_2]]\\ \nonumber
&\lesssim & (1+e_s)e^{\frac{1}{2}}_s[E_s + \iiint (2+\xi)^{-\frac{3}{2}}(2+\eta)^{-\frac{17}{20}} (e^{\frac{1}{2}}_s + e_s) (2+\xi)^{\delta}\tilde{e}_s|u_{k\xi}u_{k\eta}|d\xi d\eta dx_2\\ \nonumber
&\lesssim & (1+e_s)e^{\frac{1}{2}}_s[E_s +  \tilde{e}_s(e^{\frac{1}{2}}_s + e_s)E_s].
\end{eqnarray}
Furthermore, noting Proposition 3 and Corollary 1,
\begin{eqnarray}\nonumber
\bar{A}_{42} &=&\iiint [(2+\xi)^{-\frac{1}{10}}(2+\eta)^{-\frac{1}{10}}(1-H)F' u_{\xi}]_{\xi} u^2_{k\eta} d\xi d\eta dx_2\\ \nonumber
&\lesssim &\iiint [(2+\xi)^{-\frac{1}{10}}(2+\eta)^{-\frac{1}{10}}(1-H)|F'' u_{\xi}|] u^2_{k\eta} d\xi d\eta dx_2 + \iiint [(2+\xi)^{-\frac{1}{10}}(2+\eta)^{-\frac{1}{10}}(1-H)|F' u_{\xi \xi}|] u^2_{k\eta} d\xi d\eta dx_2\\ \nonumber
& + &  \iiint [(2+\xi)^{-\frac{1}{10}}(2+\eta)^{-\frac{1}{10}}|H_{\xi}F' u_{\xi}|] u^2_{k\eta} d\xi d\eta dx_2
\\ \nonumber
&\lesssim &\iiint [(2+\xi)^{-\frac{21}{10}}(2+\eta)^{-\frac{1}{10}}(1-H) |u_{\xi}|] u^2_{k\eta} d\xi d\eta dx_2 + \iiint [(2+\xi)^{-\frac{11}{10}}(2+\eta)^{-\frac{1}{10}}(1-H) |u_{\xi \xi}|] u^2_{k\eta} d\xi d\eta dx_2\\ \nonumber
& + &  \iiint [(2+\xi)^{-\frac{11}{10}}(2+\eta)^{-\frac{1}{10}}|(Q_{0\xi}+4F'(\xi)u_{\eta\xi}) u_{\xi}|] u^2_{k\eta} d\xi d\eta dx_2
\\ \nonumber
&\lesssim & (1+e_s)\iiint (2+\xi)^{-\frac{21}{10}}(2+\eta)^{-\frac{1}{10}}(2+\xi )^{\delta}\tilde{e}_s u^2_{k\eta} d\xi d\eta dx_2 \\ \nonumber
& + &  \iiint [(2+\xi)^{-\frac{11}{10}}(2+\eta)^{-\frac{1}{10}}(Q_{0\xi} + F'(\xi)u_{\eta\xi})][(2+\xi )^{\frac{1}{4}}(2+\eta )^{-\frac{1}{4}}e_s + (2+\xi)^{\frac{1}{2}}e_s + e_s] u^2_{k\eta} d\xi d\eta dx_2\\ \nonumber
&\lesssim & (1+e_s)\tilde{e}_sE_s
\\ \nonumber
& + &  \iiint (2+\xi)^{-\frac{11}{10}}(2+\eta)^{-\frac{1}{10}}[ (1+(2+\xi)^{\frac{1}{2}})e^{\frac{3}{2}}_s + (2+\xi)^{-\frac{1}{2}}(2+\eta)^{-\frac{1}{2}}e_s + (2+\eta)^{-\frac{1}{2}}e^{\frac{1}{2}}_s)][1+(2+\xi)^{\frac{1}{2}}]e_s u^2_{k\eta} d\xi d\eta dx_2\\ \label{3.45}
&\lesssim & [(1+e_s)\tilde{e}_s  + e_s + e^{\frac{3}{2}}_s + e_s^2 + e^{\frac{5}{2}}_s]E_s.
\end{eqnarray}

Denote
$$\bar{A}_5 =  \iiint [(2+\xi)^{-\frac{1}{10}}(2+\eta)^{-\frac{1}{10}}(1-H)F'u_{\eta}]_{\eta} u^2_{k\xi}d\xi d\eta dx_2.$$
Noting Corollary 1, we can get
\begin{eqnarray}\nonumber
\bar{A}_5 &= & \iiint [(2+\xi)^{-\frac{1}{10}}(2+\eta)^{-\frac{1}{10}}(1-H)F'u_{\eta \eta}] u^2_{k\xi}d\xi d\eta dx_2 + \iiint [(2+\xi)^{-\frac{1}{10}}(2+\eta)^{-\frac{1}{10}}H_{\eta}F'u_{\eta} ] u^2_{k\xi}d\xi d\eta dx_2 \\ \nonumber
&\lesssim & (1 + e_s )\iiint [(2+\xi)^{-\frac{11}{10}}(2+\eta)^{-\frac{1}{10}}[(2+\eta)^{-1}|\Gamma u_{\eta}+u_{\eta\eta}| + (2+\xi)^{\frac{1}{2}}(2+\eta)^{-\frac{1}{2}}|u_{x_2\eta}|] u^2_{k\xi}d\xi d\eta dx_2 \\ \nonumber
 &+& \iiint [(2+\xi)^{-\frac{1}{10}}(2+\eta)^{-\frac{1}{10}}H_{\eta}F'u_{\eta} ] u^2_{k\xi}d\xi d\eta dx_2 \\ \nonumber
&\lesssim & (1 + e_s )\iiint [(2+\xi)^{-\frac{11}{10}}(2+\eta)^{-\frac{1}{10}}[(2+\eta)^{-\frac{3}{2}} + (2+\xi)^{\frac{1}{2}}(2+\eta)^{-1}]e_s^{\frac{1}{2}} u^2_{k\xi}d\xi d\eta dx_2 \\ \nonumber
&+& \iiint (2+\xi)^{-\frac{11}{10}}(2+\eta)^{-\frac{1}{10}}|[Q_{0\eta}+4F'(\xi)u_{\eta\eta}]u_{\eta}|  u^2_{k\xi}d\xi d\eta dx_2 
\end{eqnarray}
\begin{eqnarray} \nonumber
&\lesssim & (1 + e_s )\iiint [(2+\xi)^{-\frac{3}{5}}(2+\eta)^{-1}e^{\frac{1}{2}}_s u^2_{k\xi}d\xi d\eta dx_2 \\ \nonumber
&+& \iiint (2+\xi)^{-\frac{11}{10}}(2+\eta)^{-\frac{1}{10}}[|u_{\xi \eta}||u_{\eta}|+|u_{\xi}||u_{\eta\eta}|+|u_{x_2 \eta}||u_{x_2}| + (2+\eta)^{-\frac{1}{2}}e^{\frac{1}{2}}_s](2+\eta)^{-\frac{1}{2}}e^{\frac{1}{2}}_s u^2_{k\xi}d\xi d\eta dx_2 \\ \nonumber
&\lesssim & (1 + e_s )e^{\frac{1}{2}}_s E_s
+ \iiint (2+\xi)^{-\frac{11}{10}}(2+\eta)^{-\frac{1}{10}}[(2+\xi)^{-\frac{1}{2}}(2+\eta)^{-\frac{1}{2}}e_s + (2+\eta)^{-\frac{1}{2}}e^{\frac{1}{2}}_s](2+\eta)^{-\frac{1}{2}}e^{\frac{1}{2}}_s u^2_{k\xi}d\xi d\eta dx_2 \\ \label{3.46}
&\lesssim & (1 + e_s )e^{\frac{1}{2}}_s E_s  + (e_s + e^{\frac{3}{2}}_s) E_s.
 \end{eqnarray}

Denote
\begin{eqnarray}\nonumber
\bar{A}_6 &=& \iiint  [(2+\xi)^{-\frac{1}{10}}(2+\eta)^{-\frac{1}{10}} (1-H) F' u_{x_2}]_{\eta} u_{k\xi} u_{kx_2} d\xi d\eta dx_2 \\ \label{3.47}
&+& \iiint (2+\xi)^{-\frac{1}{10}}(2+\eta)^{-\frac{1}{10}}[(1-H) F' u_{x_2}]_{x_2} u^2_{kx_2} d\xi d\eta dx_2 \doteq \bar{A}_{61} + \bar{A}_{62}.
\end{eqnarray}
Then, noting the estimate of $\bar{A}_5$,
\begin{eqnarray} \nonumber
\bar{A}_{61} &\lesssim& \iiint  (2+\xi)^{-\frac{11}{10}}(2+\eta)^{-\frac{11}{10}} (1-H) |u_{x_2}| |u_{k\xi} u_{kx_2}| d\xi d\eta dx_2 \\ \nonumber
&+ &\iiint  [(2+\xi)^{-\frac{11}{10}}(2+\eta)^{-\frac{1}{10}} (1-H) |u_{x_2\eta}|] |u_{k\xi} u_{kx_2} |d\xi d\eta dx_2 + \iiint  [(2+\xi)^{-\frac{11}{10}}(2+\eta)^{-\frac{1}{10}} |H_{\eta} u_{x_2}| u_{k\xi} u_{kx_2}| d\xi d\eta dx_2\\ \nonumber
&\lesssim & (1+e_s) \iiint  [(2+\xi)^{-\frac{11}{10}}(2+\eta)^{-\frac{3}{5}}  e^{\frac{1}{2}}_s |u_{k\xi} u_{kx_2}| d\xi d\eta dx_2 + \iiint  [(2+\xi)^{-\frac{11}{10}}(2+\eta)^{-\frac{7}{20}} |H_{\eta}|  e^{\frac{1}{2}}_s| u_{k\xi} u_{kx_2} |d\xi d\eta dx_2\\ \nonumber
&\lesssim & (1+e_s)e^{\frac{1}{2}}_s E_s + \iiint  [(2+\xi)^{-\frac{11}{10}}(2+\eta)^{-\frac{7}{20}} |(Q_{0\eta}+4F'(\xi)u_{\eta\eta})|  e^{\frac{1}{2}}_s |u_{k\xi} u_{kx_2}| d\xi d\eta dx_2
\\ \nonumber
&\lesssim & (1+e_s)e^{\frac{1}{2}}_s E_s + \iiint  [(2+\xi)^{-\frac{11}{10}}(2+\eta)^{-\frac{7}{20}} [(2+\eta)^{-1}e_s + (2+\eta)^{-\frac{1}{2}}e^{\frac{1}{2}}_s] e^{\frac{1}{2}}_s |u_{k\xi} u_{kx_2}| d\xi d\eta dx_2
\\ \nonumber
&\lesssim & (1+e_s)e^{\frac{1}{2}}_s E_s + (e_s + e^{\frac{3}{2}}_s) \iiint  [(2+\xi)^{-\frac{11}{10}}(2+\eta)^{-\frac{17}{20}} |u_{k\xi} u_{kx_2}| d\xi d\eta dx_2\\ \label{3.48}
&\lesssim &(1+e_s)e^{\frac{1}{2}}_s E_s +  (e_s + e^{\frac{3}{2}}_s) E_s.
\end{eqnarray}
and
\begin{eqnarray}\nonumber
\bar{A}_{62} 
&\lesssim &\iiint  [(2+\xi)^{-\frac{11}{10}}(2+\eta)^{-\frac{1}{10}} (1-H) |u_{x_2 x_2}|] u^2_{kx_2} d\xi d\eta dx_2 + \iiint  [(2+\xi)^{-\frac{11}{10}}(2+\eta)^{-\frac{1}{10}} |H_{x_2} u_{x_2}| u^2_{kx_2} d\xi d\eta dx_2\\ \nonumber
&\lesssim & (1+e_s) [\iiint  (2+\xi)^{-\frac{27}{20}}(2+\eta)^{-\frac{7}{20}}  e^{\frac{1}{2}}_s  u^2_{kx_2} d\xi d\eta dx_2 + \iiint  (2+\xi)^{-\frac{27}{20}}(2+\eta)^{-\frac{7}{20}} H_{x_2}  e^{\frac{1}{2}}_s  u^2_{kx_2} d\xi d\eta dx_2]\\ \nonumber
&\lesssim & (1+e_s)e^{\frac{1}{2}}_s E_s + (1+e_s)\iiint  [(2+\xi)^{-\frac{11}{10}}(2+\eta)^{-\frac{7}{20}} |Q_{0 x_2}+4F'(\xi)u_{\eta x_2}|  e^{\frac{1}{2}}_s u^2_{kx_2} d\xi d\eta dx_2\\ \nonumber
&\lesssim & (1+e_s)e^{\frac{1}{2}}_s E_s + (1+e_s)(e_s + e^{\frac{1}{2}}_s) e^{\frac{1}{2}}_s\iiint  [(2+\xi)^{-\frac{27}{20}}(2+\eta)^{-\frac{7}{20}} u^2_{kx_2} d\xi d\eta dx_2\\ \label{3.49}
&\lesssim & (1+e_s)e^{\frac{1}{2}}_s E_s + (1+e_s)(e_s + e^{\frac{3}{2}}_s) E_s.
\end{eqnarray}

Denote
\begin{eqnarray*}
 \bar{A}_7 = \iiint (2+\xi)^{-\frac{1}{10}}(2+\eta)^{-\frac{1}{10}}u_{k\xi} (1-H)J_k d\xi d\eta dx_2.
 \end{eqnarray*}
Similar to the estimate of the term $A_4$, we will first estimate $J_k$. Let
\begin{eqnarray*}
 &&J_{k1} = \sum\limits_{k_1+k_2+k_3+k_4 = k, k_3 < k, k_4 < k} \Gamma^{k_1}(1-H) Q_0(\Gamma^{k_2}u, Q_0(\Gamma^{k_3}u, \Gamma^{k_4}u))\\
 && J_{k2} = \sum\limits_{0\leq k_1\leq k} \Gamma^{k_1}(1-H)\Gamma^{k-k_1}[F' Q_0(u, u_{\eta})]\\
  && J_{k3} = \sum\limits_{0\leq k_1\leq k} \Gamma^{k_1}(1-H)\Gamma^{k-k_1}[F''u^2_\eta].
\end{eqnarray*}
For $J_{k1}$, when $k_1 \leq [\frac{1+s}{2}]$, we have
\begin{eqnarray}\nonumber
 &&|J_{k1}| = |\sum\limits_{k_1+k_2+k_3+k_4 = k, k_3 < k, k_4 < k} \Gamma^{k_1}(1-H) Q_0(\Gamma^{k_2}u, Q_0(\Gamma^{k_3}u, \Gamma^{k_4}u))|\\ \label{3.50}
 && \leq (1 + e_s)[|(\Gamma^{k_2}u)_{\xi}Q_{0\eta}(\Gamma^{k_3}u, \Gamma^{k_4}u)| + |(\Gamma^{k_2}u)_{\eta}Q_{0\xi}(\Gamma^{k_3}u, \Gamma^{k_4}u)| + |(\Gamma^{k_2}u)_{x_2}Q_{0x_2}(\Gamma^{k_3}u, \Gamma^{k_4}u)|.
\end{eqnarray}
Without loss of generality, we assume $|k_3|,|k_4| \leq [\frac{s+1}{2}]$ and noting (\ref{2.21}), then
\begin{eqnarray}\nonumber
 &&J_{k1} \lesssim(1 + e_s)|(\Gamma^{k_2}u)_{\xi}||Q_{0\eta}(\Gamma^{k_3}u, \Gamma^{k_4}u)| + |(\Gamma^{k_2}u)_{\eta}||Q_{0\xi}(\Gamma^{k_3}u, \Gamma^{k_4}u)| + |(\Gamma^{k_2}u)_{x_2}||Q_{0x_2}(\Gamma^{k_3}u, \Gamma^{k_4}u)|\\ \label{3.51}
 && \doteq J_{k11} + J_{k12} + J_{k13}.
 \end{eqnarray}
 Noting Lemma 3.2, Lemma 3.3 and Proposition 1,
 \begin{eqnarray*}\nonumber
 J_{k11} &\lesssim & (1 + e_s)|(\Gamma^{k_2}u)_{\xi}||Q_{0\eta}(\Gamma^{k_3}u, \Gamma^{k_4}u)| \\ \nonumber
 &\lesssim& (1 + e_s)|(\Gamma^{k_2}u)_{\xi}|(2+\eta)^{-1}(\partial_{\eta} + \Gamma_4 - x_2 \partial_{x_2})Q_0(\Gamma^{k_3}u, \Gamma^{k_4}u) \\
  &\lesssim & (1 + e_s)(2+\eta)^{-1}|(\Gamma^{k_2}u)_{\xi}|(\partial_{\eta} + \Gamma_4 - x_2 \partial_{x_2})[(\Gamma^{k_3}u)_{\eta}(\Gamma^{k_4}u)_{\xi}
 + (\Gamma^{k_3}u)_{x_2}(\Gamma^{k_4}u)_{x_2}] \\ \nonumber
 &\lesssim& (1 + e_s)(2+\eta)^{-1}|(\Gamma^{k_2}u)_{\xi}|[e_s + x_2 \partial_{x_2}[(\Gamma^{k_3}u)_{\eta}(\Gamma^{k_4}u)_{\xi}
 + (\Gamma^{k_3}u)_{x_2}(\Gamma^{k_4}u)_{x_2}] ]   \\ \nonumber
 &\lesssim &(1 + e_s)(2+\eta)^{-1}|(\Gamma^{k_2}u)_{\xi}|[e_s + x_2 [(\Gamma^{k_3}u)_{\eta x_2}(\Gamma^{k_4}u)_{\xi} + (\Gamma^{k_3}u)_{\eta}(\Gamma^{k_4}u)_{\xi x_2}
 + (\Gamma^{k_3}u)_{x_2 x_2}(\Gamma^{k_4}u)_{x_2}] ]   \\ \nonumber
&\lesssim & (1 + e_s)(2+\eta)^{-1}|(\Gamma^{k_2}u)_{\xi}|[e_s + x_2(\Gamma^{k_3}u)_{ x_2\eta}(\Gamma^{k_4}u)_{\xi} + (2+\eta)^{\frac{1}{2}}(2+\xi)^{\frac{1}{2}} [(2+\eta)^{-\frac{1}{2}}(2+\xi)^{-\frac{1}{2}} e_s ]]\\ \nonumber
&\lesssim & (1 + e_s)(2+\eta)^{-1}|(\Gamma^{k_2}u)_{\xi}|[e_s + x_2(\Gamma^{k_3}u)_{ x_2\eta}(\Gamma^{k_4}u)_{\xi}].
\end{eqnarray*}
By the definition of $\Gamma$ operator, we have
\begin{eqnarray*}
2x_2 \partial_{\eta} (\Gamma^{k_3}u)_{x_2}(\Gamma^{k_4}u)_{\xi} &=& (\Gamma_6 - \xi \partial_{x_2})(\Gamma^{k_3}u)_{x_2}(\Gamma^{k_4}u)_{\xi}\\  \nonumber
&=& (\Gamma^{k_4}u)_{\xi}\Gamma_6(\Gamma^{k_3}u)_{x_2} - \xi (\Gamma^{k_4}u)_{\xi}(\Gamma^{k_3}u)_{x_2 x_2} \\ \nonumber
&=&(\Gamma^{k_4}u)_{\xi}\Gamma_6(\Gamma^{k_3}u)_{x_2} - \frac{1}{2}[\Gamma_5 - x_2\partial_{x_2}](\Gamma^{k_4}u)_{\xi}(\Gamma^{k_3}u)_{x_2 x_2} \\ \nonumber
&=&(\Gamma^{k_4}u)_{\xi}\Gamma_6(\Gamma^{k_3}u)_{x_2} - \frac{1}{2}\Gamma_5(\Gamma^{k_4}u)_{\xi}(\Gamma^{k_3}u)_{x_2 x_2} - \frac{1}{2}x_2\partial_{x_2}(\Gamma^{k_4}u)_{\xi}(\Gamma^{k_3}u)_{x_2 x_2}.
\end{eqnarray*}
Then, noting Lemma 3, Proposition 1 and Proposition 2,
\begin{eqnarray*}
|x_2 \partial_{\eta} (\Gamma^{k_3}u)_{x_2}(\Gamma^{k_4}u)_{\xi} |
&\lesssim& |(\Gamma^{k_4}u)_{\xi}\Gamma_6(\Gamma^{k_3}u)_{x_2}| +|\Gamma_5(\Gamma^{k_4}u)_{\xi}(\Gamma^{k_3}u)_{x_2 x_2} | +|x_2\partial_{x_2}(\Gamma^{k_4}u)_{\xi}(\Gamma^{k_3}u)_{x_2 x_2}| \\ \nonumber
&\lesssim& e_s  + (2+\eta)^{\frac{1}{2}}(2+\xi)^{\frac{1}{2}}|\partial_{x_2}(\Gamma^{k_4}u)_{\xi}(\Gamma^{k_3}u)_{x_2 x_2}| \\ \nonumber
&\lesssim& e_s.
\end{eqnarray*}
Therefore, we can obtain
 \begin{eqnarray}\label{3.52}
 J_{k11} &\lesssim &  (1 + e_s)e_s(2+\eta)^{-1}|(\Gamma^{k_2}u)_{\xi}|.
 \end{eqnarray}
Meanwhile, noting Proposition 2, we can get
 \begin{eqnarray}\nonumber
 J_{k12} &=& C(1 + e_s)|(\Gamma^{k_2}u)_{\eta}||Q_{0\xi}(\Gamma^{k_3}u, \Gamma^{k_4}u)| \\ \nonumber
 &\lesssim& (1 + e_s)|(\Gamma^{k_2}u)_{\eta}|[|(\Gamma^{k_3}u)_{\xi\xi}(\Gamma^{k_4}u)_{\eta}| + |(\Gamma^{k_3}u)_{\xi}(\Gamma^{k_4}u)_{\xi\eta} | + |(\Gamma^{k_3}u)_{x_2\xi}(\Gamma^{k_4}u)_{x_2}| ]
 \\ \nonumber
 &\lesssim& (1 + e_s)|(\Gamma^{k_2}u)_{\eta}|[(2+\xi)^{-\frac{3}{4}}(2+\eta)^{\frac{1}{4}} e^{\frac{1}{2}}_s(2+\xi)^{\frac{1}{4}}(2+\eta)^{-\frac{3}{4}}e^{\frac{1}{2}}_s
 +(2+\xi)^{-\frac{1}{2}}(2+\eta)^{-\frac{1}{2}}e_s]\\ \label{3.53}
 &\lesssim& (1 + e_s)e_s(2+\xi)^{-\frac{1}{2}}(2+\eta)^{-\frac{1}{2}}|(\Gamma^{k_2}u)_{\eta}|.
\end{eqnarray}
and noting the estimate of $k_{k11}$
 \begin{eqnarray}\nonumber
 J_{k13} &\lesssim & (1 + e_s)|(\Gamma^{k_2}u)_{x_2}||Q_{0}(\Gamma^{k_3}u_{x_2}, \Gamma^{k_4}u)| \\ \nonumber
 &\lesssim& (1 + e_s)(2+\eta)^{-1}|(\Gamma^{k_2}u)_{x_2}||\Gamma\Gamma^{k_3}u_{x_2}||\Gamma\Gamma^{k_4}u| \\ \nonumber
 &\lesssim& (1 + e_s)|(\Gamma^{k_2}u)_{x_2}|(2+\xi)^{-\frac{1}{4}}(2+\eta)^{-\frac{1}{4}}e^{\frac{1}{2}}_s(2+\xi)^{\frac{1}{4}}(2+\eta)^{\frac{1}{4}}e^{\frac{1}{2}}_s
\\ \label{3.54}
 &\lesssim& (1 + e_s)e_s|(2+\eta)^{-1}(\Gamma^{k_2}u)_{x_2}|.
\end{eqnarray}
When $k_1 \geq [\frac{1+s}{2}]$, we have
\begin{eqnarray}\nonumber
 &&|J_{k1}| \lesssim|\sum\limits_{k'_1+k_2+k_3+k_4 = k, k'_1 < k_1, k_3 < k, k_4 < k} \Gamma^{k'_1}(Q_0 + F'u_{\eta}) Q_0(\Gamma^{k_2}u, Q_0(\Gamma^{k_3}u, \Gamma^{k_4}u))|\\ \nonumber
 &&\leq \sum\limits_{k'_1+k_2+k_3+k_4 = k, k'_1 < k_1, k_3 < k, k_4 < k} [|\Gamma^{k'_1}Q_0(u,u)| + |\Gamma^{k'_1}(F'u_{\eta})|]| Q_0(\Gamma^{k_2}u, Q_0(\Gamma^{k_3}u, \Gamma^{k_4}u))|.
\end{eqnarray}
In the following we will estimate the above two parts separately.
\begin{eqnarray}\nonumber
&&\sum\limits_{k'_1+k_2+k_3+k_4 = k, k'_1 < k_1, k_3 < k, k_4 < k} |\Gamma^{k'_1}Q_0(u,u)|
| Q_0(\Gamma^{k_2}u, Q_0(\Gamma^{k_3}u, \Gamma^{k_4}u))|\\ \nonumber
&&= \sum\limits_{k_2+k_3+k_4 +k_5+k_6 \leq k} |(\Gamma^{k_5}u_{\xi}\Gamma^{k_6}u_{\eta} + \Gamma^{k_5}u_{x_2}\Gamma^{k_6}u_{x_2})|
| Q_0(\Gamma^{k_2}u, Q_0(\Gamma^{k_3}u, \Gamma^{k_4}u))|.
\end{eqnarray}
When $k_5 \leq k_6$, noting Lemma 2, we have
\begin{eqnarray}\nonumber
&&|(\Gamma^{k_5}u_{\xi}\Gamma^{k_6}u_{\eta} + \Gamma^{k_5}u_{x_2}\Gamma^{k_6}u_{x_2})|
| Q_0(\Gamma^{k_2}u, Q_0(\Gamma^{k_3}u, \Gamma^{k_4}u))|\\ \nonumber
&\lesssim& |\Gamma^{k_6}u_{\eta}||(\Gamma^{k_5}u_{\xi}| (2+\eta)^{-1}|\Gamma\Gamma^{k_2}u||Q_0(\Gamma \Gamma^{k_3}u, \Gamma^{k_4} u)|\\ \nonumber
&\lesssim& |\Gamma^{k_6}u_{\eta}||(\Gamma^{k_5}u_{\xi}| (2+\eta)^{-1}|\Gamma\Gamma^{k_2}u|[|\Gamma \Gamma^{k_3}u_{\xi} \Gamma^{k_4} u_{\eta}|+|\Gamma \Gamma^{k_3}u_{\eta} \Gamma^{k_4} u_{\xi}|+|\Gamma \Gamma^{k_3}u_{x_2} \Gamma^{k_4} u_{x_2}|]|\\ \nonumber
&\lesssim& |\Gamma^{k_6}u_{\eta}|(2+\xi)^{-\frac{3}{4}}(2+\eta)^{-\frac{3}{4}}[(2+\xi)^{\delta}e_s^{\frac{3}{2}}\tilde{e}_s + (2+\xi)^{\frac{1}{4}}(2+\eta)^{\frac{1}{4}}(2+\xi)^{-\frac{1}{2}}(2+\eta)^{-\frac{1}{2}}e_s^{2}] \\  \label{3.55}
&\lesssim&  |\Gamma^{k_6}u_{\eta}|(2+\xi)^{-\frac{1}{2}}(2+\eta)^{-\frac{3}{4}}[e_s^{\frac{3}{2}}\tilde{e}_s + e_s^{2}].
\end{eqnarray}
For the case $|k_5|\geq |k_6|$, by Corollary 1 and Lemma 3, we have
\begin{eqnarray}\nonumber
&&|(\Gamma^{k_5}u_{\xi}\Gamma^{k_6}u_{\eta} + \Gamma^{k_5}u_{x_2}\Gamma^{k_6}u_{x_2})|
| Q_0(\Gamma^{k_2}u, Q_0(\Gamma^{k_3}u, \Gamma^{k_4}u))|\\ \nonumber
&\lesssim& |(\Gamma^{k_5}u_{\xi}| (2+\eta)^{-\frac{1}{2}}e_s^{\frac{1}{2}}(2+\eta)^{-1}|\Gamma\Gamma^{k_2}u||Q_0(\Gamma \Gamma^{k_3}u, \Gamma^{k_4} u)|\\ \nonumber
&\lesssim& |(\Gamma^{k_5}u_{\xi}| (2+\eta)^{-\frac{3}{2}}e_s^{\frac{1}{2}}|\Gamma\Gamma^{k_2}u|[|\Gamma \Gamma^{k_3}u_{\xi} \Gamma^{k_4} u_{\eta}|+|\Gamma \Gamma^{k_3}u_{\eta} \Gamma^{k_4} u_{\xi}|+|\Gamma \Gamma^{k_3}u_{x_2} \Gamma^{k_4} u_{x_2}|]|\\ \nonumber
&\lesssim& |\Gamma^{k_5}u_{\xi}|(2+\eta)^{-\frac{3}{2}}e_s^{\frac{1}{2}}[(2+\xi)^{-\frac{3}{4}}(2+\eta)^{\frac{1}{4}}(2+\eta)^{-\frac{1}{2}}(2+\xi)^{\frac{1}{4}}(2+\eta)^{\frac{1}{4}} + (2+\xi)^{-\frac{1}{4}}(2+\eta)^{-\frac{1}{4}}]e_s^{\frac{3}{2}}\\ \label{3.56}
&\lesssim& |\Gamma^{k_5}u_{\xi}|(2+\eta)^{-2}e_s^2.
\end{eqnarray}
It is easy to get the estimate of second part $ |\Gamma^{k'_1}(F'u_{\eta})|| Q_0(\Gamma^{k_2}u, Q_0(\Gamma^{k_3}u, \Gamma^{k_4}u))|$. Then, we can get the estimate of $J_{k1}$.

In the following, we will give the estimation of $J_{k2}$. When $k_1 \leq [\frac{k}{2}]$, the other case can be get easily.
\begin{eqnarray*}
 J_{k2} &\leq& \sum\limits_{0\leq k_1 + k_2 = k} |\Gamma^{k_1}(1-H)||\Gamma^{k_2}[F' Q_0(u, u_{\eta})]|\\ \nonumber
 &\lesssim& (1+e_s) \sum\limits_{0\leq k_3 + k_4 = k_2} |\Gamma^{k_3}F'||\Gamma^{k_4}Q_0(u, u_{\eta})|\\ \nonumber
 &\lesssim&  (1+e_s) (2+\xi)^{-1}\sum\limits_{k_5 + k_6 = k_4}|Q_0(\Gamma^{k_5}u, \Gamma^{k_6}u_{\eta})|\\
 &\lesssim& (1+e_s) (2+\xi)^{-1}\sum\limits_{k_5 + k_6 = k_4}[|\Gamma^{k_5}u_{\xi}||\Gamma^{k_6}u_{\eta\eta}| + |\Gamma^{k_5}u_{\eta}|| \Gamma^{k_6}u_{\eta\xi}| + |\Gamma^{k_5}u_{x_2}|| \Gamma^{k_6}u_{\eta x_2}|).
 \end{eqnarray*}
For the case of $|k_5| \geq |k_6| $, noting Corollary 1,
\begin{eqnarray}\nonumber
| \Gamma^{k_6}u_{\eta\eta}| &\lesssim& (2+\eta)^{-1}|(\partial_{\eta} + \Gamma_4 - x_2\partial_{x_2})(\Gamma^{k_6}u)_{\eta}|\\ \nonumber
&\lesssim& (2+\eta)^{-1}[|(\Gamma^{k_6}u)_{\eta\eta}| + |\Gamma_4(\Gamma^{k_6}u)_{\eta}|] + (2+\eta)^{-\frac{1}{2}}(2+\xi)^{\frac{1}{2}}(\Gamma^{k_6}u)_{x_2\eta}\\ \nonumber
&\lesssim& (2+\eta)^{-\frac{3}{2}}e_s^{\frac{1}{2}} + (2+\eta)^{-1}(2+\xi)^{\frac{1}{2}}e_s^{\frac{1}{2}}.
\end{eqnarray}
Then, we can get
\begin{eqnarray}\label{3.57}
 J_{k2} \lesssim(1+e_s) (2+\xi)^{-\frac{1}{2}}(2+\eta)^{-\frac{1}{2}}e_s^{\frac{1}{2}}[|\Gamma^{k_5}u_{\xi}| + |\Gamma^{k_5}u_{\eta}| + |\Gamma^{k_5}u_{x_2}|].
\end{eqnarray}
For the case of $|k_5| \leq |k_6| $, noting Lemma 2.2 and Lemma 2.3, we can have
\begin{eqnarray}\nonumber
 J_{k2}  &\lesssim & (1+e_s) (1+\xi)^{-1}\sum\limits_{k_5 + k_6 = k_4}|Q_0(\Gamma^{k_5}u,\Gamma^{k_6}u_{\eta}|
\\  \nonumber
 &\lesssim& (1+e_s) (2+\xi)^{-1}(2+\eta)^{-1} [|\Gamma\Gamma^{k_5}u_{\xi}||\Gamma\Gamma^{k_6}u_{\eta}|+|\Gamma\Gamma^{k_5}u||\Gamma\Gamma^{k_6}u_{\xi\eta}|] \\ \nonumber
 &\lesssim& (2+\xi)^{-\frac{1}{2}}(2+\eta)^{-\frac{1}{2}}(1+e_s)e_s^{\frac{1}{2}} [|\Gamma\Gamma^{k_6}u_{\eta}| + |\Gamma\Gamma^{k_6}u_{\xi\eta}|].
\end{eqnarray}
Moreover, it is easily to get the estimate
\begin{eqnarray}\label{3.58}
J_{k3} &=& \sum\limits_{0\leq k_1 + k_2 = k} \Gamma^{k_1}(1-H)\Gamma^{k_2}[F''u^2_\eta]\lesssim (1+e_s)(2+\xi)^{-2} (2+\eta)^{-\frac{1}{2}}\sum\limits_{0\leq  k_2 \leq k}|\Gamma^{k_2}u_{\eta}|
\end{eqnarray}
Then, the estimation of $A_4$ can be get
\begin{eqnarray}\label{3.59}
\bar{A}_7 \lesssim (e_s^{\frac{1}{2}} + e_s^{2})E_s.
\end{eqnarray}

 Denote
 $$\bar{A}_8 = \iiint P_{x_2}(x)Q^2_0(u, u_k)d\xi d\eta dx_2  +  \iiint \{4[Pu_{\eta}]_{\xi}  + [Pu_{\xi}]_{\eta}  - 2[Pu_{x_2}]_{x_2}\} Q_0 u_{kx_2} d\xi d\eta dx_2 \doteq \bar{A}_{81} + \bar{A}_{82}$$
 where,
 \begin{eqnarray}\nonumber
 \bar{A}_{81}& =& \iiint P_{x_2}(x)Q^2_0(u, u_k)d\xi d\eta dx_2\\ \nonumber
 &\lesssim & \iiint (2+\xi)^{-\frac{1}{10}}(2+\eta)^{-\frac{1}{10}}|[\frac{(1-H)^2 F' u_{x_2}}{1- 2(1-H)F'u_{\eta}}]_{x_2}|[u^2_{\xi}u^2_{k\eta}+ u^2_{\eta}u^2_{k\xi}+u^2_{x_2}u^2_{kx_2}]d\xi d\eta dx_2\\ \nonumber
 &\lesssim & (1+e_s) \iiint (2+\xi)^{-\frac{11}{10}}(2+\eta)^{-\frac{1}{10}}|H_{x_2} u_{x_2}|[u^2_{\xi}u^2_{k\eta}+ u^2_{\eta}u^2_{k\xi}+u^2_{x_2}u^2_{kx_2}]d\xi d\eta dx_2\\ \nonumber
 &+ & (1+e_s)^2 \iiint (2+\xi)^{-\frac{11}{10}}(2+\eta)^{-\frac{1}{10}}|H_{x_2}u_{x_2}u_{\eta} |[u^2_{\xi}u^2_{k\eta}+ u^2_{\eta}u^2_{k\xi}+u^2_{x_2}u^2_{kx_2}]d\xi d\eta dx_2\\ \nonumber
 &+ & (1+e_s)^3 \iiint (2+\xi)^{-\frac{11}{10}}(2+\eta)^{-\frac{1}{10}}|u_{x_2}u_{x_2\eta}|[u^2_{\xi}u^2_{k\eta}+ u^2_{\eta}u^2_{k\xi}+u^2_{x_2}u^2_{kx_2}]d\xi d\eta dx_2\\
 &\lesssim & (1+e_s) \iiint (2+\xi)^{-\frac{11}{10}}(2+\eta)^{-\frac{1}{10}}|(Q_{0x_2}+4F'(\xi)u_{\eta x_2}) u_{x_2}|[u^2_{\xi}u^2_{k\eta}+ u^2_{\eta}u^2_{k\xi}+u^2_{x_2}u^2_{kx_2}]d\xi d\eta dx_2\\ \nonumber
 &+ & (1+e_s)^2 \iiint (2+\xi)^{-\frac{11}{10}}(2+\eta)^{-\frac{1}{10}}|(Q_{0x_2}+4F'(\xi)u_{\eta x_2})u_{x_2}u_{\eta} |[u^2_{\xi}u^2_{k\eta}+ u^2_{\eta}u^2_{k\xi}+u^2_{x_2}u^2_{kx_2}]d\xi d\eta dx_2\\ \nonumber
 &+ & (1+e_s)^3 \iiint (2+\xi)^{-\frac{11}{10}}(2+\eta)^{-\frac{1}{10}}|u_{x_2}u_{x_2\eta}|[u^2_{\xi}u^2_{k\eta}+ u^2_{\eta}u^2_{k\xi}+u^2_{x_2}u^2_{kx_2}]d\xi d\eta dx_2
 \end{eqnarray}
 \begin{eqnarray} 
\nonumber
 &\lesssim & (1+e_s) \iiint (2+\xi)^{-\frac{11}{10}}(2+\eta)^{-\frac{1}{10}}|Q_{0x_2} u_{x_2}|[u^2_{\xi}u^2_{k\eta}+ u^2_{\eta}u^2_{k\xi}+u^2_{x_2}u^2_{kx_2}]d\xi d\eta dx_2 + (1+e_s)e^2_s E_s \\ \nonumber
 &+ & (1+e_s)^2 \iiint (2+\xi)^{-\frac{11}{10}}(2+\eta)^{-\frac{1}{10}}|Q_{0x_2}u_{x_2}u_{\eta} |[u^2_{\xi}u^2_{k\eta}+ u^2_{\eta}u^2_{k\xi}+u^2_{x_2}u^2_{kx_2}]d\xi d\eta dx_2 +   (1+e_s)^2e^{\frac{5}{2}}_s E_s \\ \nonumber
  &+& (1+e_s)^3e^2_s E_s\\ \nonumber
&\lesssim & (1+e_s) \iiint (2+\xi)^{-\frac{11}{10}}(2+\eta)^{-\frac{1}{10}}(|u_{\xi x_2}u_{\eta}|+|u_{\eta x_2}u_{\xi}| +|u_{x_2 x_2}u_{x_2}|)| u_{x_2}|[u^2_{\xi}u^2_{k\eta}+ u^2_{\eta}u^2_{k\xi}+u^2_{x_2}u^2_{kx_2}]d\xi d\eta dx_2  \\ \nonumber
 &+ & (1+e_s)^2 \iiint (2+\xi)^{-\frac{11}{10}}(2+\eta)^{-\frac{1}{10}}(|u_{\xi x_2}u_{\eta}|+|u_{\eta x_2}u_{\xi}| +|u_{x_2 x_2}u_{x_2}|)|u_{x_2}u_{\eta} |[u^2_{\xi}u^2_{k\eta}+ u^2_{\eta}u^2_{k\xi}+u^2_{x_2}u^2_{kx_2}]d\xi d\eta dx_2 \\ \nonumber
  &+& (1+e_s)^3e^2_s E_s + (1+e_s)^2e^{\frac{5}{2}}_s E_s\\ \nonumber
&\lesssim & (1+e_s) \iiint (2+\xi)^{-\frac{11}{10}}(2+\eta)^{-\frac{1}{10}}(|u_{\xi x_2}u_{\eta}|+|u_{\eta x_2}u_{\xi}| +|u_{x_2 x_2}u_{x_2}|)| u_{x_2}|[u^2_{\xi}u^2_{k\eta}+ u^2_{\eta}u^2_{k\xi}+u^2_{x_2}u^2_{kx_2}]d\xi d\eta dx_2  \\ \nonumber
 &+ & (1+e_s)^2 \iiint (2+\xi)^{-\frac{11}{10}}(2+\eta)^{-\frac{1}{10}}(|u_{\xi x_2}u_{\eta}|+|u_{\eta x_2}u_{\xi}| +|u_{x_2 x_2}u_{x_2}|)|u_{x_2}u_{\eta} |[u^2_{\xi}u^2_{k\eta}+ u^2_{\eta}u^2_{k\xi}+u^2_{x_2}u^2_{kx_2}]d\xi d\eta dx_2 \\ \nonumber
  &+& (1+e_s)^3e^2_s E_s + (1+e_s)^2e^{\frac{5}{2}}_s E_s \\ \nonumber
&\lesssim & (1+e_s)^2 \iiint (2+\xi)^{-\frac{11}{10}}(2+\eta)^{-\frac{1}{10}}[|u_{\eta x_2}|| u_{x_2}||u_{\xi}|^3  + |u_{\eta x_2}||u_{x_2}u_{\eta}||u_{\xi}|^3]u^2_{k\eta}d\xi d\eta dx_2 \\ \nonumber
&+& (1+e_s)^3e^2_s E_s + (1+e_s)^2e^{\frac{5}{2}}_s E_s.
 \end{eqnarray}
 By Proposition 2 , Proposition 3 and Corollary 1, we have
\begin{eqnarray}\nonumber
&&\iiint (2+\xi)^{-\frac{11}{10}}(2+\eta)^{-\frac{1}{10}}[|u_{\eta x_2}|| u_{x_2}| + |u_{\eta x_2}||u_{x_2}u_{\eta}|]|u_{\xi}|^3 u^2_{k\eta}d\xi d\eta dx_2 \\ \nonumber
&\lesssim&\iiint (2+\xi)^{-\frac{27}{20}}(2+\eta)^{-\frac{1}{10}} e^{\frac{5}{2}}_s u^2_{k\eta}d\xi d\eta dx_2 \lesssim  e^{\frac{5}{2}}_s E_s.
 \end{eqnarray}
Then,
\begin{equation}\label{3.60}
\bar{A}_{81}\lesssim (1+e_s)^2 e^{\frac{5}{2}}_s E_s + (1+e_s)^3e^2_s E_s.
\end{equation}

\begin{eqnarray*}
\bar{A}_{82} &=&  \iiint \{4[Pu_{\eta}]_{\xi}  + [Pu_{\xi}]_{\eta}  - 2[Pu_{x_2}]_{x_2}\} Q_0(u_k, u) u_{kx_2} d\xi d\eta dx_2 \\
&\lesssim&  \iiint |[P_{\xi} u_{\eta} + P_{\eta}u_{\xi} + P_{x_2}u_{x_2}]|| Q_0(u_k, u) u_{kx_2} |d\xi d\eta dx_2
+  \iiint |P[8u_{\eta \xi}  - 2u_{x_2x_2}] Q_0(u_k, u) u_{kx_2}| d\xi d\eta dx_2 \\
&\doteq& \bar{A}_{821} + \bar{A}_{822}.
\end{eqnarray*}
Noting Proposition 2,
\begin{eqnarray}\nonumber
\bar{A}_{822} &\lesssim& (1+e_s)^2\iiint (2+\xi)^{-\frac{1}{10}}(2+\eta)^{-\frac{1}{10}} |F'||u_{x_2}|(|u_{\eta\xi}| +|u_{x_2x_2}|)|u_{\eta}||u_{k\xi} u_{kx_2}| d\xi d\eta dx_2 \\ \nonumber
&+& (1+e_s)^2\iiint  (2+\xi)^{-\frac{1}{10}}(2+\eta)^{-\frac{1}{10}}|F'||u_{x_2}|(|u_{\eta \xi}| +|u_{x_2 x_2}|)|u_{\xi}||u_{k\eta}u_{kx_2}| d\xi d\eta dx_2 \\ \nonumber
&+& (1+e_s)^2\iiint (2+\xi)^{-\frac{1}{10}}(2+\eta)^{-\frac{1}{10}} |F'||u_{x_2}|(|u_{\eta \xi}| +|u_{x_2x_2}|)||u_{x_2}|u^2_{kx_2}| d\xi d\eta dx_2\\ \nonumber
&\lesssim& (1+e_s)^2 \iiint (2+\xi)^{-\frac{11}{10}}(2+\eta)^{-\frac{1}{10}} (2+\xi)^{-\frac{3}{4}}(2+\eta)^{-\frac{3}{4}}e^{\frac{3}{2}}_s (|u_{k\xi} u_{kx_2}| + u^2_{kx_2}) d\xi d\eta dx_2 \\ \nonumber
&+& \iiint  (1+e_s)^2(2+\xi)^{-\frac{11}{10}}(2+\eta)^{-\frac{1}{10}} (2+\xi)^{-\frac{1}{2}}(2+\eta)^{-\frac{1}{2}}e_s |u_{\xi}||u_{k\eta}u_{kx_2}| d\xi d\eta dx_2\} \\ \label{3.61}
&\lesssim& (1+e_s)^2e^{\frac{3}{2}}_s E_s + (1+e_s)^2e_s \iiint  (2+\xi)^{-\frac{11}{10}}(2+\eta)^{-\frac{1}{10}} |u_{\xi x_2}||u_{k\eta}u_{kx_2}| d\xi d\eta dx_2\}
\lesssim (1+e_s)^2e^{\frac{3}{2}}_s E_s.
\end{eqnarray}
In the following, we will estimate $\bar{A}_{821}$.
\begin{eqnarray*}
\bar{A}_{821} &\lesssim& \iiint |[(2+\xi)^{-\frac{1}{10}}(2+\eta)^{-\frac{1}{10}}\frac{(1-H)^2 F' u_{x_2}}{1- 2(1-H)F'u_{\eta}}]_{\xi} u_{\eta} Q_0(u_k, u) u_{kx_2}| d\xi d\eta dx_2 \\ \nonumber
&+& \iiint  |[(2+\xi)^{-\frac{1}{10}}(2+\eta)^{-\frac{1}{10}}\frac{(1-H)^2 F' u_{x_2}}{1- 2(1-H)F'u_{\eta}}]_{\eta}u_{\xi} Q_0(u_k, u) u_{kx_2} |d\xi d\eta dx_2 \\ \nonumber
&+& \iiint  (2+\xi)^{-\frac{1}{10}}(2+\eta)^{-\frac{1}{10}}|[\frac{(1-H)^2 F' u_{x_2}}{1- 2(1-H)F'u_{\eta}}]_{x_2}u_{x_2} Q_0(u_k, u) u_{kx_2} |d\xi d\eta dx_2 \\ \nonumber
&\lesssim& (1+e_s)^2\iiint (2+\xi)^{-\frac{11}{10}}(2+\eta)^{-\frac{1}{10}}|u_{x_2}||u_{\eta}| |Q_0(u_k, u)|| u_{kx_2}| d\xi d\eta dx_2 \\ \nonumber
&+& (1+e_s)^3\iiint (2+\xi)^{-\frac{11}{10}}(2+\eta)^{-\frac{1}{10}}[(|H_{\xi}u_{x_2}| + |u_{\xi x_2}|)(1 + |u_{\eta}|)+ |u_{x_2}|(|H_{\xi}u_{\eta}| + |u_{\xi\eta}|)] |u_{\eta}| |Q_0(u_k, u) ||u_{kx_2}| d\xi d\eta dx_2 \\ \nonumber
&+& (1+e_s)^2\iiint  [(2+\xi)^{-\frac{11}{10}}(2+\eta)^{-\frac{11}{10}}|u_{x_2}||u_{\xi}| |Q_0(u_k, u)|| u_{kx_2}| d\xi d\eta dx_2 \\ \nonumber
&+& (1+e_s)^3\iiint  (2+\xi)^{-\frac{11}{10}}(2+\eta)^{-\frac{1}{10}}[(|H_{\eta}u_{x_2}| + |u_{x_2\eta}|)(1+|u_{\eta}|)+ |u_{x_2}||(H_{\eta}u_{\eta}+u_{\eta\eta})|]|u_{\xi}|| Q_0(u_k, u) ||u_{kx_2}| d\xi d\eta dx_2 \\ \nonumber
&+& (1+e_s)^3\iiint  (2+\xi)^{-\frac{11}{10}}(2+\eta)^{-\frac{1}{10}}[(|H_{x_2}u_{x_2}| + |u_{x_2x_2}|)(1+ |u_{\eta}|) + (|H_{x_2}u_{\eta}| + |u_{x_2\eta})u_{x_2}|]|u_{x_2}|| Q_0(u_k, u)|| u_{kx_2}| d\xi d\eta dx_2.
\end{eqnarray*}
Noting Proposition 2, we can estimate the above inequality terms by terms.
\begin{eqnarray}\nonumber
\bar{A}_{821} &\lesssim& (1+e_s)^2e^{\frac{3}{2}}_s E_s + (1+e_s)^2\iiint  [(2+\xi)^{-\frac{21}{10}}(2+\eta)^{-\frac{11}{10}}|u_{x_2}||u_{\xi}||u_{\xi}||u_{k\eta} u_{kx_2}| d\xi d\eta dx_2\\ \nonumber
&+& (1+e_s)^3\iiint (2+\xi)^{-\frac{21}{10}}(2+\eta)^{-\frac{1}{10}}[(|H_{\xi}u_{x_2}| + |u_{\xi x_2}|) |u_{\xi}|| u_{k\eta} u_{kx_2}| d\xi d\eta dx_2 \\ \nonumber
&+& (1+e_s)^3(e^{\frac{3}{2}}_s + e^2_s)E_s + (1+e_s)^3(e_s + e^{\frac{3}{2}}_s + e^2_s)E_s \\ \nonumber
&\lesssim& (1+e_s)^3(e_s + e^{\frac{3}{2}}_s + e^2_s)E_s + (1+e_s)^2\iiint  [(2+\xi)^{-\frac{21}{10}}(2+\eta)^{-\frac{11}{10}} (2+\xi)^{\frac{1}{4}}(2+\eta)^{\frac{1}{4}}e^{\frac{3}{2}}_s |u_{k\eta} u_{kx_2}| d\xi d\eta dx_2\\ \nonumber
&+& (1+e_s)^3\iiint (2+\xi)^{-\frac{11}{10}}(2+\eta)^{-\frac{1}{10}}(|(Q_{0}+4F'(\xi)u_{\eta})_{\xi}u_{x_2}||u_{\xi}| + e_s) |u_{k\eta} u_{kx_2}| d\xi d\eta dx_2 \\ \nonumber
&\lesssim& (1+e_s)^3(e_s + e^{\frac{3}{2}}_s + e^2_s)E_s + (1+e_s)^3\iiint (2+\xi)^{-\frac{11}{10}}(2+\eta)^{-\frac{1}{10}}(e^2_s + e^{\frac{3}{2}}_s) |u_{k\eta} u_{kx_2}| d\xi d\eta dx_2 \\  \label{3.62}
&\lesssim& (1+e_s)^3(e_s + e^{\frac{3}{2}}_s + e^2_s)E_s.
\end{eqnarray}
Therefore, we can get
\begin{eqnarray}\label{3.63}
\bar{A}_8 \lesssim (1+e_s)^3(e_s + e^{\frac{3}{2}}_s + e^2_s)E_s.
\end{eqnarray}
Denote
$$\bar{A}_9 \doteq \iiint [PF'u_{\xi}]_{x_2}u^2_{k\eta} - 8[PF'u_{\xi}]_{\eta}u_{k\eta} u_{kx_2}d\xi d\eta dx_2 = \bar{A}_{91} + \bar{A}_{92}.$$
By Proposition 2, it is easily to get
\begin{eqnarray}\nonumber
\bar{A}_{91} &=& \iiint [PF'u_{\xi}]_{x_2}u^2_{k\eta} d\xi d\eta dx_2 \\ \nonumber
&\lesssim &\iiint (2+\xi)^{-\frac{21}{10}}(2+\eta)^{-\frac{1}{10}} |[\frac{(1-H)^2u_{x_2}u_{\xi}}{1- 2(1-H)F'u_{\eta}}]_{x_2}|u^2_{k\eta}d\xi d\eta dx_2 \\ \nonumber
&\lesssim & (1+e_s)\iiint (2+\xi)^{-\frac{21}{10}}(2+\eta)^{-\frac{1}{10}} |[(H_{x_2}u_{x_2}u_{\xi} + u_{x_2x_2}u_{\xi} + u_{x_2}u_{\xi x_2})(1+u_{\eta})+ u_{x_2}u_{\xi}u_{x_2\eta}]|u^2_{k\eta}d\xi d\eta dx_2 \\ \label{3.64}
&\lesssim & (1+e_s)(e_s + e^{\frac{3}{2}}_s + e^2_s + e^{\frac{5}{2}}_s)E_s.
\end{eqnarray}
and
\begin{eqnarray}\nonumber
\bar{A}_{92} &=& \iiint -8[PF'u_{\xi}]_{\eta}u_{k\eta} u_{kx_2}d\xi d\eta dx_2\\ \nonumber
&\lesssim &\iiint (2+\xi)^{-\frac{21}{10}}(2+\eta)^{-\frac{11}{10}} |\frac{(1-H)^2u_{x_2}u_{\xi}}{1- 2(1-H)F'u_{\eta}}| |u_{kx_2}u_{k\eta}| d\xi d\eta dx_2 \\ \nonumber
&+ &\iiint (2+\xi)^{-\frac{21}{10}}(2+\eta)^{-\frac{1}{10}} |[\frac{(1-H)^2u_{x_2}u_{\xi}}{1- 2(1-H)F'u_{\eta}}]_{\eta}||u_{kx_2}u_{k\eta}| d\xi d\eta dx_2 \\ \nonumber
&\lesssim & (1+e_s)^2\iiint (2+\xi)^{-\frac{21}{10}}(2+\eta)^{-\frac{11}{10}} |u_{x_2}u_{\xi}| |u_{kx_2}u_{k\eta}| d\xi d\eta dx_2 \\ \nonumber
&+ & (1+e_s)^3\iiint (2+\xi)^{-\frac{21}{10}}(2+\eta)^{-\frac{1}{10}} [(|H_{\eta}u_{x_2}u_{\xi}| + |u_{x_2\eta}u_{\xi}| + |u_{x_2}u_{\xi\eta}|)(1+|u_{\eta}|)+ |u_{x_2}u_{\xi}u_{\eta\eta}|]|u_{kx_2}u_{k\eta}|d\xi d\eta dx_2 \\ \label{3.65}
&\lesssim & (1+e_s)^2e_s E_s + (1+e_s)^3(e_s + e^{\frac{3}{2}}_s + e^2_s + e^{\frac{5}{2}}_s)E_s.
\end{eqnarray}
Therefore, we can obtain the estimate
\begin{eqnarray}\label{3.66}
\bar{A}_9 \lesssim  (1+e_s)^2e_s E_s + (1+e_s)^3(e_s + e^{\frac{3}{2}}_s + e^2_s + e^{\frac{5}{2}}_s)E_s.
\end{eqnarray}
Using the similar method to estimate $\bar{A}_9$, we can get
\begin{eqnarray}\nonumber
\bar{A}_{10} &\doteq& \iiint  [PF'u_{\eta}]_{x_2}u^2_{kx_2}d\xi d\eta dx_2 + \iiint [P F'u_{x_2}]_{\eta}u^2_{kx_2}d\xi d\eta dx_2 \\ \label{3.67}
&\lesssim &(1+e_s)^2e_s E_s + (1+e_s)^3(e_s + e^{\frac{3}{2}}_s + e^2_s + e^{\frac{5}{2}}_s)E_s.
\end{eqnarray}
We denote
\begin{eqnarray} \label{3.68}
 \bar{A}_{11} = \iiint (2+\xi)^{-\frac{1}{10}}(2+\eta)^{-\frac{1}{10}} \frac{(1-H)^2u_{x_2}u_{\xi}}{1- 2(1-H)F'u_{\eta}} J_k u_{kx_2}d\xi d\eta dx_2.
\end{eqnarray}
Using the similar method to estimate $\bar{A}_4$, we can get the estimate of $\bar{A}_{11}$.
Similarly, we can get the estimate
\begin{eqnarray}\nonumber
\bar{A}_{12} &=& \iiint (2+\xi)^{-\frac{1}{10}}(2+\eta)^{-\frac{1}{10}}u_{k\xi} (1-H)4\Gamma^k(F'^2u_{\eta\eta}H)d\xi d\eta dx_2 \\ \nonumber
&\lesssim& \iiint (2+\xi)^{-\frac{21}{10}}(2+\eta)^{-\frac{1}{10}}u_{k\xi}u_{k\eta\eta}H d\xi d\eta dx_2 + \iiint (2+\xi)^{-\frac{21}{10}}(2+\eta)^{-\frac{1}{10}}u_{k\xi}u_{k\eta}(\Gamma H) d\xi d\eta dx_2\\ \nonumber
 &+& \iiint (2+\xi)^{-\frac{21}{10}}(2+\eta)^{-\frac{1}{10}}u_{k\xi}u_{\eta\eta}(Q_0(u, u_k) + F'u_{k\eta}) d\xi d\eta dx_2\\ \nonumber
&\lesssim & \varepsilon + (e_s + \tilde{e}_se^{\frac{1}{2}}_s + e^2_s)E_s
\end{eqnarray}
It is easily to get the estimate
\begin{eqnarray}
\bar{A}_{13} &=& - \iiint P[4 \Gamma^k(F'^2 u_{\eta\eta})]u_{kx_2} d\xi d\eta dx_2 \\ \nonumber
&\lesssim & \varepsilon + (e_s^{\frac{1}{2}} + e_s + e_s^{\frac{3}{2}}) E_s
\end{eqnarray}
and
\begin{eqnarray}
\bar{A}_{14} &=&\iiint P[8 F' u_{\eta}\Gamma^k(F'^2 u_{\eta\eta})] u_{kx_2} d\xi d\eta dx_2 \\ \nonumber
&\lesssim & \varepsilon + (\tilde{e}_se_s^{\frac{1}{2}} + e_s + e_s^{\frac{3}{2}})E_s
\end{eqnarray}

\subsection{Low order energy estimates}
In this subsection, we will give the low order energy estimates. Firstly, we can get the following important lower order $L^\infty$ estimate.
\begin{proposition} If we suppose that $e_s \lesssim \varepsilon$, we can get
\begin{eqnarray}\label{3.69}
\tilde{e}_s\lesssim \varepsilon.
\end{eqnarray}
\begin{proof}
Using the foundational solution of system (\ref{3.3}), we have
\begin{eqnarray*}
|\Gamma^k u| &=& |\Gamma^k u_0 + \int_0^t\iint_{R^2} \frac{\Gamma^k [\frac{Q_0(u+F, Q_0(u, u) + 2F'(u_t - u_{x_1}))}{1 - Q_0(u, u) -2 F'(u_t - u_{x_1})}]}{\sqrt{|t-t'|^2 - |x-x'|^2}}d x' dt'|\\
&=& \int_0^t\iint_{R^2} \frac{\Gamma^k [\frac{Q_0(u, Q_0(u, u) + 2F'(u_t - u_{x_1}))}{1 - Q_0(u, u) -2 F'(u_t - u_{x_1})}]}{\sqrt{|t-t'|^2 - |x-x'|^2}}d x' dt' + \int_0^t\iint_{R^2} \frac{\Gamma^k [\frac{Q_0(F, Q_0(u, u))}{1 - Q_0(u, u) -2 F'(u_t - u_{x_1})}]}{\sqrt{|t-t'|^2 - |x-x'|^2}}d x' dt' \\
&+& \int_0^t\iint_{R^2} \frac{|\Gamma^k [\frac{Q_0(F, F'(u_t - u_{x_1}))}{1 - Q_0(u, u) -2 F'(u_t - u_{x_1})}]|}{\sqrt{|t-t'|^2 - |x-x'|^2}}dx' dt'
\end{eqnarray*}
\begin{eqnarray*}
&\lesssim& |\Gamma^k u_0 + \int_0^t\iint_{R^2} \frac{\Gamma^k [\frac{Q_0(u, Q_0(u, u) + 2F'(u_t - u_{x_1}))}{1 - Q_0(u, u) -2 F'(u_t - u_{x_1})}]}{\sqrt{|t-t'| + |x-x'|}}|x-x'|d \sqrt{|t-t'| - |x-x'|} d|t-t'||\\
&+&|\int_0^t\iint_{R^2} \frac{\Gamma^k [\frac{Q_0(F, Q_0(u, u))}{1 - Q_0(u, u) -2 F'(u_t - u_{x_1})}]}{\sqrt{|t-t'| + |x-x'|}}|x-x'|d \sqrt{|t-t'| - |x-x'|} d|t-t'||\\
&+& \int_0^t\iint_{R^2} \frac{|\Gamma^k [\frac{Q_0(F, F'(u_t - u_{x_1}))}{1 - Q_0(u, u) -2 F'(u_t - u_{x_1})}]|}{\sqrt{|t-t'|^2 - |x-x'|^2}}dx' dt'\\
&\lesssim& \varepsilon + \int_0^t\iint_{R^2} |\Gamma^k [\frac{Q_0(u, Q_0(u, u) + 2F'(u_t - u_{x_1}))}{1 - Q_0(u, u) -2 F'(u_t - u_{x_1})}]|d \sqrt{|t-t'| - |x-x'|} d|t-t'|\\
&+&|\int_0^t\iint_{R^2} \frac{\Gamma^k [\frac{Q_0(F, Q_0(u, u))}{1 - Q_0(u, u) -2 F'(u_t - u_{x_1})}]}{\sqrt{|t-t'| + |x-x'|}}|x-x'|d \sqrt{|t-t'| - |x-x'|} d|t-t'||\\
&+& \int_0^t\iint_{R^2} \frac{|\Gamma^k [\frac{Q_0(F, F'(u_t - u_{x_1}))}{1 - Q_0(u, u) -2 F'(u_t - u_{x_1})}]|}{\sqrt{|t-t'|^2 - |x-x'|^2}}dx' dt'.
\end{eqnarray*}
By integrating in part, we have
\begin{eqnarray*}
|\Gamma^k u|&\lesssim& \varepsilon + \int_0^t\iint_{R^2} \partial |\Gamma^k [\frac{Q_0(u, Q_0(u, u) + 2F'(u_t - u_{x_1}))}{1 - Q_0(u, u) -2 F'(u_t - u_{x_1})}]| dx dt \\
&+&\int_0^t\iint_{R^2} \partial|\Gamma^k [\frac{Q_0(F, Q_0(u, u))}{1 - Q_0(u, u) -2 F'(u_t - u_{x_1})}] | dx dt\\
&+& \int_0^t\iint_{R^2} \frac{|\Gamma^k [\frac{Q_0(F, F'(u_t - u_{x_1}))}{1 - Q_0(u, u) -2 F'(u_t - u_{x_1})}]|}{\sqrt{|t-t'|^2 - |x-x'|^2}}dx dt.
\end{eqnarray*}
Then
\begin{eqnarray}\nonumber
|\Gamma^k u| &\lesssim& \varepsilon + |\iiint  \Gamma |\Gamma^k [(1-H)(Q_0(u, Q_0(u, u)) + 6F'Q_0(u,u_{\eta}) + 8F''u^2_{\eta}) + 4F'^2u_{\eta\eta} H| d\xi d\eta dx_2|\\ \nonumber
&+& \iiint   |\frac{\Gamma^k (F'^2u_{\eta\eta})}{\sqrt{(\xi - \xi')(\eta - \eta')-(x_2-x'_2)^2}} | d\xi d\eta dx_2\\ \nonumber
&\lesssim& \varepsilon + \iiint  |\Gamma^{k+1} [(Q_0(u, Q_0(u, u)) + F'Q_0(u,u_{\eta}) + F''u^2_{\eta})+ F'^2u_{\eta\eta} H| d\xi d\eta dx_2\\ \nonumber
&+& \iiint   |\frac{(2+\xi)^{-\frac{7}{4}}(2+\eta)^{-\frac{7}{4}}e_k^\frac{1}{2}+(2+\xi)^{-\frac{5}{4}}(2+\eta)^{-\frac{5}{4}}e_k^\frac{1}{2}}{\sqrt{(\xi - \xi')(\eta - \eta')-(x_2-x'_2)^2}} | d\xi d\eta dx_2\\ \nonumber
&\lesssim& \varepsilon + \iiint  |\Gamma^{k+1}Q_0(u, Q_0(u, u))| + (2+\xi)^{-2}|\Gamma^{k+1}Q_0(u, u_{\eta})| + (2+\xi)^{-3}|\Gamma^{k+1}u^2_{\eta}| d\xi d\eta dx_2\\ \nonumber
&+& \iiint   |\frac{(2+\xi)^{-\frac{5}{4}}(2+\eta)^{-\frac{5}{4}}e_k^\frac{1}{2}}{\sqrt{(\xi - \xi')(\eta - \eta')-(x_2-x'_2)^2}} | d\xi d\eta dx_2\\ \nonumber
&\lesssim& \varepsilon + \iiint  (2+\xi)^{-1}(2+\xi +\eta)^{-1}[|\nabla\Gamma^{k_1}u||\nabla\Gamma^{k_2}u||\Gamma^{k_3}u| + |\Gamma^{k'_1}u||\Gamma^{k'_2}u||\nabla\nabla\Gamma^{k_3}u|] d\xi d\eta dx_2  \\ \nonumber &+&  \iiint (2+\xi)^{-1}(2+\eta)^{-1}(2+\xi)^{\frac{1}{2}}(2+\eta)^{\frac{1}{2}}|\Gamma^{k_1}u_{\xi x_2}||\Gamma^{k_2}u_{\eta}|d\xi d\eta dx_2 \\ \nonumber
&+& \iiint  (2+\xi)^{-2}|\Gamma^{k_1}u_{\eta}||\Gamma^{k_2}u_{\eta}| d\xi d\eta dx_2 \\ \nonumber
&\lesssim& \varepsilon + \iiint  (2+\xi)^{-1}(2+\xi + \eta)^{-1}[|\nabla\Gamma^{k_1}u||\nabla\Gamma^{k_2}u||\Gamma^{k_3}u| + |\Gamma^{k'_1}u||\Gamma^{k'_2}u||\nabla\nabla\Gamma^{k_3}u|] d\xi d\eta dx_2  \\ \nonumber &+&  \iiint (2+\xi)^{-1}(2+\eta)^{-1}(2+\xi)^{\frac{1}{2}}(2+\eta)^{\frac{1}{2}}|\Gamma^{k_1}u_{\xi x_2}||\Gamma^{k_2}u_{\eta}|d\xi d\eta dx_2 +  e_k^\frac{1}{2}\\ \nonumber
&+& \iiint  (2+\xi)^{-2}|\Gamma^{k_1}u_{\eta}||\Gamma^{k_2}u_{\eta}| d\xi d\eta dx_2 \\ \nonumber
&\lesssim& \varepsilon + \int  (2+\xi)^{-1}(2+ \xi +\eta)^{-1}(2+\xi)^{\frac{1}{2}}(2+\eta)^{\frac{1}{2}}(2+\xi)^{\delta}\tilde{e}_k e_k d\xi   +  e_k  + e_k^\frac{1}{2}\\ \nonumber
&\lesssim& \varepsilon + \int (2+\xi)^{\delta - 1}d\xi \tilde{e}_s e_k   +  e_k +e_k^\frac{1}{2}\\ \label{3.70}
&\lesssim&  \varepsilon  + (2+\xi)^{\delta}\tilde{e}_s e_k + e_k + e_k^\frac{1}{2}.
\end{eqnarray}
where we use the bootstrap step in the last two step. Then, we can get the conclusion.
\end{proof}
\end{proposition}

Taking the operator $\Gamma^l$ to (\ref{3.3}), we have
  \begin{eqnarray*}
\Box \Gamma^l u  + 4\Gamma^lF'^2u_{\eta\eta})=\tilde{\Gamma}^l\{\frac{1}{2}(1-H)[Q_0(u, Q_0(u, u))+ 6F' Q_0(u, u_{\eta}) + 8 F'' u_{\eta}^2]- 4 F'^2u_{\eta\eta}H\},
\end{eqnarray*}
where $\tilde{\Gamma}^l = \Gamma^l + \sum\limits_{l'<l}A_{l,l'}\Gamma^{l'}$.
Multiplying $u_{l\eta}e^{-B(\xi)}$ into the above equation and integrating it for the variables $x_2$ and $\eta$, we have
\begin{eqnarray*}
2\frac{d}{d\xi} \iint u^2_{l\eta}e^{-B(\xi)} dx_2d\eta =\iint e^{-B(\xi)} u_{l\eta} \tilde{\Gamma}^l\{\frac{1}{2}(1-H)[Q_0(u, Q_0(u, u))+ 6F' Q_0(u, u_{\eta}) + 8 F'' u_{\eta}^2]- 4 F'^2u_{\eta\eta}H\}dx_2d\eta.
\end{eqnarray*}
Without loss of generality, we assume $|l|\geq 7$. Noting the bound of $B(\xi)$, we can get
\begin{eqnarray*}\nonumber
&&\frac{d}{d\xi} \iint u^2_{l\eta} dx_2d\eta \\ \nonumber
 &\lesssim&  \sum\limits_{|k|\leq s}\iint |u_{k\eta}||\tilde{\Gamma}^{l-7}\{(1-H)[Q_0(u, Q_0(u, u))+ 6F' Q_0(u, u_{\eta}) + 4 F'' u_{\eta}^2]- 4 F'^2u_{\eta\eta}H\}|dx_2d\eta\\
&\lesssim& \sum\limits_{|k|\leq s}\sum\limits_{|l_1|+|l_2|+|l_3|+|l_4|\leq s-14}\iint |u_{k\eta}||([Q_0(\Gamma^{l_2}u, Q_0(\Gamma^{l_3}u, \Gamma^{l_4}u)) + \Gamma^{l_2}F' Q_0(\Gamma^{l_3}u, \Gamma^{l_4}u_{\eta}) ])|dx_2d\eta \\ \nonumber
&+& \sum\limits_{|k|\leq s}\sum\limits_{|l_1|+|l_2|+|l_3|+|l_4|\leq s-14}\iint |u_{k\eta}|(|\Gamma^{l_2}F'' \Gamma^{l_3}u_{\eta}\Gamma^{l_4}u_{\eta}| +| F'^2\Gamma^{l_3}u_{\eta\eta}\Gamma^{l_4}H|\})dx_2d\eta.
\end{eqnarray*}
Then, integrating $\xi$ from the Goursat boundary to $\xi$  and noting Lemma 2.2, we obtain
\begin{eqnarray} \nonumber
 &&\iint u^2_{l\eta} dx_2d\eta \\ \nonumber
 &\lesssim& \varepsilon +  (1+e_s)E^{\frac{1}{2}}_s (\sum\limits_{|l_2|+|l_3|+|l_4|\leq s-14}\iiint(2+\xi)^{\frac{11}{10}}(2+\eta)^{\frac{1}{10}}[Q_0(\Gamma^{l_2}u, Q_0(\Gamma^{l_3}u, \Gamma^{l_4}u))]^2dx_2d\eta d\xi )^{\frac{1}{2}}\\ \nonumber
&+& (1+e_s)E^{\frac{1}{2}}_s (\sum\limits_{|l_1|+|l_2|+|l_3|+|l_4|\leq s-14}\iiint(2+\xi)^{-\frac{9}{10}}(2+\eta)^{\frac{1}{10}}[ Q_0(\Gamma^{l_3}u, \Gamma^{l_4}u_{\eta}) ]^2dx_2d\eta d\xi )^{\frac{1}{2}}\\ \nonumber
&+& (1+e_s)E^{\frac{1}{2}}_s (\sum\limits_{|l_1|+|l_2|+|l_3|+|l_4|\leq s-14}\iiint(2+\xi)^{-\frac{19}{10}}(2+\eta)^{\frac{1}{10}}  (\Gamma^{l_3}u_{\eta}\Gamma^{l_4}u_{\eta})^2dx_2d\eta d\xi )^{\frac{1}{2}}\\ \nonumber
&+& E^{\frac{1}{2}}_s (\sum\limits_{|l_1|+|l_2|+|l_3|+|l_4|\leq s-14}\iiint(2+\xi)^{-\frac{19}{10}}(2+\eta)^{\frac{1}{10}}  (\Gamma^{l_3}u_{\eta\eta}\Gamma^{l_4}H)^2dx_2d\eta d\xi )^{\frac{1}{2}}\\ \nonumber
&\lesssim& \varepsilon +  (1+e_s)E^{\frac{1}{2}}_s (\sum\limits_{|l_2|+|l_3|+|l_4|\leq s-14}\iiint(2+\xi)^{-\frac{9}{10}}(2+\eta)^{\frac{1}{10}}[|\nabla\Gamma^{l_2}u|^2|\Gamma Q_0(\Gamma^{l_3}u, \Gamma^{l_4}u)|^2+ |\Gamma\Gamma^{l_2}u|^2| \nabla Q_0(\Gamma^{l_3}u, \Gamma^{l_4}u)|^2]dx_2d\eta d\xi )^{\frac{1}{2}}\\ \nonumber
&+& (1+e_s)e^{\frac{1}{2}}_sE^{\frac{1}{2}}_s (\sum\limits_{|l_3|+|l_4|\leq s-14}\iiint(2+\xi)^{-\frac{9}{10}}(2+\eta)^{\frac{1}{10}}\{ (2+\eta)^{-2}(\Gamma^{l_3}u)^2_{\xi} + (2+\xi)^{-\frac{1}{2}}(2+\eta)^{-\frac{1}{2}} [(\Gamma^{l_3}u)^2_{\eta} + (\Gamma^{l_3}u)^2_{x_2}]\}dx_2d\eta d\xi )^{\frac{1}{2}}\\ \nonumber
&+& (1+e_s)e^{\frac{1}{2}}_sE^{\frac{1}{2}}_s (\sum\limits_{|l_3|+|l_4|\leq s-14}\iiint(2+\xi)^{-\frac{19}{10}}(2+\eta)^{-\frac{9}{10}}  (\Gamma^{l_4}u_{\eta})^2dx_2d\eta d\xi )^{\frac{1}{2}}\\\label{3.71}
&+& E^{\frac{1}{2}}_s (\sum\limits_{|l_1|+|l_2|+|l_3|+|l_4|\leq s-14}\iiint(2+\xi)^{-\frac{19}{10}}(2+\eta)^{\frac{1}{10}}  (\Gamma^{l_3}u_{\eta\eta}\Gamma^{l_4}(Q_0(u,u) + F''u^2_{\eta}))^2dx_2d\eta d\xi )^{\frac{1}{2}}\\ \nonumber
&\lesssim& \varepsilon +  (1+e_s)E^{\frac{1}{2}}_s (\sum\limits_{|l_2|+|l'_3|+|l'_4|\leq s-13}\iiint(2+\xi)^{-\frac{9}{10}}(2+\eta)^{-\frac{19}{10}}|\nabla\Gamma^{l_2}u|^2|\Gamma^{l'_3}u|^2|\nabla\Gamma^{l'_4}u|^2 dx_2d\eta d\xi)^{\frac{1}{2}}\\ \nonumber
&+& (1+e_s)E^{\frac{1}{2}}_s (\sum\limits_{|l_2|+|l'_3|+|l'_4|\leq s-13}\iiint(2+\xi)^{-\frac{9}{10}}(2+\eta)^{-\frac{19}{10}}|\Gamma^{\tilde{l}_2}u|^2 |\Gamma^{\tilde{l}_3}u|^2|\nabla\Gamma^{\tilde{l}_4}u|^2dx_2d\eta d\xi)^{\frac{1}{2}} + (1+e_s)e^{\frac{1}{2}}_sE_s\\ \nonumber
&+& E^{\frac{1}{2}}_s (\sum\limits_{|l_1|+|l_2|+|l_3|+|l_4|\leq s-14}\iiint(2+\xi)^{-\frac{19}{10}}(2+\eta)^{\frac{1}{10}}  (\Gamma^{l_3}u_{\eta\eta}\Gamma^{l_4}(Q_0(u,u) + F''u^2_{\eta}))^2dx_2d\eta d\xi )^{\frac{1}{2}}\\ \nonumber
\end{eqnarray}
Using the similar procedure to the proof of Proposition 3, we have
\begin{eqnarray} \nonumber
 &&\iint u^2_{l\eta} dx_2d\eta \\ \nonumber
&\lesssim& \varepsilon +  (1+e_s)E^{\frac{1}{2}}_s (\sum\limits_{|l_2|+|l'_3|+|l'_4|\leq s-13}\iiint(2+\xi)^{-\frac{9}{10}}(2+\eta)^{-\frac{19}{10}}|\nabla\Gamma^{l_2}u|^2|\Gamma^{l'_3}u|^2|\nabla\Gamma^{l'_4}u|^2 dx_2d\eta d\xi)^{\frac{1}{2}}\\ \nonumber
&+& (1+e_s)E^{\frac{1}{2}}_s (\sum\limits_{|l_2|+|l'_3|+|l'_4|\leq s-13}\iiint(2+\xi)^{-\frac{9}{10}}(2+\eta)^{-\frac{19}{10}}|\Gamma^{\tilde{l}_2}u|^2 |\Gamma^{\tilde{l}_3}u|^2|\nabla\Gamma^{\tilde{l}_4}u|^2dx_2d\eta d\xi)^{\frac{1}{2}} \\ \nonumber
&+& (1+e_s)e^{\frac{1}{2}}_sE_s +  e^{\frac{1}{2}}_sE_s\\ \nonumber
&\lesssim& \varepsilon +  (1+e_s)E^{\frac{1}{2}}_s e_s(\iiint(2+\xi)^{-\frac{2}{5}}(2+\eta)^{-\frac{7}{5}}|\Gamma^{l'_3}u_{x_2}|^2 dx_2d\eta d\xi)^{\frac{1}{2}}\\ \nonumber
&+& (1+e_s)E^{\frac{1}{2}}_s (\iiint(2+\xi)^{-\frac{9}{10}}(2+\eta)^{-\frac{19}{10}}(2+\xi)^{2\delta}\tilde{e}_s^2 (2+\xi)(2+\eta) |\Gamma^{\tilde{l'}_3}u_{x_2}|^2(2+\xi)^{-\frac{1}{2}}(2+\eta)^{-\frac{1}{2}}e_k^2dx_2d\eta d\xi)^{\frac{1}{2}} \\ \nonumber
&+& (1+e_s)e^{\frac{1}{2}}_sE_s\\ \nonumber
&\lesssim& \varepsilon + (1+e_s)e_sE_s + (1+e_s)e_s \tilde{e}_s E^{\frac{1}{2}}_s (\iiint(2+\xi)^{2\delta-\frac{2}{5}}(2+\eta)^{-\frac{7}{5}} |\Gamma^{\tilde{l'}_3}u_{x_2}|^2dx_2d\eta d\xi)^{\frac{1}{2}}\\ \label{3.72}
&\lesssim & \varepsilon + (1+e_s)e_sE_s + (1+e_s)e_s \tilde{e}_s E_s
\end{eqnarray}
where $\tilde{l'}_3$ is equal to $\tilde{l}_2$ or $\tilde{l}_3$ and we take the positive constant $\delta < \frac{3}{20}$.

Similarly, multiplying $u_{l\xi}e^{-B(\xi)}$ into the above equation and integrating it for the variables $x_2$ and $\eta$, by the bound of $B(\xi)$ we have
\begin{eqnarray}\nonumber
&&\iint u^2_{lx_2}e^{-B(\xi)} dx_2d\eta \\ \nonumber
&\lesssim& \varepsilon +  \iiint u_{l\xi} \tilde{\Gamma}^l\{(1-H)[Q_0(u, Q_0(u, u))+ 4F' Q_0(u, u_{\eta}) + 4 F'' u_{\eta}^2]- 4 F'^2u_{\eta\eta}H\}dx_2d\eta \\ \nonumber
&\lesssim& \varepsilon +  (1+e_s)E^{\frac{1}{2}}_s (\sum\limits_{|l_2|+|l_3|+|l_4|\leq s-14}\iiint(2+\xi)^{\frac{1}{10}}(2+\eta)^{\frac{11}{10}}[Q_0(\Gamma^{l_2}u, Q_0(\Gamma^{l_3}u, \Gamma^{l_4}u))]^2dx_2d\eta d\xi )^{\frac{1}{2}}\\ \nonumber
&+& (1+e_s)E^{\frac{1}{2}}_s (\sum\limits_{|l_1|+|l_2|+|l_3|+|l_4|\leq s-14}\iiint(2+\xi)^{-\frac{19}{10}}(2+\eta)^{\frac{11}{10}}[ Q_0(\Gamma^{l_3}u, \Gamma^{l_4}u_{\eta}) ]^2dx_2d\eta d\xi )^{\frac{1}{2}}\\ \nonumber
&+& (1+e_s)E^{\frac{1}{2}}_s (\sum\limits_{|l_1|+|l_2|+|l_3|+|l_4|\leq s-14}\iiint(2+\xi)^{-\frac{29}{10}}(2+\eta)^{\frac{11}{10}}  (\Gamma^{l_3}u_{\eta}\Gamma^{l_4}u_{\eta})^2dx_2d\eta d\xi )^{\frac{1}{2}}\\ \nonumber
&+& e_sE^{\frac{1}{2}}_s (\sum\limits_{|l_1|+|l_2|+|l_3|+|l_4|\leq s-14}\iiint(2+\xi)^{-\frac{29}{10}}(2+\eta)^{\frac{11}{10}}(\Gamma^{l_3}u_{\eta\eta}\Gamma^{l_4}H)^2  dx_2d\eta d\xi )^{\frac{1}{2}}\\ \nonumber
&\lesssim& \varepsilon +  (1+e_s)E^{\frac{1}{2}}_s (\sum\limits_{|l_2|+|l_3|+|l_4|\leq s-14}\iiint(2+\xi)^{\frac{1}{10}}(2+\eta)^{\frac{11}{10}}[|\nabla\Gamma^{l_2}u|^2|\Gamma Q_0(\Gamma^{l_3}u, \Gamma^{l_4}u)|^2+ |\Gamma\Gamma^{l_2}u|^2| \nabla Q_0(\Gamma^{l_3}u, \Gamma^{l_4}u)|^2]dx_2d\eta d\xi )^{\frac{1}{2}}\\ \nonumber
&+& (1+e_s)e^{\frac{1}{2}}_sE^{\frac{1}{2}}_s (\sum\limits_{|l_3|+|l_4|\leq s-14}\iiint(2+\xi)^{-\frac{9}{10}}(2+\eta)^{\frac{11}{10}}\{ (2+\eta)^{-2}(\Gamma^{l_3}u)^2_{\xi} + (2+\xi)^{-\frac{1}{2}}(2+\eta)^{-\frac{1}{2}} [(\Gamma^{l_3}u)^2_{\eta} + (\Gamma^{l_3}u)^2_{x_2}]\}dx_2d\eta d\xi )^{\frac{1}{2}}\\ \nonumber
&+& (1+e_s)e^{\frac{1}{2}}_sE^{\frac{1}{2}}_s (\sum\limits_{|l_3|+|l_4|\leq s-14}\iiint(2+\xi)^{-\frac{19}{10}}(2+\eta)^{-\frac{9}{10}}  (\Gamma^{l_4}u_{\eta})^2dx_2d\eta d\xi )^{\frac{1}{2}}\\ \label{3.73}
&+& e_sE^{\frac{1}{2}}_s (\sum\limits_{|l_1|+|l_2|+|l_3|+|l_4|\leq s-14}\iiint(2+\xi)^{-\frac{29}{10}}(2+\eta)^{\frac{11}{10}}
(\Gamma^{l_3}u_{\eta\eta}\Gamma^{l_4}(Q_0(u,u) + F''u^2_{\eta}))^2  dx_2d\eta d\xi )^{\frac{1}{2}}.
\end{eqnarray}
Using the similar procedures to $\iint u^2_{l\eta} dx_2d\eta$, we can get
 \begin{eqnarray} \label{3.74}
\iint u^2_{lx_2} dx_2d\eta \lesssim \varepsilon + (1+e_s)e_sE_s + (1+e_s)e_s \tilde{e}_s E_s.
\end{eqnarray}

Therefore, Combining all the above estimates and by the bootstrap method, we can get the energy estimates
\begin{eqnarray} \label{3.75}
E_s \lesssim  \varepsilon, \ \ \ \ \ e_s \lesssim \varepsilon.
\end{eqnarray}
Then, we can get the stability of the traveling wave solution $F(x_1 + t)$ to the time-like extremal hypersurface in Minkowski space $\mathbb{R}^{1+(2+1)}$.

\begin{remark}
Here we will give the difference of the proof of stability result to the general traveling wave solutions $(a + b x_2)F(x_1 + t)$. There is one more term $Q_0(u, F u_{x_2})$ in (\ref{1.12}) than the terms in (\ref{3.2}). In the proof of stability result, the main step is to get the decay of the variables $\xi$ and $\eta$. We can get the following two decay estimates
\begin{eqnarray*}
&&|Q_0(u, F u_{x_2})|\lesssim (2+ \xi)^{-2}[|\Gamma u||\nabla u_{x_2}| + |\nabla u||\Gamma u_{x_2}|]\\
&&|Q_0(u, F u_{x_2})|\lesssim (2+ \xi)^{-1}(2+ \eta)^{-1}[|\Gamma u|(|u_{\xi x_2}|+|u_{x_2 x_2}| + (|u_{x_2}|+ |u_\xi|)|\Gamma u_{x_2}|]
\end{eqnarray*}
There is at least one good derivative in the right hand side. Using the similar procedures, we can get the main stability result for the traveling wave solutions with the general form. We omit the details.
\end{remark}

\section*{Acknowledgement}
This work was done when Jianli liu was visiting Key Laboratory of Mathematics for Nonlinear Sciences (Fudan University), Ministry of Education of China, P.R.China during 2018. He would like to thank the institute for their hospitality.

\section*{References}
\bibliographystyle{elsarticle-num}

\begin{thebibliography}{99}
\bibitem{Allen} P. Allen, L. Andersson and J. Isenberg, Timelike minimal submanifolds of general codimension in Minkowski space time, J. Hyperbolic Differ. Equ., 3 (2006), 691-700.


\bibitem{Alinhac} S. Alinhac, The null condition for quasilinear wave equations in two space dimensions
I, Invent. Math. 145, (2001), 597-618.

\bibitem{Almgren}
F.J.Jr. Almgren, Some interior regularity theorems for minimal surfaces
and an extension of Bernstein¡¯s theorem. Ann. of Math., 84 (1966), no. 2,
277-292.


\bibitem{Alias and Palmer}
L.J. $Al\acute{i}as$ and B. Palmer, On the Gaussian curvature of maximal surfaces and the Calabi-Bernstein theorem, Bull. London Math. Soc. 33 (2001), 454-458.



\bibitem{Barbashov} B. M. Barbashov, V. V. Nesterenko, and A. M. Chervyakov, General solutions of nonlinear equations in the geometric theory of the relativistic string,
    Comm. Math. Phys. 84 (1982), no. 4, 471-481.

\bibitem{Bernstein}
S. Bernstein, Sur un th¡äeor`eme de g¡äeom¡äetrie et son application aux
¡äequations aux d¡äeriv¡äees partielles du type elliptique, Comm. Soc. Math.
de Kharkov 2 (15) (1915-1917), 38-45. German translation: Uber ein geometrisches Theorem und seine Anwendung auf die partiellen Differentialgleichungen
vom elliptischen Typus, Math. Z. 26 (1) (1927), 551-558.

\bibitem{Bombieri}
 E. Bombieri, E. De Giorgi and E. Giusti, Minimal cones and the Bernstein
problem, Invent. Math. 7 (1969), 243-268.



\bibitem{Brendle} S. Brendle, Hypersurfaces in Minkowski space with vanishing mean curvature. Comm.
Pure Appl. Math. 55(2002), 1249-1279 .


\bibitem{Calabi}
E. Calabi, Examples of Bernstein problems for some nonlinear equations. 1970 Global Analysis
(Proc. Sympos. Pure Math., Vol. XV, Berkeley, Calif., (1968) pp. 223-230 Amer. Math. Soc.,
Providence, R.I.

\bibitem{Cheng and Yau}
S.Y. Cheng and S.T. Yau, Maximal space-like hypersurfaces in the Lorentz-Minkowski spaces,
Ann. of Math. 104 (1976), no. 2, 407-419.

\bibitem{Christodoulou} D. Christodoulou, The formation of black holes in general relativity. Monographs in
Mathematics, European Mathematical Society. (2009).



\bibitem{De Giorgi}
E. De Giorgi, Una estensione del teorema di Bernstein. Ann. Scuola Norm.
Sup. Pisa, 19 (1965), no. 3, 79-85.

\bibitem{Donninger Krieger Szeftel and Wong} R. Donninger, J.Krieger, J. Szeftel, W.W.Y Wong.: Codimension one stability of the
catenoid under the vanishing mean curvature flow in Minkowski space. Duke Math. J.
165, (2016), 723-791 .



\bibitem{Eichmair} M. Eichmair, The Plateau problem for marginally outer trapped surfaces. Journal of Differential Geometry. 83 (2009), no. 2, 551-584.


\bibitem{Estudillo and Romero 1}
F.J.M. Estudillo and A. Romero, On maximal surfaces in the n-dimensional Lorentz-Minkowski
space, Geom. Dedicata 38 (1991), 167-174.


\bibitem{Estudillo and Romero 2}
F.J.M. Estudillo and A. Romero, Generalized maximal surfaces in Lorentz-Minkowski space
L3, Math. Proc. Cambridge Philos. Soc. 111 (1992), 515-524.


\bibitem{Estudillo and Romero 3}
F.J.M. Estudillo and A. Romero, On the Gauss curvature of maximal surfaces in the 3-
dimensional Lorentz-Minkowski space, Comment. Math. Helv. 69 (1994), 1-4.


\bibitem{Fleming}
Wendell H. Fleming, On the oriented Plateau problem, Rendiconti del Circolo Matematico di Palermo. Serie II, 11 (1962), 69-90.


\bibitem{Hopf}
E. Hopf, On S. Bernstein¡¯s theorem on surfaces $z(x, y)$ of nonpositive
curvature, Proc. Amer. Math. Soc. 1, (1950), 80-85.



\bibitem{Klainerman} S. Klainerman, The null condition and global existence to nonlinear wave equations, Nonlinear
systems of partial differential equations in applied mathematics, Part 1 (Santa Fe, N.M., 1984),
Lectures in Appl. Math., vol. 23, Amer. Math. Soc., Providence, RI, 1986, pp. 293-326.


\bibitem{Klainerman and Rodnianski} S.Klainerman and I.Rodnianski, On the Formation of Trapped Surfaces, Acta Math. 208, (2012),
211-333.


\bibitem{Krieger and Lindblad} J.Krieger and H.Lindblad, On stability of the catenoid under vanishing mean curvature
flow on Minkowski space, Dyn. Partial Differ. Equ. 9, (2012), 89-119 .



\bibitem{Kobayashi}
O. Kobayashi, Maximal surfaces in the 3-dimensional Minkowski space $L^3$, Tokyo J. Math. 6
(1983), 297-309.

\bibitem{Kong Sun and Zhou} D.X. Kong, Q.Y. Sun and Y. Zhou, The equation for time-like extremal surfaces in Minkowski space $R^{2+n}$, J. Math. Phys. 47(1), (2006), 013503.


\bibitem{Lindblad}
H. Lindblad, A remark on global existence for small initial data of the minimal surface equation in Minkowskian space time, Proc. Amer. Math. Soc. 132(4), (2004), 1095-1102.


\bibitem{Liu and Zhou} J.L. Liu and Y. Zhou, Asymptotic behaviour of global classical solutions of
diagonalizable quasilinear hyperbolic systems, Math. Meth. Appl. Sci. 30 (2007), 479-500.

\bibitem{Li and zhou} T.T. Li and Y. Zhou, Nonlinear wave equations,  Spring-Verelag CmbH Germany and Shanghai Scientific and Technical Publishers, 2017.
\bibitem{Liu and Zhou 1} J.L. Liu and Y. Zhou,  Initial-boundary value problem for the equation of time-like extremal surfaces in Minkowski space,
    Journal of Mathematical Physics, 49 (2008), 043507.


\bibitem{Liu and Zhou 2} J.L. Liu and Y. Zhou, The initial¨Cboundary value problem on a strip for the equation of time-like extremal surfaces, Discrete Contin. Dyn. Syst. 23(1-2) (2009) 381-397.


\bibitem{Liu and Liu} C.M. Liu and J. L. Liu, Stability of traveling wave solutions to Cauchy problem of diagnolizable quasilinear hyperbolic systems, Discrete Contin. Dyn. Syst. 34 (2014),4735-4749.


\bibitem{Milnor} T. Milnor, Entire timelike minimal surfaces in $E^{3,1}$, Michigan Math. J. 37 (1990), no. 2, 163-177.


\bibitem{Miao Pei and Yu} S.Miao, L.Pei and P. Yu, On classical goblal solutions of nonlinear wave equations with
large data, To appear in International Mathematics Research Notices.


\bibitem{Moser}
J. Moser, On Harnack¡¯s theorem for elliptic differential equations. Comm.
Pure Appl. Math. 14 (1961), 577-591.



\bibitem{Ncnertey}
L.V. McNertey, One-parameter families of surfaces with constant curvature in Lorentz 3-space,
Ph.D. thesis, Brown University (USA), 1980.

\bibitem{Poitr} Chru$\acute{s}$ciel, Piotr T.; Galloway, Gregory J. and Pollack, Daniel. Mathematical general relativity: a sampler, Bull. Amer. Math. Soc., 47 (2010), 567-638.



\bibitem{Romero}
A. Romero, Simple proof of Calabi-Bernstein¡¯s theorem on maximal surfaces, Proc. Amer.
Math. Soc. 124 (1996), 1315-1317.



\bibitem{Sabitov}
I.Kh.Sabitov, Bernstein theorem, in Hazewinkel, Michiel, Encyclopedia of Mathematics, Springer, 2001.

\bibitem{Simons}
J. Simons, Minimal varieties in riemannian manifolds. Ann. of Math. 88 (1968), no. 2, 62-105.

\bibitem{Straume}
 E.Straume, Bernstein problem in differential geometry, in Hazewinkel, Michiel, Encyclopedia of Mathematics, Springer, 2011.

\bibitem {Wang and Yu 1} J.Wang, P.Yu, Long time solutions for wave maps with large data. J. Hyperbolic Differ
Equ. 10, (2013), 371-414.

\bibitem {Wang and Yu 2} J.Wang, P. Yu, A large data regime for nonlinear wave equations. J. Eur. Math. Soc.
(JEMS), 18, (2016), 575-622.

\bibitem{Yang} S.Yang, Global solutions of nonlinear wave equations with large data. Selecta Math.
(N.S.) 21, (2015), 1405-1427.

\bibitem{Wang and Wei} J.H. Wang and C.H. Wei, Global existence of smooth solution to relativistic membrane
equation with large data, arXiv:1708.03839v1 [math.AP] 13 Aug 2017.




%



%
%
%

%


%






%
%
%
%
%

\end{thebibliography}

\end{document}